\documentclass[a4paper]{article}
\usepackage{preamble}
\usepackage{longtable} 
\usepackage{booktabs}
\usepackage{authblk}
\graphicspath{ {./} }

\usetikzlibrary{decorations.markings}

\pgfmathdeclarefunction{atanh}{1}{\pgfmathparse{0.5*ln((1+#1)/(1-#1))}}

\title{Riemannian Gradient Descent for Gaussian Mixture Models with unknown diagonal covariances}

\bgroup

\author[1]{Romane Giard}
\author[1,2]{, Yohann De Castro}
\author[1]{, Roland Denis}
\author[1]{ and Clément Marteau}

{
\affil[1]{{Centrale Lyon, INSA Lyon, Université Lyon 1, Université Jean Monnet, CNRS, ICJ UMR5208, 69130 Ecully, France.
}}
\affil[2]{Institut Universitaire de France (IUF)}
}
\date{version as of \today}
\egroup

\begin{document}

\maketitle

\begin{abstract}
This paper investigates the numerical resolution of the Beurling-LASSO (BLASSO), a convex optimization framework that promotes sparsity in the space of measures. We consider its application to the estimation of Gaussian mixture models (GMMs) with an unknown number of components and unknown diagonal covariance matrices. Our approach combines the Conic Particle Gradient Descent (CPGD) principle with Riemannian gradient descent, to account for the underlying Fisher-Rao geometry of Gaussian distributions.
Our contributions are twofold. First, we provide theoretical guarantees for the convergence of our algorithm. In particular, we establish exponential local convergence under a non-degeneracy condition on the solution and relate this assumption to a separation condition on the underlying statistical target. Second, we address practical implementation aspects of CPGD and present numerical experiments illustrating its performance. On the test cases considered, these experiments suggest that CPGD is more robust to overspecification of the number of components than the EM algorithm. We also investigate the impact of component separation on recovery accuracy.
\end{abstract}
\medskip

\noindent Corresponding author: Romane Giard\\
\noindent E-mail address: \texttt{romane.giard@ec-lyon.fr}

\section{Introduction}
Our aim in this paper is to investigate the numerical and theoretical properties of a new algorithm designed to address a specific instance of the BLASSO. This first section recalls the motivation underlying this optimization problem and sheds light on the setting considered.

\paragraph{The Beurling-LASSO}
Let $\L$ be an Hilbert space with associated norm $\norm{\dotp}_\L$. We consider an observation $y\in\L$ that we seek to decompose or approximate as a nonnegative linear combination of $p\in \mathbb{N}^*$ parameterized signals:
\begin{equation}
\sum_{j=1}^p \omega_j f_{x_j} \,,
\label{eq:sparseparam}
\end{equation}
where for any $j\in \lbrace 1,\dots, p \rbrace$, the amplitude $\omega_j$ is nonnegative, and the location $x_j$ is a real vector belonging to a compact set $\X$, indexing $f_x \in \L$. 
Our goal is to recover a sparse representation, in the sense that only a small number of weights $\omega_j$ are nonzero. This question can be naturally reformulated as the optimization problem 
\begin{equation}
\tag{$\mathcal{P}'$}
\label{eq:problemP'}
\underset{ (\omega_j)_{j=1}^p \subset \R_+, (x_j)_{j=1}^p \subset \X}{\arg\min} \; \frac{1}{2} \norm{y-\sum_{j=1}^p \omega_j f_{x_j}}_\L^2 \,.
\end{equation}
However, we stress that \eqref{eq:problemP'} is highly nonconvex due to the joint optimization over amplitudes and locations. Moreover, nothing guarantees that the corresponding solution will be sparse.

To overcome these difficulties, the Beurling–LASSO (BLASSO), introduced by \citep{decastro_2012} and further studied in, \eg, \citep{blasso_duval_peyre, off-the-grid_cs}, lifts the problem to the space of measures via the embedding
\[(\omega_j)_{j=1}^p, (x_j)_{j=1}^p \mapsto \mu= \sum_{j=1}^p \omega_j \delta_{x_j} \,.\]
Defining the linear operator $\Psi: \mu \mapsto \int_\X f_x \, \d\mu(x)$, the BLASSO formulation then reads as
\begin{equation}
\tag{$\mathcal{P}_\kappa'$}
\label{eq:problemPkappa'}
\underset{\mu \in \M(\X)^+}{\arg\min} \frac{1}{2} \norm{y-\Psi\mu}_\L^2 + \kappa \norm{\mu}_{\rm TV} \,,
\end{equation}
where $\M(\X)^+$ denotes the space of nonnegative Radon measures on $\X$. The problem \eqref{eq:problemPkappa'} includes a penalty term expressed in the total variation norm and controlled by the regularization parameter $\kappa>0$. The total variation norm plays a role analogous to the $\ell^1$-norm in finite dimensions and promotes sparsity of the solution (see, \eg, \citep{candes_fernandez_granda_2014}).

The formulation \eqref{eq:problemPkappa'} is convex over the space of measures and does not require specifying the number of components $p$ in advance. As such, it has been the subject of several investigations, focusing either on the statistical performance of the solution $\mu^\star$ (\wrt to a given ground truth $\mu^0$) or on algorithmic aspects of the underlying optimization problem. Nevertheless, almost all of these contributions have been developed in the setting where the kernel $(x,x') \mapsto K(x,x')=\innerprod{\Psi\delta_x}{\Psi\delta_{x'}}_\L$ induced by the feature map $\Psi$ is translation-invariant (see Section \ref{section:optimization_problem} for further details). In this work, we extend these investigations to a specific kernel that is not translation-invariant. Our motivation stems from the classical Gaussian mixture estimation problem. This extension raises interesting new challenges, both from analytical and algorithmic perspectives.

\paragraph{Gaussian mixtures and non-translation-invariant kernels} The particular case of Gaussian mixture models (GMMs) provides an interesting example in which the kernel is not translation-invariant. This situation arises when the covariance structures are unknown. In the GMM setting, $y$ represents an approximation of an underlying target density built from i.i.d.\ measurements $X_1,\ldots,X_n \in \R^d$. The terms $f_x$ in \eqref{eq:sparseparam} then correspond to Gaussian densities. The parameter $x$ encodes both the mean and the diagonal covariance matrix.

Classically, GMM estimation is performed using the Expectation–Maximization (EM) algorithm \citep{mclachlan_em}, which maximizes the likelihood of the data. However, this optimization problem is inherently nonconvex and may suffer from local minima and sensitivity to initialization. In contrast, the BLASSO provides a convex alternative, offering a principled framework for sparse mixture recovery together with theoretical guarantees. The properties of the corresponding estimator have been investigated in \citep{article_giard_decastro_marteau}.

From a methodological perspective, previous implementations in the BLASSO literature (see, \eg, \citep{supermix, fastpart}) have focused on translation-invariant kernels. In the GMM framework, such a setting corresponds to the case where the covariance matrices associated with all components are equal (and known in advance). Imposing a known common covariance structure is quite restrictive and unrealistic from a practical standpoint. The framework developed in this contribution allows us to relax these constraints within the BLASSO methodology. In particular, we consider a specific operator $\Psi$ that allows us to handle Gaussian components with unknown (diagonal) covariance structures. In this setting, we present an implementable algorithm supported by a theoretical convergence analysis, together with competitive numerical results.

\paragraph{Conic Particle Gradient Descent}
To address the BLASSO problem \eqref{eq:problemPkappa'}, we develop and investigate, in Section \ref{section:optimization_procedure}, a new instance of the Conic Particle Gradient Descent (CPGD) algorithm introduced by \citep{cpgd_chizat}. 

The CPGD approach relies on a discretization of the space $\M(\X)^+$. Since this space cannot be directly parameterized, the optimization is carried out over discrete measures of the form $\sum_{j=1}^p \omega_j \delta_{x_j}$. The algorithm then updates the weights $\omega_j$ and positions $x_j$ through gradient-based iterations. To enforce the nonnegativity constraint on the weights, a conic retraction step is applied.
In our setting, particular attention is paid to the update of the parameters $x_j$. We exploit the underlying geometry induced by the Gaussian structure and consider updates based on the Fisher–Rao metric \citep{amari_information_book}, leading to Riemannian gradient descent and natural gradient descent schemes.

\paragraph{Related works}
Conic Particle Gradient Descent belongs to a broader lineage of overparameterized particle methods for optimization over the space of measures. \citep{chizat_bach_2018} analyze the Wasserstein gradient flow that arises as the mean-field limit of gradient descent on the weights and positions of an overparameterized model. Under a homogeneity assumption and a separation property of the initialization that can only be attained in the infinite-width limit, any limit point of the flow is a global minimizer. This yields consistency-type guarantees for the continuous-time flow, without providing a convergence rate for a finite step size. The conic version studied by \citep{cpgd_chizat} couples a multiplicative (mirror) update of the particle weights with an additive update of their positions; its mean-field limit is the gradient flow of the objective with respect to the Wasserstein-Fisher-Rao, or Hellinger-Kantorovich, metric. This cone geometry was introduced and developed independently by three groups \citep{savare_hellinger_kantorovich,liero_mielke_savare_inventiones,chizat_hk1,chizat_hk2,kondratyev_hk}. The Hellinger-Kantorovich distance is characterized through a lifting to the metric cone over the base space, and interpolates between pure transport (Wasserstein) and pure correlation (Hellinger/Fisher-Rao). \citep{gallouet_monsaingeon_jko} realize the corresponding gradient flows through a JKO (Jordan-Kinderlehrer-Otto) splitting scheme that alternates between a Wasserstein transport substep and a Fisher–Rao correlation substep.

The theory available for CPGD is, however, sharply divided between the flow and the discrete scheme. At the continuous-time level, \citep{cpgd_chizat} establishes asymptotic global convergence under a full-support initialization, but without a rate. For the iterates, a sharpness (Polyak-Lojasiewicz) inequality valid within a sublevel set yields only \emph{local} exponential convergence; global convergence of the discrete descent additionally requires an initialization that densely samples the parameter space, which forces the number of particles to grow exponentially with the dimension~$d$, together with a small position-to-weight step-size ratio~$\frac{\beta}{\alpha}$ that can make the guaranteed local rate slow. A faster $O\left(\frac{\ln k}{k}\right)$ rate is available only for the pure-mirror update ($\beta=0$, \ie frozen positions) initialized from a smooth full-support density, whereas the general descent rate is $O\left(\frac{\ln k}{\sqrt{k}}\right)$. Stochastic and birth-death variants address scalability and this initialization burden: \citep{fastpart} combine CPGD with stochastic gradient descent and random features, proving bounded total-variation norms along the trajectory and convergence to stationary points at rate $O\left(\frac{\ln k}{\sqrt{k}}\right)$; \citep{decastro2026fsp} add a spawn-and-prune (birth-death) process, in the spirit of the birth-death dynamics of \citep{rotskoff_birth_death}, a non-local reweighting term that keeps the mean-field flow mass-conserving, and thereby obtain global convergence without an exponentially large initialization.

A widely used grid-based competitor is the Sliding Frank-Wolfe algorithm \citep{sliding_frank_wolfe_denoyelle}, which interleaves conditional-gradient steps that add a spike at the global maximizer of the dual certificate with a non-convex ``sliding'' step jointly re-optimizing all amplitudes and positions. This algorithm terminates in finitely many iterations under uniqueness of the BLASSO solution together with a non-degeneracy hypothesis on its dual certificate. Both ingredients are costly in our setting: the per-iteration global maximization is typically performed by a grid search. The corresponding cost grows with the dimension of the parameter domain $\X \subset \R^d \times [u_{\min},u_{\max}]^d$, whose points $x=(t,u)$ encode both means and diagonal scales, while the sliding subproblem is non-convex. Conic particle gradient descent instead moves particles continuously and avoids any per-iteration grid search. Such moving-particle dynamics also deliberately steer clear of the lazy, or tangent-kernel, regime \citep{chizat_lazy}, in which a large output scaling keeps the parameters frozen near their initialization so that no features are learned. Relatedly, high-dimensional analyses of the effective dynamics of stochastic gradient descent characterize when fixed-step trajectories converge to suboptimal solutions from a random initialization \citep{benarous_sgd_highdim}.

Against this backdrop, the non-translation-invariant kernel associated with unknown diagonal covariances induces a Fisher-Rao location-scale geometry into the position updates, making the dynamics Riemannian. In this setting, we establish local exponential convergence with explicit constants for the joint weight-and-position dynamics. We further connect the underlying non-degeneracy condition to an interpretable separation condition on the statistical target through a dual-certificate and Fisher-metric analysis in the spirit of off-the-grid super-resolution \citep{blasso_duval_peyre, off-the-grid_cs}.

\paragraph{Outline of the paper and main results} Our contribution begins in Section \ref{section:optimization_problem}, where we provide details on the framework under consideration. We first introduce a generic algorithm that encompasses several possible strategies.

Section \ref{section:theoretical_analysis} provides theoretical guarantees for the convergence of the CPGD algorithm. Building on the works of \citep{cpgd_chizat,fastpart}, we first establish convergence of the weights $\omega_j$, provided that the initialization assigns sufficient mass near the BLASSO solution. We then establish exponential local convergence of CPGD under a non-degeneracy condition on the solution, relying on the convergence of the position parameters $x_j$. We further connect this non-degeneracy assumption—which depends on the (unknown) solution—to a more interpretable separation condition on the underlying statistical target. This connection relies on the analysis of dual certificates and their behavior, following ideas developed in \citep{blasso_duval_peyre, off-the-grid_cs, article_giard_decastro_marteau}.

In Section \ref{section:numerical_experiments}, we turn to practical aspects of the method. We make explicit the relationship between GMMs and the BLASSO problem. We then present numerical experiments illustrating the performance of the algorithm in this setting. We implement CPGD using natural gradient descent to account for the underlying geometry of the problem. Our experiments highlight the robustness of CPGD to misspecification of the number of components, in contrast to the EM algorithm. We also investigate the influence of component separation and sample size on recovery accuracy: while larger sample sizes improve performance, insufficient separation between components significantly degrades recovery, and increasing the sample size in this regime brings only limited gains.

The code to reproduce the numerical illustrations of this article can be found online at \citep{repo_zenodo_cpgd}. All proofs and technical details are provided in the appendix to this paper.

\section{Optimization problem}
\label{section:optimization_problem}

\subsection{Framework}\label{section:framework}

\paragraph{Functional framework}
We introduce the space $\L$ involved in the data fidelity term in \eqref{eq:problemPkappa'}. Let $\tau \in (0, u_{\min}]$ (playing the role of a smoothing parameter, see Section \ref{section:numerical_experiments}). We consider the Hilbert space $\L$ as the RKHS associated with the Gaussian kernel 
\[
\lambda: z \in \R^d \mapsto \frac{e^{-\frac{\norm{z}_2^2}{2\tau^2}}}{(2 \pi \tau^2)^{d/2}} \,.
\]
In particular, by the Fourier characterization of translation-invariant RKHS \citep[Chapter 4]{steinwart}:
\[
\L \coloneq \left\lbrace f: \R^d \rightarrow \R\; : \; f \in L^2(\R^d) \text { s.t.\ }\norm{f}_{\L}^2=\frac{1}{(2\pi)^d}\int_{\R^d} \frac{|\F{f}(\xi)|^2}{\F{\lambda}(\xi)} \d \xi<+\infty\right\rbrace\,,
\]
with dot product
\begin{equation}
\label{eq:dot_product_L}
 \forall f, g \in \L, \quad\innerprod{f}{g}_{\L}=\frac{1}{(2\pi)^d}\int_{\R^d} \frac{\overline{\F{f}(\xi)} \F{g}(\xi)}{\F{\lambda}(\xi)} \d \xi\,,
\end{equation}
where for the Fourier transform we adopt the convention $\F{f}(\xi)=\int f(z)e^{-i\innerprod{z}{\xi}}\, \d z$.

We denote by $x=(t_1,\ldots,t_d,u_1,\ldots,u_d)$ any element of $\R^d \times [u_{\min},+\infty)^d$. We write $t=(t_k)_{k=1}^d$ and $u=(u_k)_{k=1}^d$. In the specific Gaussian mixture setting, $t \in \R^d$ (resp.\ $u\in [u_{\min},+\infty)^d$) denotes the mean (resp.\ the vector of componentwise standard deviations, \ie the square root of the diagonal of the covariance matrix) of a given component.
We define the linear operator $\Psi:\M(\R^d \times [u_{\min},+\infty)^d)^+\mapsto \L$ by \[\Psi \mu: z=(z_1,\ldots,z_d) \in \R^d \mapsto \int_{\R^d \times [u_{\min},+\infty)^d} \prod_{k=1}^d \frac{(2u_k^2+\tau^2)^{1/4}}{(2\pi)^{1/4}(u_k^2+\tau^2)^{1/2}} e^{-\frac{(z_k-t_k)^2}{2(u_k^2+\tau^2)}} \,\d \mu(x)\,.\] This operator maps a measure $\mu$ to a function in $\L$, obtained as a superposition of Gaussian-shaped functions. Each point $x=(t,u)$ indexes a function $\Psi\delta_x$. Note that these functions are not probability densities, but rather normalized features adapted to the RKHS structure (\ie $\norm{\Psi\delta_x}_\L = 1$ for any $x\in \X$).
We will make explicit the relationship between this setting and the GMM estimation problem in Section \ref{section:statistical_motivation}. We refer also to \citep{article_giard_decastro_marteau} for more details and related properties. In particular, the kernel associated with $\Psi$ naturally arises as a scalar product between two Gaussian components, while the lower bound $u_{\min}$ on the covariance matrices appears to be a standard constraint to avoid singular models. 

\paragraph{Optimization problem} Let $\X$ be a compact of $\R^d \times [u_{\min},u_{\max}]^d$. We denote $\M(\X)^+$ the space of nonnegative Radon measures on $\X$.
Let $\kappa>0$ (regularization parameter). We aim to approach $\mu_\omega^\star$ a solution to the BLASSO problem \eqref{eq:problemPkappa'} that can be rewritten as
\begin{equation}\label{eq:pb_BLASSO_gaussian_kernel_norm}
\tag{$\mathcal{P}_\kappa$}
\underset{\mu \in \M(\X)^+}{\arg \min} \: J_W(\mu) \quad \text{where} \quad  J_W(\mu) \coloneq \frac{1}{2}\norm{\Psi\mu - y}_\L^2 + \kappa\norm{\mu}_{\rm TV}\,,
\end{equation}
with $y\in \L$ and $\kappa>0$.

This problem can be interpreted as an infinite-dimensional sparse recovery problem over the space of measures. The first summand of $J_W$ enforces fidelity to the data $y$, while the total variation norm $\norm{\mu}_{\rm TV}$ promotes sparsity, favoring measures that are concentrated on a small number of points.

The operator $\Psi$ plays the role of a forward model: solving \eqref{eq:pb_BLASSO_gaussian_kernel_norm} amounts to recovering a sparse representation of $y$ in terms of the parameterized family of functions $\Psi\delta_x$. The underlying statistical motivation for this problem is presented in Section \ref{section:statistical_motivation}.

This problem has a solution (not necessarily unique), see for instance \citep[Proposition B.1]{article_giard_decastro_marteau}. In general, solutions are not necessarily discrete.

\paragraph{Kernel and geometry}
Problem \eqref{eq:pb_BLASSO_gaussian_kernel_norm} is associated with the feature map $\Psi$.
We consider the corresponding (normalized) kernel: for $x,x' \in \R^d \times [u_{\min},+\infty)^d$,
\begin{equation}
   K_{\rm norm}(x,x') = \innerprod{\Psi\delta_x}{\Psi\delta_{x'}}_{\L} =\prod_{k=1}^d \frac{(2u_k^2+\tau^2)^{1/4} (2{u_k'}^2+\tau^2)^{1/4}}{(u_k^2+{u_k'}^2+\tau^2)^{1/2}}e^{-\frac{(t_k-t_k')^2}{2(u_k^2+{u_k'}^2+\tau^2)}}\,. \label{eq:kernel_K_norm}
\end{equation}

Interestingly, this kernel induces a specific geometry over $x\in \R^d \times [u_{\min},+\infty)^d$. We will take advantage of this underlying geometry when designing our algorithm (see Section \ref{section:optimization_procedure}). Formally, the space $\R^d \times [u_{\min},+\infty)^d$ can be endowed with the Fisher–Rao metric associated with $K_{\rm norm}$ \citep{article_giard_decastro_marteau}. For $x\in \R^d \times [u_{\min},+\infty)^d$, the corresponding metric tensor is defined by
\begin{equation}
\label{eq:def_fisher_rao_metric}
 \mathfrak{g}_x\coloneq\nabla_1\nabla_2K_{\rm norm}(x,x)=\diag \left(\frac{1}{2u_1^2+\tau^2},\ldots, \frac{1}{2u_d^2+\tau^2},\frac{2u_1^2}{(2u_1^2+\tau^2)^2} ,\ldots, \frac{2u_d^2}{(2u_d^2+\tau^2)^2} \right)\,.
\end{equation}
This metric naturally induces norm and scalar product, defined by
\[\innerprod{v}{v'}_x= v^T \mathfrak{g}_x v' \quad \text{and} \quad \norm{v}_x = \sqrt{v^T \mathfrak{g}_x v} \qquad \forall \: v,v' \in \R^{2d}\,, \;x \in \R^d \times [u_{\min}, +\infty)^d \,.\]

We denote by $\gamma_{x\to x'}$ the Fisher-Rao geodesic parameterized between $0$ and $1$ connecting $x=\gamma_{x\to x'}(0)$ to $x'=\gamma_{x\to x'}(1)$. In particular, it minimizes $\int_0^1 \norm{\dot \gamma(y)}_{\gamma(y)}^2 \, \d y$ among smooth functions from $[0,1]$ to $\R^d \times [u_{\min},+\infty)^d$ such that $\gamma(0)=x, \gamma(1)=x'$. 
The corresponding distance between any $x,x'\in \R^d \times [u_{\min},+\infty)^d$ can then be defined as  
\[
\mathfrak{d}_\mathfrak{g}(x,x')=\norm{\dot\gamma_{x\to x'}(y)}_{\gamma_{x\to x'}(y)} \quad \text{for any} \quad  y\in[0,1].
\]
Finally, we can define the associated Riemannian gradient, which will be of particular interest when designing our algorithm.

\begin{definition}[Riemannian gradient]\label{def:riemannian_gradient}
Let $\psi \in \C^2(\R^d \times [u_{\min}, +\infty)^d)$. Let $x \in \R^d \times [u_{\min}, +\infty)^d$.
We define \[\nabla^{\mathfrak{g}} \psi(x)=\mathfrak{g}_{x}^{-1} \nabla \psi(x)\,.\]
\end{definition}
Additional details on the metric $\mathfrak{g}_x$ can be found in \citep[Appendix H]{article_giard_decastro_marteau}.

\subsection{Optimization procedure} \label{section:optimization_procedure}
We have now all the ingredients to design our algorithm based on the CPGD principle. In this context, the Fréchet derivative of the target objective function plays an important role. We first recall its main properties and discuss how we can take advantage of this quantity to address our problem. 

\paragraph{Fréchet derivative} Let $x\in \R^d \times [u_{\min},+\infty)^d$ and $\mu \in \M(\R^d \times [u_{\min},+\infty)^d)^+$. 
The Fréchet derivative $J'_\mu$ of $J_W$  verifies 
\begin{equation}\label{eq:dev_J_mu_plus_sigma_minus_J_mu}
    J_W(\mu + \sigma) - J_W(\mu) = \int_{\R^d \times [u_{\min},+\infty)^d} J_{\mu}' \, \d \sigma + \frac{1}{2} \norm{\Psi \sigma}_\L^2 \,.
\end{equation}
for any $\sigma \in \M(\R^d \times [u_{\min},+\infty)^d)$ such that $\mu + \sigma \in \M(\R^d \times [u_{\min},+\infty)^d)^+$. In particular, there exists an explicit expression for this derivative, namely
\begin{equation}\label{eq:def_Jp}
    J_{\mu}'(x)=-\innerprod{\Psi \delta_x}{y}_\L+\innerprod{\Psi\delta_x}{\Psi \mu}_\L+\kappa \,.
\end{equation}

The Fréchet derivative is strongly connected to our optimization problem. Indeed, as proved in \citep[Proposition 3.1]{cpgd_chizat}, the KKT conditions associated with the problem \eqref{eq:pb_BLASSO_gaussian_kernel_norm} impose that for $\mu_\omega^\star$ a solution to \eqref{eq:pb_BLASSO_gaussian_kernel_norm}, $J_{\mu_\omega^\star}' \geq 0$ on $\X$ and $J_{\mu_\omega^\star}'=0$ on $\supp(\mu_\omega^\star)$.

\paragraph{Discretization} To numerically solve \eqref{eq:pb_BLASSO_gaussian_kernel_norm}, we must discretize the space $\M(\X)^+$, as it cannot be directly parameterized. We therefore restrict our attention to discrete measures supported on $p$ particles. 

The gradients of $J_W$ with respect to the weights and the positions of these particles are linked to $J'$.
This relationship is formalized in \citep[Section 2.1]{cpgd_chizat}. In particular, for any discrete measure $\mu= \sum_{j=1}^p \omega_j \delta_{x_j}$, we can check that 
\[
\partial_{\omega_j} J_W(\mu) = J_{\mu}'(x_j) \quad \text{and} \quad \nabla_{x_j} J_W(\mu) = \omega_j \nabla_{x_j} J_{\mu}'(x_j) \quad \forall j\in \lbrace 1,\dots, p \rbrace\,.
\]
Starting with an initialization $\mu_\omega^1=\sum_{j=1}^p \omega_j^k \delta_{x_j^k}$ with $p \in \mathbb{N}^*$ particles, the general CPGD principle consists in designing iteratively a sequence of measures $(\mu_\omega^k)$ where at each step $k>1$, weights and positions of the current measure  $\mu_\omega^k=\sum_{j=1}^p \omega_j^k \delta_{x_j^k}$ are (roughly speaking) updated according to a gradient descent. Since our target is a positive measure, this gradient descent includes a conic retraction to constrain the nonnegativity of the weights. In the same time, the updates on the positions will take advantage of the underlying geometry of the considered problem. 

\paragraph{Weight update}
Following \citep{cpgd_chizat}, we use a conic retraction for the weight updates, with step size $\alpha\geq 0$. In particular, given $\mu_\omega^k$ for $k\geq 1$, we write 
\begin{equation}
\label{eq:update_poids}
\omega_j^{k+1}=\omega_j^ke^{-\alpha J_{\mu_\omega^k}'(x_j^k)} \quad \forall j\in \lbrace 1,\dots, p \rbrace\,.
\end{equation}
Contrary to a direct gradient descent, this update ensures that the weights remain positive throughout the iterations. Note that $\omega_j^{k+1}$ can be written as the solution of
\[
\underset{\omega \in \R^+}{\arg \min} \left[\: J_{\mu_\omega^k}'(x_j^k) (\omega - \omega_j^k) + \frac{1}{\alpha} \left(\omega_j^k - \omega + \omega \ln\left(\frac{\omega}{\omega_j^k}\right) \right)\right]\,,
\]
where the second term is the Bregman divergence of the entropy, leading to an entropic mirror descent step for the weights. We stress that the update \eqref{eq:update_poids} is in accordance with the KKT condition. Particles $x_j$ associated with a negative Fréchet derivative are promoted (\ie $\omega_j^{k+1}>\omega_j^k$) to secure mass guidance in regions where we observe a default in the KKT condition, while particles for which $J_{\mu_\omega^k}'(x_j^k)>0$ are scaled down.

\paragraph{Position update} The underlying Riemannian geometry associated with the statistical problem (see Section \ref{section:framework} above) gives different strategies for the position update. Formally, we use a custom operator $\mathcal{T}$ and a step size $\beta \geq 0$: $\mathcal{T}_{\beta}(x_j^k,J_{\mu_\omega^k}')$ denotes an update on the location-scale parameter of the jth particle, with step size $\beta$, depending on $J_{\mu_\omega^k}'$. 
We can think of three different approaches for this update:
\begin{align}
    \mathcal{T}_{\beta}(x_j^k,J_{\mu_\omega^k}')&= x_j^k - \beta \nabla_x J_{\mu_\omega^k}'(x_j^k) \,, \tag{EGD} \label{eq:EGD} \\
    \mathcal{T}_{\beta}(x_j^k,J_{\mu_\omega^k}')&=x_j^k - \beta\nabla^{\mathfrak{g}} J_{\mu_\omega^k}'(x_j^k)\,, \tag{NGD}\label{eq:NGD} \\
    \mathcal{T}_{\beta}(x_j^k,J_{\mu_\omega^k}') &= \exp_{x_j^k}(- \beta\nabla^{\mathfrak{g}} J_{\mu_\omega^k}'(x_j^k))\,. \tag{RGD}\label{eq:RGD}
\end{align}
These updates (illustrated in Figure \ref{fig:EGD_NGD_RGD}) correspond respectively to Euclidean gradient descent (EGD), natural gradient descent (NGD) and Riemannian gradient descent (RGD). In the expression \eqref{eq:RGD} above, $\exp_{\dotp}$ denotes the exponential map, \ie 
$\exp_{x_0}(v) = \gamma(1)$ where $\gamma$ is a Fisher-Rao geodesic with starting point $\gamma(0)=x_0$ and initial direction $\dot \gamma(0)=v$. In dimension $d=1$, the Fisher-Rao geodesics are either straight lines parallel to $\{t=0\}$ or semicircles centered on $\{u=0\}$ \citep[Appendix H.2]{article_giard_decastro_marteau}.

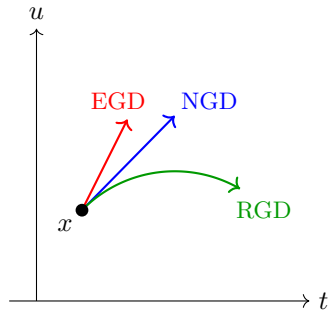
\begin{figure}[ht]
\centering
\begin{tikzpicture}[scale=1.2]

\draw[->] (-0.3,0) -- (3,0) node[right] {$t$};
\draw[->] (0,0) -- (0,3) node[above] {$u$};

\def\t{0.5}
\def\u{1}

\def\tau{0.2}
\pgfmathsetmacro{\TwoUSquare}{2*\u*\u}
\pgfmathsetmacro{\TauSquare}{\tau*\tau}
\pgfmathsetmacro{\TwoUSquarePlusTauSquare}{\TwoUSquare+\TauSquare}
\pgfmathsetmacro{\dt}{0.5*\TwoUSquarePlusTauSquare}
\pgfmathsetmacro{\du}{1*\TwoUSquarePlusTauSquare*\TwoUSquarePlusTauSquare/(2*\TwoUSquare)}
\draw[->, blue, thick] (\t,\u) -- ++(\dt,\du);
\node[blue] at (1.9,2.2) {\small NGD};

\pgfmathsetmacro{\lval}{sqrt(\dt*\dt/\TwoUSquarePlusTauSquare + \TwoUSquare*\du*\du/(\TwoUSquarePlusTauSquare*\TwoUSquarePlusTauSquare))}
\pgfmathsetmacro{\cOne}{\dt/(\TwoUSquarePlusTauSquare*\lval)}
\pgfmathsetmacro{\cTwo}{atanh(-\du*sqrt(2)*\u/(\TwoUSquarePlusTauSquare*\lval))}
\pgfmathsetmacro{\cThree}{\t + (\u*\du)/\dt}

\draw[green!60!black,thick,postaction={decorate,decoration={markings,mark=at position 1 with {\arrow{>}}}},domain=0:\lval,samples=100]
plot ({\cThree + (sqrt(2)*tanh(\cTwo + sqrt(2)*\x))/(2*\cOne)},{sqrt(-(\tau*\tau)/2+ 1/(2*\cOne^2*cosh(\cTwo + sqrt(2)*\x)^2))});

\node[green!60!black] at (2.5,1) {\small RGD};

\draw[->, red, thick] (\t,\u) -- ++(0.5,1);
\node[red] at (0.9,2.2) {\small EGD};

\fill (\t,\u) circle (2pt) node[below left] {$x$};
\end{tikzpicture}

\caption{Illustration of the three update strategies in dimension $d=1$. The parameters $t$ (mean) and $u$ (standard deviation) define a two-dimensional location–scale space $(t,u)$. Each arrow starts at $x$ and points to the corresponding updated position.
}
\label{fig:EGD_NGD_RGD}
\end{figure}

The updates \eqref{eq:NGD} and \eqref{eq:RGD} account for the intrinsic geometry of the space of Gaussian distributions (means and diagonal covariances). We refer to \citep[Section 4.1]{nielsen_intro_information_geometry} and \citep{amari_natural_gd}. Furthermore, the RGD update is invariant under any invertible smooth reparameterization of the parameter space.
These two schemes are closely related: natural gradient descent can be viewed as a first-order approximation of Riemannian gradient descent when the step size $\beta$ is small.

We emphasize that the existence of these different strategies is specific to our setting. In particular, when considering GMM with \textit{known} and \textit{identical} covariance structures, the Fisher-Rao metric reduces to the Euclidean one, and \eqref{eq:EGD}, \eqref{eq:NGD}, \eqref{eq:RGD} lead in this case to similar descent schemes (up to the choice of the learning rate $\beta$). This is a strong difference with previous implementations of the BLASSO problem (see, \eg, \citep{fastpart,decastro2026fsp}), all of them being restricted to translation-invariant kernels. Considering a non-translation-invariant setting generates new questions that are addressed in the following.

\begin{remark}[Ensuring well-defined updates]
We stress that these updates are well-defined under certain conditions or modifications.
We can for instance think of some projected version of \eqref{eq:EGD} or \eqref{eq:NGD}, to keep the location-scale parameters in $\R^d \times [u_{\min},+\infty)^d$, or in $\X$. We could use the Euclidean projection, or the projection \wrt to the Fisher-Rao distance.
Thinking about a projection step is particularly important for \eqref{eq:RGD}, because the exponential map is not defined everywhere (\ie $\exp_{x_0}(v)$ is defined for all $v$ such that $\norm{v}_{x_0}\leq c_{x_0}$, with $c_{x_0}>0$ depending on $x_0$, but not for all $v \in \R^{2d}$). Moreover, alternative instances can be considered using alternative optimization tools or strategies. A solution involving freezing the particles outside a certain set will be presented in Section \ref{section:exponential_convergence}, while numerical experiments displayed in Section \ref{section:numerical_experiments} use AdaGrad strategies for the calibration of the learning rates $\alpha$ and $\beta$.
\end{remark}

\paragraph{A general algorithm} Gathering the weights and position update strategies discussed above, Algorithm \ref{algo:CPGD_general_form} proposes a generic instance of the CPGD algorithm adapted to our framework. 

\begin{algorithm}[H]
\caption{CPGD}
\label{algo:CPGD_general_form}
\begin{algorithmic}[1]
\Require Step sizes $\alpha,\beta \geq 0$, number of iterates $K$, initialization $\mu_\omega^1 = \sum_{j=1}^p \omega_j^1 \delta_{x_j^1}$
\For{$k = 1,\ldots K-1$}
    \For{$j = 1,\ldots,p$}
      \State 
    $\omega_j^{k+1} \gets \omega_j^k e^{-\alpha J'_{\mu_\omega^k}(x_j^k)}$ \Comment{Weight update}
        
    \State
       $x_j^{k+1} \gets \mathcal{T}_{\beta}(x_j^k,J_{\mu_\omega^k}')$ \Comment{Location-scale update}
    \EndFor
    \State $\mu_\omega^{k+1} \gets \sum_{j=1}^p \omega_j^{k+1} \delta_{x_j^{k+1}}$
\EndFor
\end{algorithmic}
\end{algorithm}

\paragraph{Geometry and gradient descent}
The relative updates for the weights and positions are adapted from \citep{cpgd_chizat}. They rely on the cone metric, whose metric tensor at point $(\omega, x) \in \R_+^* \times \R^{2d}$, and for step sizes $\alpha,\beta>0$ is given by \[\diag\left(\frac{\alpha^{-1}}{\omega}, \beta^{-1} \omega, \ldots, \beta^{-1} \omega\right) \in \R^{(1+2d)\times (1+2d)} \,.\]
This choice induces multiplicative updates in $\omega$, while the updates in $x$ are independent of $\omega$. In addition, a mirror retraction is applied to $\omega$ in order to ensure its positivity. The retraction \eqref{eq:EGD} together with the weight update \eqref{eq:update_poids} is no more than a natural gradient descent for the cone metric and using a logarithmic reparameterization ($\tilde \omega = \ln \omega$).

On the other hand, using \eqref{eq:NGD} corresponds to the more intricate metric 
\[\diag\left(\frac{\alpha^{-1}}{\omega}, \beta^{-1} \omega \frac{1}{2u^2+\tau^2}, \ldots, \beta^{-1} \omega \frac{1}{2u^2+\tau^2},\beta^{-1} \omega \frac{2u^2}{(2u^2+\tau^2)^2}, \ldots, \beta^{-1} \omega \frac{2u^2}{(2u^2+\tau^2)^2} \right) \,,\]
which combines both the cone and Fisher-Rao metrics.

For more details on the cone metric and related properties, we refer to \citep{savare_hellinger_kantorovich,liero_mielke_savare_inventiones, chizat_hk1, chizat_hk2, kondratyev_hk}.

\section{Theoretical analysis of the algorithm}\label{section:theoretical_analysis}

This section provides a theoretical analysis of the CPGD principle in our setting. We investigate the convergence properties of Algorithm \ref{algo:CPGD_general_form} and one of its variant under generic conditions. 
\subsection{Convergence of the objective values via the weights dynamic}

We first establish a sublinear $O\left(\frac{\ln K}{\sqrt K}\right)$ rate on the objective values of the Ces\`aro average, driven by the weight dynamics only. More precisely, if $\mu_\omega^1$ is constructed on an grid dense enough in $\X$, then we can expect to recover the target measure. We revisit here \citep[Theorem 1.2]{fastpart} and part of the proof of \citep[Theorem 4.1]{cpgd_chizat} in our specific setting. At this step, we do not select a specific position update among \eqref{eq:EGD}, \eqref{eq:NGD}, \eqref{eq:RGD}, but work instead with a generic scheme under general conditions (in particular, freezing positions would work). The proof of the following proposition is displayed in Appendix \ref{section:proof_prop_weight_convergence}.

\begin{proposition} \label{prop:weight_convergence}
\underline{Choice of initialization:} Assume that $\mu_\omega^1 = \sum_{j=1}^p \omega_j^1 \delta_{x_j^1}$ with $\omega_j^1>0$ and $x_j^1\in\X$. 
Moreover, we suppose that there exists disjoint regions $\{\X_j\}_{j=1}^p \subset \X$ such that for all $j=1,\ldots,p$, $x_j^1 \in \X_j$ and $\sup_{\substack{x \in \X_j \,, \\ j=1,\ldots,p}} \mathfrak{d}_{\mathfrak{g}}(x_j^1,x) \leq r_{\mu_\omega^1,\X}$ for a radius $r_{\mu_\omega^1,\X}>0$. We write $\X_0\coloneq \X \setminus \left(\cup_{j=1}^p \X_j \right)$.
\\\underline{Choice of location-scale update:} We choose $\mathcal{T}$ and $\beta$ such that for $\mu_\omega^k = \sum_{j=1}^p \omega_j^k \delta_{x_j^k} \in \M(\X)^+$, $x_j^{k+1}=\mathcal{T}_{\beta}(x_j^k,J_{\mu_\omega^k}')\in \X$ and $\mathfrak{d}_\mathfrak{g}(x_j^k,x_j^{k+1}) \leq C_{\mathcal{T}} \beta$. 
\\\underline{Control of the Cesàro average:}
Let $K \geq 2$. Let $\kappa \leq 1$.
If $\alpha  = \frac{1}{\sqrt{K}}$ and $0 \leq \beta \leq \frac{1}{K^{3/2}C_{\mathcal{T}}}$, then
\begin{subequations}\label{eq:weight_convergence}
\begin{equation}
    J_W\left(\frac{1}{K}\sum_{k=1}^K \mu_\omega^k\right) - J_W(\mu_\omega^\star)
    \lesssim_{\norm{\mu_\omega^\star}_{\rm TV}, \norm{y}_\L, \X, \norm{\mu_\omega^1}_{\rm TV}} \frac{H_{\mu_\omega^1,\mu_\omega^\star} + 1}{\sqrt{K}} + r_{\mu_\omega^1,\X} + \mu_\omega^\star(\X_0)
\label{eq:borne1}
\end{equation}
where
\begin{equation}
   H_{\mu_\omega^1,\mu_\omega^\star} = \sum_{\substack{j\in \{1,\ldots, p\} \,, \\ \mu_\omega^\star(\X_j)>0}}\left( -\ln\left(\frac{\omega_j^1}{\mu_\omega^\star(\X_j)} \right)\mu_\omega^\star(\X_j) - \mu_\omega^\star(\X_j) \right) + \norm{\mu_\omega^1}_{\rm TV} \,. \label{eq:def_H_mu1_mu_star}
\end{equation}
\end{subequations}
\end{proposition}

Proposition \ref{prop:weight_convergence} describes the convergence of the algorithm. In particular, it provides a control on the performances of the Cesaro average $\frac{1}{K}\sum_{k=1}^K \mu_\omega^k$ in terms of the quantities displayed in the r.h.s.\ of \eqref{eq:borne1}. Inside this bound, the term $H_{\mu_\omega^1,\mu_\omega^\star}$ measures the divergence between the target $\mu_\omega^\star$ and the initialization. The smaller this divergence, the faster the convergence. This upper bound is completed by the mass of $\mu_\omega^\star$ outside the search space $\X_0$ and the radius $r_{\mu_\omega^1,\X}$. In particular, the grid should be thin enough to guarantee a good recovery precision. 

Note that Proposition \ref{prop:weight_convergence} only establishes convergence of the objective values, in the Cesàro sense. In particular, it does not guarantee that $\mu_\omega^k$ converges to $\mu_\omega^\star$.
When the descent property holds, \ie when $k \mapsto J_W(\mu_\omega^k)$ is decreasing (see Lemma \ref{lemma:descent_property} for sufficient conditions), the bound in Proposition \ref{prop:weight_convergence} applies directly to $J_W(\mu_\omega^k) - J_W(\mu_\omega^\star)$.

\begin{remark}[Choice of $\beta$ and limits of Proposition \ref{prop:weight_convergence}]
In Proposition \ref{prop:weight_convergence}, the step size $\beta$ can be taken equal to zero without affecting the bound on $J_W\left(\frac{1}{K}\sum_{k=1}^K \mu_\omega^k\right) - J_W(\mu_\omega^\star)$. The choice of location–scale updates is likewise largely immaterial, as it only needs to satisfy a very mild condition. In particular, one may simply use the trivial update $\mathcal{T}_\beta(x_j^k,J_{\mu_\omega^k}')=x_j^k$. 
This reflects the fact that this convergence analysis depends solely on the behavior of the weights, and not on the evolution of the location–scale parameters.
Consequently, the initialization must place sufficient mass near the support of the solution $\mu_\omega^\star$. More precisely, it requires that there exists $\X_0 \subset \X$ such that $\mu_\omega^\star(\X_0)$ is small, and for every $x\in \supp(\mu_\omega^\star) \setminus \X_0$, $\min_{j=1,\ldots,p} \mathfrak{d}_\mathfrak{g}(x,x_j^1)$ is small.
\end{remark}

\begin{remark}[Construction of the initialization] \label{remark:construct_init_grid}
The initialization $\mu_\omega^1$ can be specifically designed to control the terms $H_{\mu_\omega^1,\mu_\omega^\star}$, $r_{\mu_\omega^1,\X}$ and $\mu_\omega^\star(\X_0)$ appearing in Proposition \ref{prop:weight_convergence}.
Setting $\mu_\omega^1=\sum_{j=1}^p \frac{1}{p} \delta_{x_j^1}$ where the points $x_j^1$ form a grid and the regions $\X_j$ correspond to the associated cells, then we have $\mu_\omega^\star(\X_0)=0$ and we can expect $r_{\mu_\omega^1,\X}$ to be of order $p^{-1/(2d)}$ (omitting the dependence on the total volume of $\X$). Moreover, recalling \eqref{eq:def_H_mu1_mu_star}, with this choice of $\mu_\omega^1$,
\begin{align*}
H_{\mu_\omega^1,\mu_\omega^\star}
&\leq \max\left\{\norm{\mu_\omega^\star}_{\rm TV} \ln\left(\frac{\norm{\mu_\omega^\star}_{\rm TV}}{\min_{j} \omega_j^1}\right) \,, \, 0\right\} - \norm{\mu_\omega^\star}_{\rm TV} +\mu_\omega^\star(\X_0) + \norm{\mu_\omega^1}_{\rm TV} \,, \\
&\leq \norm{\mu_\omega^\star}_{\rm TV} \left( \max\left\{\ln(p \norm{\mu_\omega^\star}_{\rm TV})\,, \, 0\right\}-1\right)+ 1 \,.
\end{align*}
In particular, choosing $p=\ceil{K^d}$, we get $H_{\mu_\omega^1,\mu_\omega^\star} \lesssim_{\norm{\mu_\omega^\star}_{\rm TV}} d\ln(K)$ and
\[
    J_W\left(\frac{1}{K}\sum_{k=1}^K \mu_\omega^k\right) - J_W(\mu_\omega^\star)
    \lesssim_{\norm{\mu_\omega^\star}_{\rm TV}, d, \norm{y}_\L, \X} \frac{\ln(K)}{\sqrt{K}} \,.
\]
Note that with our choice of initialization, $\norm{\mu_\omega^1}_{\rm TV}=1$ does not depend on $p$.
\end{remark}

\begin{remark}[Need to work on a compact set, positivity of the kernel]
Working on a compact set ensures the existence of a solution to the BLASSO. Moreover, in the absence of prior information on $\mu_\omega^\star$, restricting the problem to a compact set is necessary in order to construct the initialization as described in Remark \ref{remark:construct_init_grid}.

In addition, we used in the proof the explicit form of $K_{\rm norm}$ and in particular the fact that $\inf_{x,x' \in \X} K_{\rm norm}\eqcolon C_{\X}>0$ since $\X$ is compact. This control allows to establish a bound on $\norm{\mu_\omega^k}_{\rm TV}$. An alternative approach, which does not rely on the positivity of $C_{\X}$, can still yield a bound on $\norm{\mu_\omega^k}_{\rm TV}$, but this bound then depends on $p$ (see Remark \ref{remark:control_tv_iterates_wo_C_X} in appendix). In particular, without prior information on $\mu_\omega^\star$, one cannot expect convergence as $p$ grows, since the bound on $J_W\left(\frac{1}{K}\sum_{k=1}^K \mu_\omega^k\right) - J_W(\mu_\omega^\star)$ then involves the term $p r_{\mu_\omega^1,\X}$, which may scale as $p^{1-\frac{1}{(2d)}}$ when constructing $\mu_\omega^1$ on a grid.
\end{remark}

\subsection{Exponential local convergence} \label{section:exponential_convergence}

The investigation presented in this section relies on \citep[Corollary 3.5]{cpgd_chizat}. We leverage here not only the dynamic of the weights but also that of the location–scale parameters. To this end, we employ the Riemannian gradient descent updates defined in \eqref{eq:RGD} together with the weights update \eqref{eq:update_poids}. To avoid difficulties related to projection, we adopt a simple strategy: freezing particles outside a region of interest.
More precisely, we assume that the compact set $\X$ satisfies $\X \subset \R^d \times [u_{\min},u_{\max}]^d$, for some $\tau\leq u_{\min}\leq u_{\max}$. 
We also consider a known compact set $\tilde \X \subset \mathring \X$. Particles lying outside $\tilde \X$ are frozen during the optimization, which ensures that the iterates remain in $\X$ (see Lemma \ref{lemma:descent_property}). This strategy is formalized in Algorithm \ref{algo:CPGD_Frozen_Riemannian}.

\begin{algorithm}[H]
\caption{Conic Particles Riemannian Descent with freezing process}
\label{algo:CPGD_Frozen_Riemannian}
\begin{algorithmic}[1]
\Require Step sizes $\alpha,\beta \geq 0$, number of iterates $K$, initialization $\mu_\omega^1 = \sum_{j=1}^p \omega_j^1 \delta_{x_j^1}$
\For{$k = 1,\ldots K-1$}
    \For{$j = 1,\ldots,p$}
      \State 
    $\omega_j^{k+1} \gets \omega_j^k e^{-\alpha J'_{\mu_\omega^k}(x_j^k)}$ \Comment{Weight update}
     
    \If{$x_j^k \in \X \setminus \tilde \X$}   
    \State 
       $x_j^{k+1} \gets x_j^k$
    \Else
    \State
       $x_j^{k+1} \gets \exp_{x_j^k}(-\beta \nabla^\mathfrak{g}J_{\mu_\omega^k}'(x_j^k))$ \Comment{Location-scale update}
    \EndIf
    \EndFor
    \State $\mu_\omega^{k+1} \gets \sum_{j=1}^p \omega_j^{k+1} \delta_{x_j^{k+1}}$
\EndFor
\end{algorithmic}
\end{algorithm}

The update \eqref{eq:RGD} is well-defined, and $x_j^{k+1}$ remains in $\X$ provided that $\beta$ is sufficiently small.
Note that the underlying manifold is not geodesically complete, so $\exp_x(v)$ is not defined for all $v \in \R^d$. In practice, this can be addressed by restricting to vectors $v$ with bounded norm $\norm{v}_x$. This requirement, together with the need to ensure that $x_j^k$ remains in $\X$, motivates the assumption below that $\beta$ is sufficiently small.

\paragraph{Descent property}  The following lemma (whose proof is postponed to Appendix \ref{section:proof_descent_property}) provides a first description of the performances of Algorithm \ref{algo:CPGD_Frozen_Riemannian}. In particular, it establishes that the objective function $J_W$ is non-increasing along the trajectory $(\mu_\omega^k)_{k\geq 1}$.

\begin{lemma}\label{lemma:descent_property}
    Assume that $\kappa \leq 1$. For $0 \leq \alpha,\beta \leq 1$ small enough (depending on $\X, \tilde \X, \norm{y}_\L, \norm{\mu_\omega^1}_{\rm TV}, \tau$), the position updates of Algorithm \ref{algo:CPGD_Frozen_Riemannian} are well-defined and for all $k\geq 1$, \[J_W(\mu_\omega^{k+1})-J_W(\mu_\omega^k)\leq -\frac{1}{2} \int_{\tilde \X} \left(\alpha |J_{\mu_\omega^k}'(x)|^2 +\beta \norm{\nabla^{\mathfrak{g}} J_{\mu_\omega^k}'(x)}_{x}^2\right) \, \d\mu_\omega^k(x) -\frac{1}{2} \int_{\X\setminus \tilde \X} \alpha |J_{\mu_\omega^k}'(x)|^2  \, \d\mu_\omega^k(x) \,.\]
\end{lemma}

The inequality displayed in Lemma \ref{lemma:descent_property} describes the behavior of the increment of the sequence $J_W(\mu_\omega^k)$ in terms of some energy terms related to the Fréchet derivative $J'_{\mu_\omega^k}$ and its gradient. This result is in line with the KKT conditions associated with our problem. In particular, the descent property holds as long as $J'_{\mu_\omega^k}\neq 0$ on a set where $\mu_\omega^k$ is strictly positive. However, the descent breaks off as soon as $J'_{\mu_\omega^k}$ cancels out in the support of $\mu_\omega^k$, which does not guarantee that $\mu_\omega^k$ is indeed a global minimum.

\paragraph{Non-degeneracy of the solution}
Establishing strong convergence guarantees requires structural assumptions on the solution $\mu_\omega^\star$. Specifically, the latter must be a discrete measure satisfying a \emph{non-degeneracy} condition (see Assumption \ref{assumption:non_degenerate_certif_solution} below). This condition involves the Riemannian Hessian, which we define next.

\begin{definition}[Riemannian Hessian]\label{def:riemannian_hessian}
Let $\psi \in \C^2(\R^d \times [u_{\min}, +\infty)^d)$. Let $x \in \R^d \times [u_{\min}, +\infty)^d$.
We define \[H^{\mathfrak{g}} \psi(x)=\nabla^2 \psi(x)-\sum_{k=1}^d\Gamma^{t_k} \partial_{t_k}\psi(x)-\sum_{k=1}^d\Gamma^{u_k} \partial_{u_k}\psi(x)\,,\] with $\Gamma^{t_k},\Gamma^{u_k}$ defined by \eqref{eq:christoffel_symbols}.
\end{definition}

\begin{assumption}[Non-degenerate certificate for the solution] \label{assumption:non_degenerate_certif_solution}
\textbf{Location of the solution:} The solution to \eqref{eq:pb_BLASSO_gaussian_kernel_norm} is unique and of the form
$\mu_\omega^\star=\sum_{j=1}^{p^\star} \omega_j^\star \delta_{x_j^\star}$. Moreover, it satisfies $\omega_j^\star \geq \omega_\star>0$ and $x_j^\star \in \mathring{\tilde \X}$.
\\\textbf{Non-degeneration:}
In addition, there exist $\varepsilon_2>0$, $C_{\tilde \Upsilon}>0$ and a non-degenerate certificate associated with $\mu_\omega^\star$, namely \[\eta^\star : x \in \R^d \times [u_{\min},+\infty)^d \mapsto -\frac{1}{\kappa} \innerprod{\Psi \delta_x}{\Psi \mu_\omega^\star - y}_\L= 1-\frac{J_{\mu_\omega^\star}'(x)}{\kappa}\] satisfying
\begin{enumerate}
 \item $\eta^\star \leq 1$ on $\X$,
 \item $\eta^\star(x)=1 \iff x \in \{x_j^\star\}_{j=1}^{p^\star}$,
 \item $\nabla \eta^\star(x_j^\star)=0$ for all $j=1,\ldots,p^\star$,
  \item $-v^T H^\mathfrak{g} \eta^\star(x_j^\star) v \geq \varepsilon_2 \norm{v}_{x_j^\star}^2$ for all $v\in \R^{2d}$, $j\in\{1,\ldots,p^\star\}$\,,
 \item $\tilde \Upsilon \succeq C_{\tilde \Upsilon} I$ with $C_{\tilde \Upsilon} >0$,
where 
\begin{equation}
\label{eq:tilde_Upsilon}
\tilde \Upsilon\coloneq  \begin{pmatrix}
(K_{\rm norm}(x_j^\star,x_i^\star))_{j,i=1,\ldots,p^\star} & (\nabla_2 K_{\rm norm}(x_j^\star,x_i^\star)^T \mathfrak{g}_{x_i^\star}^{-1/2})_{j,i=1,\ldots,p^\star} \\
(\mathfrak{g}_{x_j^\star}^{-1/2}\nabla_1 K_{\rm norm}(x_j^\star,x_i^\star))_{j,i=1,\ldots,p^\star} & (\mathfrak{g}_{x_j^\star}^{-1/2}\nabla_1 \nabla_2 K_{\rm norm}(x_j^\star,x_i^\star) \mathfrak{g}_{x_i^\star}^{-1/2})_{j,i=1,\ldots,p^\star}
\end{pmatrix} \,.
\end{equation}
\end{enumerate}
\end{assumption}

The non-degeneracy assumption is in line with \citep{off-the-grid_cs,blasso_duval_peyre}, where it is introduced to establish the closeness between $\mu_\omega^\star$ and the statistical target measure $\mu^0$. It is referred to as the non-degenerate source condition (NDSC), which postulates the existence of a non-degenerate certificate associated with $\mu^0$.

Here, however, the non-degenerate certificate is associated with $\mu_\omega^\star$ rather than with $\mu^0$. The connection between the non-degeneracy of $\mu^0$ and that of the BLASSO solution is discussed in \citep[Section 3]{blasso_duval_peyre}.

Note that any solution of the BLASSO satisfies $\eta^\star \leq 1$ on $\X$ and $\eta^\star=1$ on $\supp(\mu_\omega^\star)$ (this follows from the KKT conditions, \cf \citep[Proposition 3.1]{cpgd_chizat}).

\paragraph{Convergence guarantees} Under Assumption \ref{assumption:non_degenerate_certif_solution}, combined with a strong condition on the initialization, we establish exponential convergence of the objective values, as stated in the following theorem whose proof is given in Appendix \ref{section:proof_exponential_cv}.

\begin{theorem}[Exponential local convergence under non-degeneracy of the solution]\label{th:exponential_cv} Let $0<\kappa \leq 1$.
\\\underline{Conditions on the solution:} We suppose that Assumption \ref{assumption:non_degenerate_certif_solution} is satisfied.
\\\underline{Choice of initialization, updates:} Assume that $\mu_\omega^1 = \sum_{j=1}^p \omega_j^1 \delta_{x_j^1}$ with $\omega_j^1>0$ and $x_j^1\in\X$. Assume that it satisfies 
\begin{subequations}\label{eq:exponential_cv}
\begin{equation}
\label{eq:cond_init}
J_W(\mu_\omega^1) - J_W(\mu_\omega^\star) \leq J_0 \kappa^{12}\, ,
\end{equation}
with $J_0>0$ depending on $\X$, $\tilde \X$, $\tau$, $\mu_\omega^\star$, $\eta^\star$, $\norm{y}_\L$, $\norm{\mu_\omega^1}_{\rm TV}$.
We use Algorithm \ref{algo:CPGD_Frozen_Riemannian} together with the parameter choices $\alpha, \beta$ described in Lemma \ref{lemma:descent_property}.
\\\underline{Convergence of CPGD:} There exists $R_0>0$, depending on $\X$, $\tilde \X$, $\tau$, $\mu_\omega^\star$, $\eta^\star$, $\norm{y}_\L$ and $\norm{\mu_\omega^1}_{\rm TV}$, such that $\frac12\min\{\alpha,\beta\}R_0\kappa^6<1$ and, for all $k\geq 1$,
\begin{equation}
\label{eq:rate_exponential_cv}
J_W(\mu_\omega^k)- J_W(\mu_\omega^\star) \leq \left(1-\frac{1}{2} \min\{\alpha,\beta\} R_0 \kappa^6 \right)^{k-1} (J_W(\mu_\omega^1)- J_W(\mu_\omega^\star))\,.
\end{equation}
\end{subequations}
\end{theorem}

In contrast to Lemma \ref{lemma:descent_property}, Theorem \ref{th:exponential_cv} describes the behavior of the difference between the current state $J_W(\mu_\omega^k)$ and the target $J_W(\mu_\omega^\star)$ in terms of the iteration number $k$. In particular, it establishes exponential convergence of the objective values, provided that $\kappa$ is small enough and that the initialization condition \eqref{eq:cond_init} holds; the convergence is hence said to be local.

Three limitations should be kept in mind when reading Theorem \ref{th:exponential_cv}. First, the constants $J_0$ and $R_0$ depend on the solution $\mu_\omega^\star$ and on its associated certificate $\eta^\star$, both of which are unknown. Condition \eqref{eq:cond_init} should therefore be understood as a qualitative locality requirement, and not as a criterion that can be evaluated on a given initialization; likewise, the rate in \eqref{eq:rate_exponential_cv} is not a quantity the user can compute in advance. Second, for the sake of proof readability, the exponents for $\kappa$ in \eqref{eq:cond_init} and \eqref{eq:rate_exponential_cv} are not optimized, as discussed in Remark \ref{remark:dependence_kappa_exp_convergence}. For $\kappa$ small, the resulting basin of attraction is correspondingly narrow, and the guarantee is qualitative rather than quantitative in nature. Third, Theorem \ref{th:exponential_cv} says nothing about how the initialization condition may be met: Proposition \ref{prop:weight_convergence} controls the Cesàro average rather than the last iterate, and we do not claim that combining the two yields a complete convergence theory. We note that the numerical behavior reported in Section \ref{section:numerical_experiments} is considerably more favorable than these constants would suggest.

Since the problem \eqref{eq:pb_BLASSO_gaussian_kernel_norm} restricted to discrete measures is non-convex, obtaining global convergence guarantees remains challenging. We refer, for instance, to \citep{decastro2026fsp} for a similar problem involving a translation-invariant kernel.

\paragraph{Checking the non-degeneracy assumption}
Assumption \ref{assumption:non_degenerate_certif_solution} allows to obtain an exponential local convergence for Algorithm \ref{algo:CPGD_Frozen_Riemannian} and plays a crucial role in our analysis. Nevertheless, it is impossible to check in practice. We can however try to give some hints on this condition in a specific setting, by shifting the condition on the solution $\mu_\omega^\star$ to an underlying  statistical ground truth. More precisely, we assume that the measurement $y$ can be written as
\begin{equation}\label{eq:form_y_noise_target}
    y=\Psi \mu_\omega^0 + b \,,
\end{equation} where $b \in \L$ and $\mu_\omega^0 \coloneq \sum_{j=1}^s \omega_j^0 \delta_{x_j^0}$, with $\omega_j^0>0$ and $x_j^0 \in \X$. We interpret $b$ as noise and refer to $\mu_\omega^0$ as the statistical target. Section \ref{section:statistical_motivation} discusses the relevance of this model in a statistical context. 

We prove below that Assumption \ref{assumption:non_degenerate_certif_solution} holds provided that the noise $b$ and the regularization parameter $\kappa$ are sufficiently small, and that the particles of $\mu_\omega^0$ are well separated with respect to a tailored semi-distance. This semi-distance is defined, for $x,x'\in \R^d \times [u_{\min},+\infty)^d$, by
\begin{equation}\label{eq:def_semi_distance}
 \mathpzc{d}(x,x')^2= -2 \ln(K_{\rm norm}(x,x'))= \sum_{k=1}^d \left(\frac{(t_k-t_k')^2}{{u_k'}^2+u_k^2+\tau^2} + \ln\left(\frac{{u_k'}^2+u_k^2+\tau^2}{\sqrt{2u_k^2+\tau^2}\sqrt{2{u_k'}^2+\tau^2}}\right)\right)\,.
\end{equation} 
These conditions are formalized in the following theorem whose proof is postponed to Appendix \ref{section:proof_checking_non_degeneracy_solution}.

\begin{theorem}[Checking Assumption \ref{assumption:non_degenerate_certif_solution}]\label{th:check_assumptions_solution}
Assume that $y$ is of the form \eqref{eq:form_y_noise_target}, that $s\geq 2$ and $\{x_j^0\}_{j=1}^s \subset \mathring{\tilde \X}$. 
Assume that $\min_{i \neq j}\mathpzc{d}(x_j^0,x_i^0) \geq \Delta_\tau$, where
\begin{subequations}\label{eq:check_assumptions}
\begin{equation}
\label{eq:def_Delta_tau}
\Delta_\tau \coloneq \max \left\{\frac{\sqrt{u_{\max}^2 +\frac{1}{4} \frac{0.3025^2}{d}(2u_{\max}^2+\tau^2)}}{u_{\min}} \left(\Delta+\frac{0.3025}{\sqrt{d}}\right) \, , \, 2\frac{u_{\max}}{u_{\min}} \Delta \right\} + \sqrt{d\ln \left(\frac{u_{\max}^2}{u_{\min}^2}\right)}
\end{equation}
and
\begin{equation}
\label{eq:def_Delta}
\Delta= 2\sqrt{11.9+3\ln(d+6.62)+\ln(s-1)}\,.
\end{equation}
\end{subequations}

There exists $\gamma_0>0$ depending on $d$, and $C_\kappa>0$ depending on $\X,\tilde \X, \tau, \mu_\omega^0$ such that if $0<\kappa\leq C_\kappa$ and $\norm{b}_\L \leq \gamma_0 \kappa$, then Assumption \ref{assumption:non_degenerate_certif_solution} is satisfied with $C_{\tilde \Upsilon}=\frac{1}{4}$.
\end{theorem}

This result relies on the connection between the Non-Degenerate Source Condition introduced by \citep{blasso_duval_peyre}, studied in our setting by \citep{article_giard_decastro_marteau}, which concerns the non-degeneracy of the target $\mu_\omega^0$, and the non-degeneracy of the BLASSO solution. 

The numerical constants appearing in \eqref{eq:def_Delta_tau} and \eqref{eq:def_Delta} are those of the local positive curvature property satisfied by $K_{\rm norm}$: the radius $\frac{0.3025}{\sqrt{d}}$ and the associated curvature parameters are established in \citep[Theorem 5.3 and Lemma 5.4]{article_giard_decastro_marteau}, and the separation \eqref{eq:def_Delta_tau} is the corresponding condition of \citep[Theorem 5.1]{article_giard_decastro_marteau}.

It is worth making explicit what is imported from that work and what is established here. The construction of non-degenerate dual certificates for the statistical target $\mu_\omega^0$, together with the separation condition under which it succeeds, is the object of \citep{article_giard_decastro_marteau}. Assumption \ref{assumption:non_degenerate_certif_solution}, by contrast, bears on the BLASSO solution $\mu_\omega^\star$, and it is this form that the analysis of Section \ref{section:exponential_convergence} requires: the descent argument is driven by the behavior of $\eta^\star$ in the vicinity of $\supp(\mu_\omega^\star)$, a set that is not known in advance and does not coincide with $\supp(\mu_\omega^0)$. Theorem~\ref{th:check_assumptions_solution} accordingly performs two steps that have no counterpart in \citep{article_giard_decastro_marteau}: it transfers non-degeneracy from $\mu_\omega^0$ to $\mu_\omega^\star$, which calls for quantitative control of both the locations and the weights of the solution in the small-$\kappa$, small-noise regime; and it upgrades the resulting condition to a Riemannian one, expressed through the Fisher-Rao Hessian $H^{\mathfrak{g}}$ and the metric-normalized matrix $\tilde \Upsilon$ of \eqref{eq:tilde_Upsilon}. The normalization by $\mathfrak{g}_{x_j^\star}^{-1/2}$ is what renders the curvature bound uniform in the location-scale parameters, and hence usable to produce the explicit rate of Theorem \ref{th:exponential_cv}.

\medskip
Note that Theorems \ref{th:exponential_cv} and \ref{th:check_assumptions_solution} establish pointwise convergence rates. In Theorem \ref{th:check_assumptions_solution}, the condition on $b$ required for Assumption \ref{assumption:non_degenerate_certif_solution} to hold depends on the statistical target $\mu_\omega^0$. Meanwhile, in Theorem \ref{th:exponential_cv}, both the condition on the initialization and the resulting convergence rate depend on the solution $\mu_\omega^\star$.

\paragraph{Positioning relative to \citep{cpgd_chizat}} Our contribution relies on \citep[Corollary 3.5]{cpgd_chizat}.
A difference with the setting of \citep{cpgd_chizat} is that the domain of the particle positions is not a compact manifold without boundary. The presence of a boundary in $\X$ requires a specific treatment, which we handle by freezing particles outside $\tilde \X$.

In addition, we provide an interpretable condition under which \citep[Assumption 5]{cpgd_chizat}, is satisfied. This yields a concrete criterion for verifying the assumption in our setting. To the best of our knowledge, this is the first work showing that, in a statistical framework, the non-degeneracy assumption of \citep{cpgd_chizat} is met under a more explicit condition on the support separation and noise boundedness.

\section{Numerical experiments}
\label{section:numerical_experiments}

In this section, we discuss the practical implementation of CPGD for the statistical estimation of a Gaussian mixture model with unknown (diagonal) covariances and present numerical experiments on simulated data. The code is shared in the archived release \citep{repo_zenodo_cpgd}. We first provide a proof of concept in a favorable setting, illustrating the good convergence properties of CPGD. We then compare its performance to the EM algorithm. Finally, we investigate the influence of the separation between target components on recovery accuracy.

\subsection{Statistical motivation}\label{section:statistical_motivation}
The optimization problem \eqref{eq:pb_BLASSO_gaussian_kernel_norm} is particularly relevant in a statistical setting.
In this framework, the goal is to recover the target probability measure encoding a Gaussian mixture distribution,
\[\mu^0=\sum_{j=1}^{s} a_j^0\delta_{x_j^0} 
\quad \text{where} 
\quad a_1^0,\ldots a_s^0>0\,, \; \sum_{j=1}^{s} a_j^0=1\,.\] We assume that the location-scale parameters $(x_1^0, \ldots, x_s^0)$ are distinct, where for any $j\in \lbrace 1,\dots, s \rbrace$, $x_j^0 = (t_j^0, u_j^0)$, with $t_j^0 = (t_{j,k}^0)_{k=1}^d \in \R^d$ and $u_j^0 = (u_{j,k}^0)_{k=1}^d \in [u_{\min}, +\infty)^d$. The parameters $t_j^0$ and $u_j^0$ represent respectively the mean and square root of the diagonal covariance of the $j^{th}$ Gaussian component with weight $a_j^0$. The lower bound $u_{\min}$ on the minimal eigenvalue of each covariance matrix is a standard constraint for Gaussian mixture estimation based on likelihood maximization (see, \eg, \citep{bouveyron_model-based_2019}).

We observe a sample $X_1,\ldots,X_n \iid f_0$ drawn according to a Gaussian mixture distribution, whose density $f_0$ is associated with the measure $\mu^0$. In particular, it can be rewritten as $f_0\coloneq \Phi \mu^0$ where $\Phi:\M(\R^d \times [u_{\min},+\infty)^d) \longrightarrow L^2(\R^d)$ is the linear operator defined by \[\Phi\mu : z \in \R^d \mapsto \int_{\R^d \times [u_{\min},+\infty)^d} \prod_{k=1}^d \frac{1}{\sqrt{2\pi}u_k} e^{\frac{-(z_k - t_k)^2}{2u_k^2}} \,\d \mu(t,u) \quad \forall \, \mu \in \M(\R^d \times [u_{\min},+\infty)^d)\,.\]

To compare the observations with a candidate density, we use a Gaussian convolution with smoothing parameter $\tau$: introducing $\lambda: z \in \R^d \mapsto \frac{e^{-\frac{\norm{z}_2^2}{2\tau^2}}}{(2 \pi \tau^2)^{d/2}}$, we define
$y=\frac{1}{n} \sum_{i=1}^n \lambda(\dotp - X_i)$. The term $y$ involved in \eqref{eq:problemPkappa'} corresponds in such a case to a kernel estimator associated with the density $f_0$.

We then introduce a reparameterized version of the target measure, $\mu_\omega^0 = \sum_{j=1}^s \omega_j^0 \delta_{x_j^0}$ with $\omega_j^0=W(x_j^0)a_j^0$, where \begin{equation}\label{eq:def_W}
 W(x)\coloneq \prod_{k=1}^d (2\pi)^{-1/4}(2u_k^2+\tau^2)^{-1/4} \quad \forall \, x=((t_1,\ldots,t_d),(u_1,\ldots,u_d)) \in \R^d \times [u_{\min},+\infty)^d\,.
\end{equation} 
With this notation, one has \[\Psi \mu = \lambda * \Phi \frac{\mu}{W} \quad \forall \, \mu \in \M(\R^d \times [u_{\min},+\infty)^d)\,,\] which makes explicit the connection between the BLASSO problem \eqref{eq:pb_BLASSO_gaussian_kernel_norm} and the recovery of the parameters of a Gaussian mixture model. In this formulation, a solution to \eqref{eq:pb_BLASSO_gaussian_kernel_norm} approximates $\mu_\omega^0$ rather than $\mu^0$. This renormalization is essential for obtaining recovery guarantees \citep{article_giard_decastro_marteau}, as it leads to a kernel that is normalized (\ie, $x \mapsto K_{\rm norm}(x,x)$ is constant).

Note that in this statistical setting, $y = \lambda *  \hat{f}_n$ where $\hat{f}_n=\frac{1}{n}\sum_{i=1}^n \delta_{X_i}$. 
The updates \eqref{eq:EGD}, \eqref{eq:NGD}, \eqref{eq:RGD} admit closed-form expressions when $y$ is of the form $\lambda * \nu$ with $\nu \in \M(\R^d)$ a discrete measure. In this setting, the integrals appearing in the norms and inner products on $\L$ can be computed exactly, and no Monte Carlo approximation is required (the formal computation is given in \texttt{cpgd\_sympy.ipynb}, and the implementation is available in \texttt{methods.py} in \citep{repo_zenodo_cpgd}).

\begin{remark}[Extension to general covariance structures]
We focus our attention on the case where the covariance matrices associated with each component of the mixture are diagonal. These investigations are grounded in the statistical analysis displayed by \citep{article_giard_decastro_marteau} in the same framework. Extending this theoretical analysis to general covariance structures is still an open problem. From a computational perspective, the extension of the CPGD principle to general covariance matrices is not straightforward. It indeed requires to handle a non-diagonal metric tensor. Guaranteeing the positive definiteness of the estimated matrices is also a key point. Such an extension hence goes outside the scope of this paper.
\end{remark}

\subsection{AdaGrad with natural gradient descent}
In Section \ref{section:theoretical_analysis}, theoretical results are presented, when the Riemannian gradient descent with fixed step sizes is used, together with a freezing mechanism (see Algorithm \ref{algo:CPGD_Frozen_Riemannian}).
In practice, the step sizes $\alpha,\beta$ must be chosen sufficiently small to ensure that the updates are well defined (Lemma~\ref{lemma:descent_property}). Moreover, fixing a prior region outside of which the location–scale parameters are frozen might be cumbersome. It requires identifying, at each iteration, the particles that lie outside this region, and selecting $\beta$ sufficiently small to ensure that the iterates remain in $\R^d \times (0,+\infty)^d$.
Finally, computing the geodesic trajectories—even though they are available in closed form—requires more computation than a classical Euclidean gradient descent.

For these reasons, we use the natural gradient descent \eqref{eq:NGD} for our numerical experiments. We stress that the \eqref{eq:NGD} scheme can be considered as an approximation of the Riemannian gradient descent \eqref{eq:RGD} when the step sizes are small. This choice is theoretically justified (see Section \ref{section:optimization_procedure}), and also leads to better results than Euclidean gradient descent in practice.
Moreover, we do not perform any projection or freezing outside a prescribed domain; instead, we retain only the componentwise absolute value of $u_j^k$, noting that the loss $J_W$ is even in $u_{j,l}$, for $1\leq l \leq d$.
Finally, we use AdaGrad \citep{duchi_adagrad} to avoid dealing with step sizes, and to improve convergence. This construction leads to Algorithm \ref{algo:CPGD_ngd_adagrad}.

\begin{algorithm}[H]
\caption{CPGD with NGD and AdaGrad}
\label{algo:CPGD_ngd_adagrad}
\begin{algorithmic}[1]
\Require Numerical stability term $\epsilon>0$, number of iterates $K$, initialization $\mu_\omega^1 = \sum_{j=1}^p \omega_j^1 \delta_{x_j^1}$

\State $v_{\omega_j}^0 \gets 0$, \quad $v_{x_j}^0 \gets 0$ \quad for $j=1,\dots,p$ \Comment{Initialize AdaGrad accumulators}

\For{$k = 1,\ldots K-1$}
    \For{$j = 1,\ldots,p$}
    \State $v_{\omega_j}^{k+1} \gets v_{\omega_j}^k + |J_{\mu_\omega^k}'(x_j^k)|^2$, \quad $v_{x_j}^{k+1} \gets v_{x_j}^k + \frac{1}{2d}\norm{\nabla^{\mathfrak{g}}J_{\mu_\omega^k}'(x_j^k)}_{x_j^k}^2$ \Comment{Update AdaGrad accumulators}
    
      \State 
    $\omega_j^{k+1} \gets \omega_j^k\exp\left(- \frac{J_{\mu_\omega^k}'(x_j^k)}{\sqrt{v_{\omega_j}^{k+1}}+\epsilon}\right)$ \Comment{Weight update}
        
    \State
       $x_j^{k+1}=(t_j^{k+1},u_j^{k+1}) \gets x_j^k - \frac{\nabla^{\mathfrak{g}}J_{\mu_\omega^k}'(x_j^k)}{\sqrt{v_{x_j}^{k+1}}+\epsilon}$ \Comment{Location-scale update}

\State $u_j^{k+1} \gets |u_j^{k+1}|$
\Comment{Componentwise absolute value}
    \EndFor
    \State $\mu_\omega^{k+1} \gets \sum_{j=1}^p \omega_j^{k+1} \delta_{x_j^{k+1}}$
\EndFor
\end{algorithmic}
\end{algorithm}
Note that we do not use separate adaptive steps for each coordinate. Instead, we accumulate the squared norms of the gradients \emph{per particle}. The underlying metric is essential: the Riemannian gradients determine how the relative updates of $(t_1,\ldots,t_d,u_1,\ldots,u_d)$ should be scaled. In practice, the adaptation operates at the level of particles, controlling the relative step sizes between weights and positions, as well as across different particles.

We emphasize that Algorithm \ref{algo:CPGD_ngd_adagrad}, which produces all the numerical results reported below, differs from Algorithm \ref{algo:CPGD_Frozen_Riemannian}, to which the analysis of Section \ref{section:theoretical_analysis} applies: it replaces the Riemannian update \eqref{eq:RGD} by \eqref{eq:NGD}, the fixed step sizes of Lemma \ref{lemma:descent_property} by AdaGrad, and the freezing mechanism by the componentwise absolute value. None of the guarantees of Section \ref{section:theoretical_analysis} therefore covers Algorithm \ref{algo:CPGD_ngd_adagrad}; the experiments below should be read as an empirical study of a practical variant.

\paragraph{Hyperparameters}
The estimation procedure is sensitive to both the smoothing parameter $\tau$ and the regularization parameter $\kappa$.
If $\kappa$ is chosen too large, the weights tend to shrink to zero. Conversely, if $\kappa$ is too small, additional local minima may appear, and some irrelevant particles may persist in the solution.

Regarding $\tau$, excessively large values may deteriorate the accuracy of the estimation when the components are not well separated. In particular, $\tau$ should not be too large relative to $u_{\min}$. On the other hand, if $\tau$ is taken too small, the theoretical bounds governing the estimation become less favorable \citep[Theorem 1.1]{article_giard_decastro_marteau}. Moreover, in certain settings, the estimator returned by CPGD is not robust with respect to $\tau$, as even small variations in this parameter can result in significant changes in estimation quality.

In practice, the parameters $\kappa$ and $\tau$ can be tuned by cross-validation, for instance using a Monte Carlo estimate of the KL divergence on a test sample.

\subsection{Scores}
\label{section:numerical_scores}
We aim to evaluate the recovery of $\mu^0$ achieved by CPGD, by comparing the last estimate $\mu^K$ produced by Algorithm \ref{algo:CPGD_ngd_adagrad} to the ground truth $\mu^0$. 
Since many discrepancy measures require probability measures, we will in fact compare $\mu^0$ with the renormalized version of $\mu^K$ defined as 
\[\mu_{\rm norm}^K = \frac{\mu^K}{\norm{\mu^K}_{\rm TV}} = \sum_{j=1}^p a_{j, \rm norm}^K \delta_{x_j^K}\,, \quad  \text{where} \quad a_{j, \rm norm}^K = \frac{a_j^K}{\sum_{j=1}^p a_j^K} \ \forall j\in \lbrace 1,\dots, p \rbrace.\]

We consider several quantitative criteria to illustrate the discrepancy between the reconstructed measure $\mu_{\rm norm}^K$ and the target $\mu^0$.  First, we compute a Monte–Carlo approximation of the Kullback–Leibler divergence $D_{\rm KL}(\Phi \mu^0, \Phi \mu_{\rm norm}^K)$ between $\Phi\mu^0$ and $\Phi\mu_{\rm norm}^K$. Formally, given a sample $Z_1, \ldots, Z_m \iid \Phi \mu^0$, we use the approximation
\[D_{\rm KL}(\Phi \mu^0, \Phi \mu_{\rm norm}^K) \approx \frac{1}{m} \sum_{i=1}^m \ln\left(\frac{\Phi \mu^0(Z_i)}{\Phi\mu_{\rm norm}^K(Z_i)}\right) \eqcolon S^{\rm KL}\,,\]
with $m=10000$. Second, we compute a Wasserstein-type distance between $\mu^0$ and $\mu_{\rm norm}^K$, defined by \[S =\min_{\substack{\gamma_{ij} \geq 0, \\ \sum_j \gamma_{ij} = a_i^0,  \\ \sum_i \gamma_{ij} = a_{j, \rm norm}^K }} \sum_{ij} \gamma_{ij} d(x_i^0,x_j^K)\,.\]
Several choices of ground distance $d$, independent of $\tau$, can be considered:
\begin{itemize}
    \item The semi-distance defined for $x,x' \in \R^d \times (0,+\infty)^d$ by \[\mathpzc{d}_0(x,x')^2 = \sum_{k=1}^d \frac{(t_k-t_k')^2}{u_k^2+u_k'^2} + \ln\left( \frac{u_k^2 + u_k'^2}{2u_k u_k'}\right)\,.\]
We refer to \eqref{eq:def_semi_distance} or \citep{article_giard_decastro_marteau} for more details on $\mathpzc{d}_0$. Please note that the associated Wasserstein discrepancy $S^\mathpzc{d}$ is only a semi-distance in such a case. 
   \item The Fisher-Rao distance with $\tau=0$, \ie the Poincaré half-plane distance \[\mathfrak{d}_{\mathfrak{g} 0}(x,x')^2 = 2\sum_{k=1}^d  \ln \left( \frac{\sqrt{(t_k-t_k')^2 + (u_k-u_k')^2} + \sqrt{(t_k-t_k')^2 + (u_k+u_k')^2}}{2\sqrt{u_ku_k'}} \right)^2\,.\]
   In such a case, $S^{\mathfrak{d}_\mathfrak{g}}$ is a distance.
    \item The Euclidean distance, for which $S^e$ reduces to the classical Wasserstein–1 distance.
\end{itemize}

In the following, we report results using the scores $S^{\mathpzc{d}}$, $S^{\mathfrak{d}_\mathfrak{g}}$, $S^e$ and $S^{\rm KL}$. 

\subsection{Proof of concept}

We first illustrate the behavior of the algorithm in a favorable setting where successful recovery is expected. More precisely, we consider the one-dimensional case with well-separated components.
The target measure consists of two particles with equal weights and identical scale parameters, but distinct means, namely
\begin{equation}
\label{eq:poc_measure}
\mu^0 = \frac{1}{2} \delta_{(-1,1)} + \frac{1}{2} \delta_{(1,1)}\,.
\end{equation}

\begin{figure}[ht]
    \centering
    \includegraphics[width=0.5\linewidth]{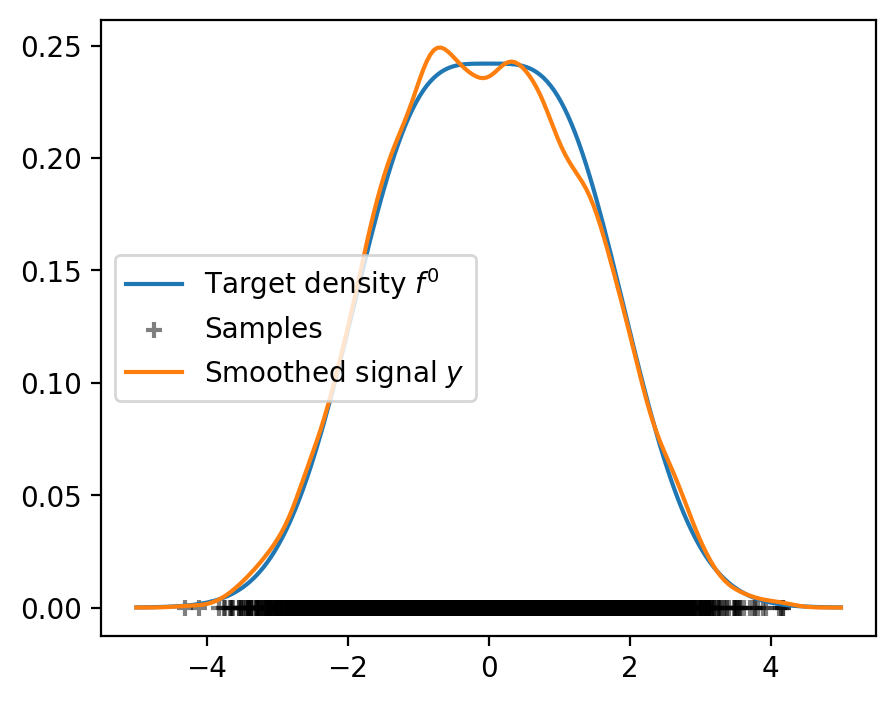}
    \caption{Target mixture and observation $y$ where $\mu^0$ is defined in \eqref{eq:poc_measure}.
    }
    \label{fig:proof_concept_target}
\end{figure}
We consider a sample of size $n=5000$ drawn from the Gaussian mixture distribution associated with $\mu^0$ (Figure \ref{fig:proof_concept_target}). The hyperparameters are set to $\tau=0.2$ and $\kappa = 10^{-5}$. Algorithm~\ref{algo:CPGD_ngd_adagrad} is initialized with three particles and run for 5000 iterations. 

The performance of our approach is illustrated in Figure \ref{fig:proof_concept_iterates}. The first curve, on the left, illustrates the descent behavior. In particular, the loss decreases sharply over the iterations (on a logarithmic scale). Although the overall behavior is not monotone, the loss decreases for the most part before reaching a plateau after approximately $10^2$ iterations.
Starting from three particles at initialization, we observe that two remain significant after a few iterations, while the third is associated with a small final weight (around $0.15$). The plot on the right illustrates the trajectories of the individual particles: one converges to one of the target atoms (located in $(-1,1)$), while the other two particles, including the one associated with a small weight, converge to the same target atom.
Although this issue is beyond the scope of this paper, one could consider introducing a post-processing procedure to merge particles with similar positions, based on an appropriate similarity criterion.

To conclude this discussion, we can notice that our algorithm produces the scores
\[S^{\mathpzc{d}}=0.064\,, \quad S^{\mathfrak{d}_\mathfrak{g}}=0.059 \,, \quad S^e = 0.091\,, \quad \text{and} \quad S^{\rm KL}=0.00088 \,, \]
for this experiment. 

\begin{figure}[ht]
    \centering
    \includegraphics[width=\linewidth]{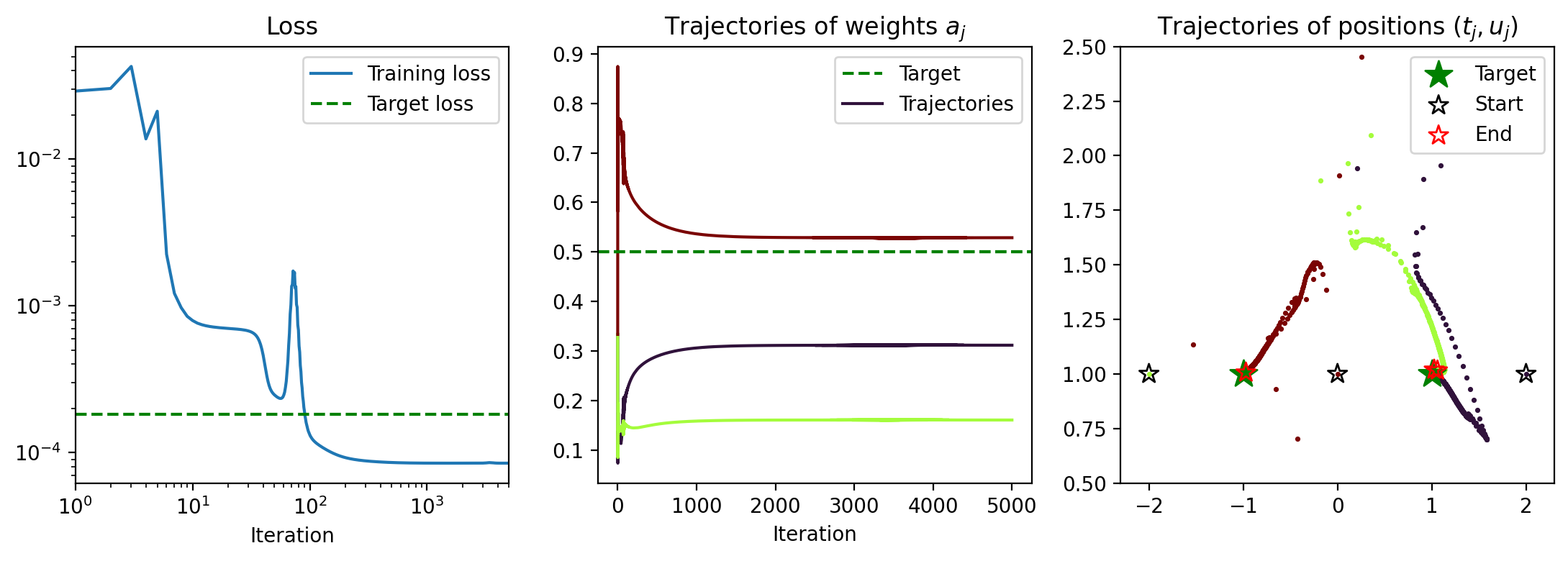}
    \caption{Illustration of the convergence of Algorithm \ref{algo:CPGD_ngd_adagrad}. From left to right: evolution of the loss in terms of the iteration number, trajectories of the weights according to the iteration number and trajectories of the positions (mean $t$ on the x-axis, scale parameter $u$ on the y-axis). Target: $\mu^0$ of \eqref{eq:poc_measure}, $n=5000$, $p=3$ particles, $\tau=0.2$, $\kappa=10^{-5}$, $K=5000$ iterations.
    }
    \label{fig:proof_concept_iterates}
\end{figure}

\subsection{Comparison with EM}
The EM algorithm is the standard approach for estimating Gaussian mixture models and can be considered as a benchmark for numerical comparisons. Specifically, we use here the vanilla diagonal EM algorithm (which estimates only diagonal covariance matrices) as implemented in scikit-learn. No regularization is applied to the covariance estimates.

The comparison displayed in this section is not intended to be exhaustive. It rather illustrates performance in a representative setting. We consider below a two-dimensional problem with three target components sharing the same mean and equal weights, but with different covariance structures. More precisely, our target measure is 
\begin{equation}
\mu^0 = \frac{1}{3} \delta_{(0,0,\sigma_1,\sigma_2)} + \frac{1}{3} \delta_{(0,0,\sigma_2,\sigma_1)} + \frac{1}{3} \delta_{(0,0,\sigma_3,\sigma_3)} \quad \text{with} \: \sigma_1=0.5,\: \sigma_2=2, \: \sigma_3= \sqrt{\sigma_1^2+\sigma_2^2} \,.
\label{eq:cross_measure}
\end{equation} 
The first two components have anisotropic covariances forming a cross-shaped pattern in the observation space, while the third one is isotropic.
We consider a sample of size $n=1000$ drawn from $\Phi \mu^0$.

\begin{figure}[ht]
    \centering
    \includegraphics[width=0.6\linewidth]{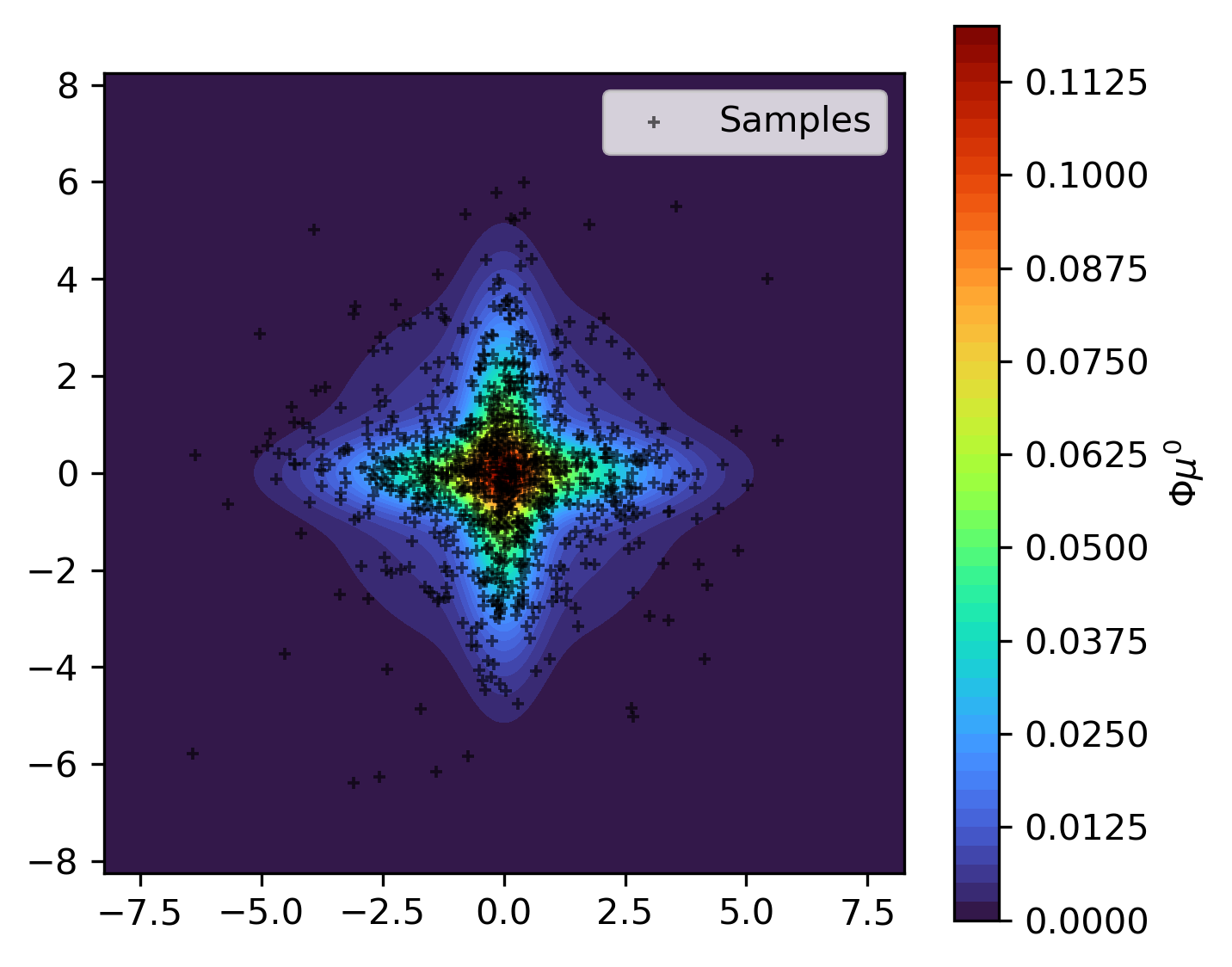}
    \caption{Target mixture and samples when $\mu^0$ is defined according to \eqref{eq:cross_measure}.
    }
    \label{fig:comparison_target}
\end{figure}

In this specific setting, we set $\tau=0.7$, $\kappa = 0.01$ in Algorithm \ref{algo:CPGD_ngd_adagrad}. Both EM and Algorithm \ref{algo:CPGD_ngd_adagrad} are initialized identically with $p$ particles of equal weights, whose means are uniformly distributed on a circle of radius $1$ centered at $(0,0)$, and with identical initial scale parameters $u_j^1 = (1,1)$. Both algorithms are then run for 3000 iterations.

\paragraph{Initialization with 3 particles}
The case $p=3$ is particularly favorable to the EM algorithm, since it is specifically designed for situations in which the number of mixture components is known in advance.

In this setting, both algorithms yield comparable results.
For CPGD (Figure \ref{fig:comparison_cpgd_iterates}), we obtain
\[S^{\mathpzc{d}}=0.11\,, \quad S^{\mathfrak{d}_\mathfrak{g}}=0.11 \,, \quad S^e = 0.28\,, \quad S^{\rm KL}=0.023 \,, \]
while for EM (Figure \ref{fig:comparison_em_iterates}), the corresponding scores are
\[S^{\mathpzc{d}}=0.14\,, \quad S^{\mathfrak{d}_\mathfrak{g}}=0.15 \,, \quad S^e = 0.25\,, \quad S^{\rm KL}=0.0087 \,. \]
The Wasserstein-type scores are of similar magnitude for both methods, whereas the KL divergence is smaller for the EM algorithm.

Concerning the trajectories of the parameter estimates, we also notice globally similar behaviors. Note that the estimated weights for Algorithm \ref{algo:CPGD_ngd_adagrad} are (as expected) smaller than those of the target (Figure \ref{fig:comparison_cpgd_iterates}). This is a consequence of the $\ell^1$ regularization of the weights induced by the total variation norm in the loss.

\begin{figure}[ht]
    \centering
    \includegraphics[width=\linewidth]{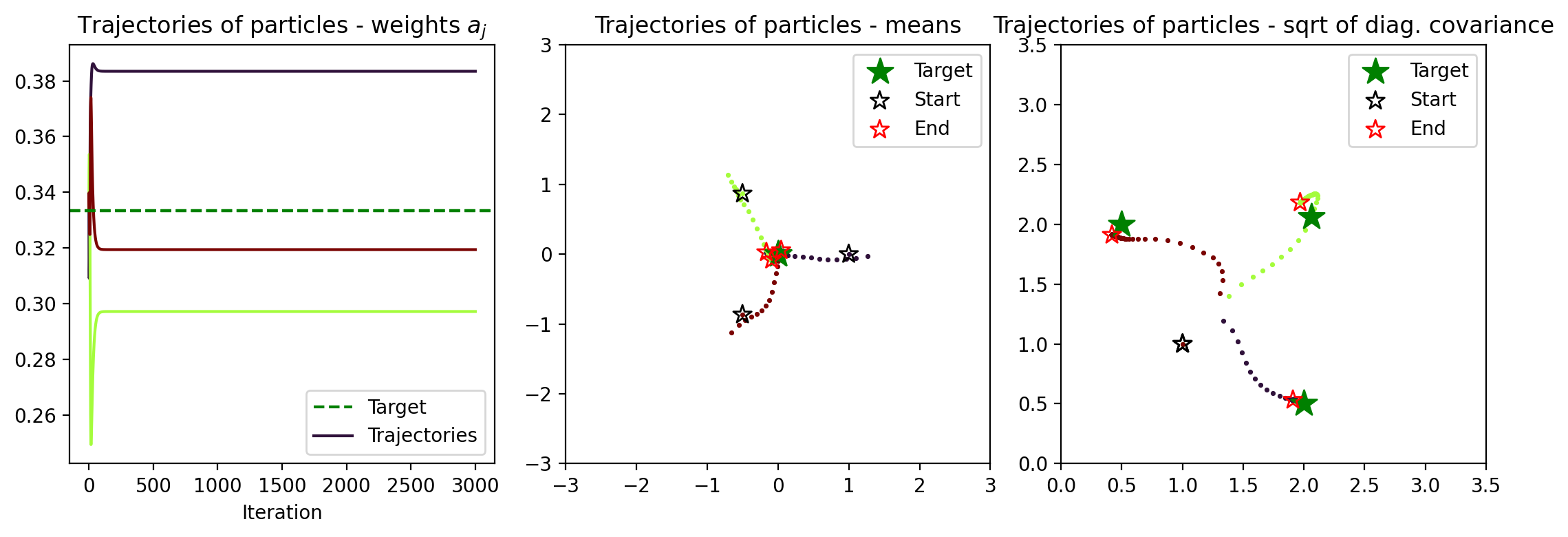}
    \caption{Convergence of the diagonal EM when the target measure is defined in \eqref{eq:cross_measure}. From left to right: evolution of the weights in terms of the iteration number, evolution of the mean along the iteration, evolution of the variance along the iteration.
    }
    \label{fig:comparison_em_iterates}
\end{figure}

\begin{figure}[ht]
    \centering
    \includegraphics[width=\linewidth]{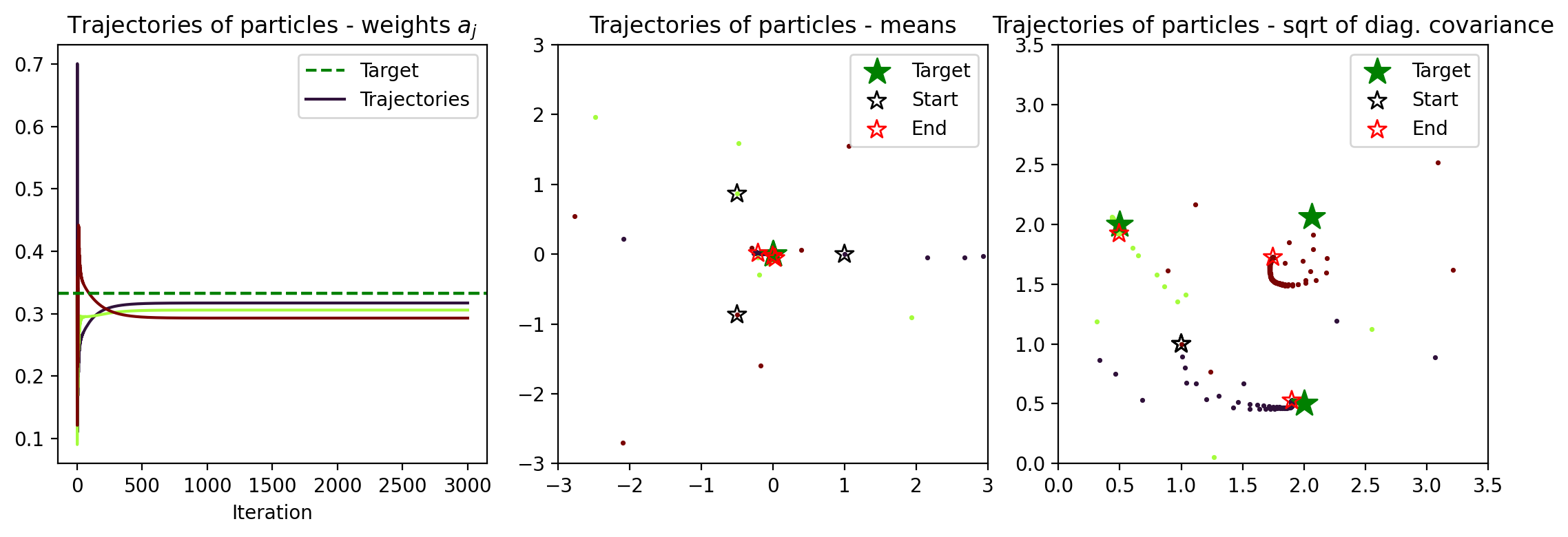}
    \caption{Convergence of Algorithm \ref{algo:CPGD_ngd_adagrad} when the target measure is defined in \eqref{eq:cross_measure}. From left to right: evolution of the weights in terms of the iteration number, evolution of the mean along the iteration, evolution of the variance along the iteration.
    }
    \label{fig:comparison_cpgd_iterates}
\end{figure}

\paragraph{Influence of overparameterization}
To extend these investigations, we considered a similar analysis with several values for the particle number $p$. For the sake of readability, we focus our attention on the evolution of the metrics introduced in Section \ref{section:numerical_scores} in terms of $p$. Corresponding results are displayed in Figure \ref{fig:comparison_scores_w1_number_particles} (Wasserstein) and Figure \ref{fig:comparison_scores_kl_number_particles} (KL). As $p$ increases, Algorithm \ref{algo:CPGD_ngd_adagrad} exhibits stable and consistent behavior, whereas the performance of EM deteriorates. 

The EM algorithm aims to maximize the log-likelihood associated with the observed sample for a mixture model with a prescribed number of $p$ components. It is therefore particularly sensitive to misspecification of the number of components. This is a well-known issue that can be addressed by introducing an additional model-selection step, in which the log-likelihood is penalized by a function of $p$ designed to promote sparsity. We refer, among others, to \citep{maugis2011} for a description and analysis of such a strategy.

In contrast, the CPGD principle offers greater flexibility in the choice of the initial number of components $p$. It is based on the optimization problem \eqref{eq:pb_BLASSO_gaussian_kernel_norm} , which includes a regularization term (the total variation norm) that promotes sparsity of the solution. In the experiments reported here, the accuracy of Algorithm \ref{algo:CPGD_ngd_adagrad} is therefore largely insensitive to the number of particles, provided that $p$ exceeds the number of target components. We stress that this statement concerns the accuracy of the reconstructed measure: the algorithm does not by itself return exactly $s$ particles, and the surviving particles may remain split across a single target atom, as observed in Section \ref{section:numerical_experiments}.

\begin{figure}[ht]
    \centering
    \includegraphics[width=\linewidth]{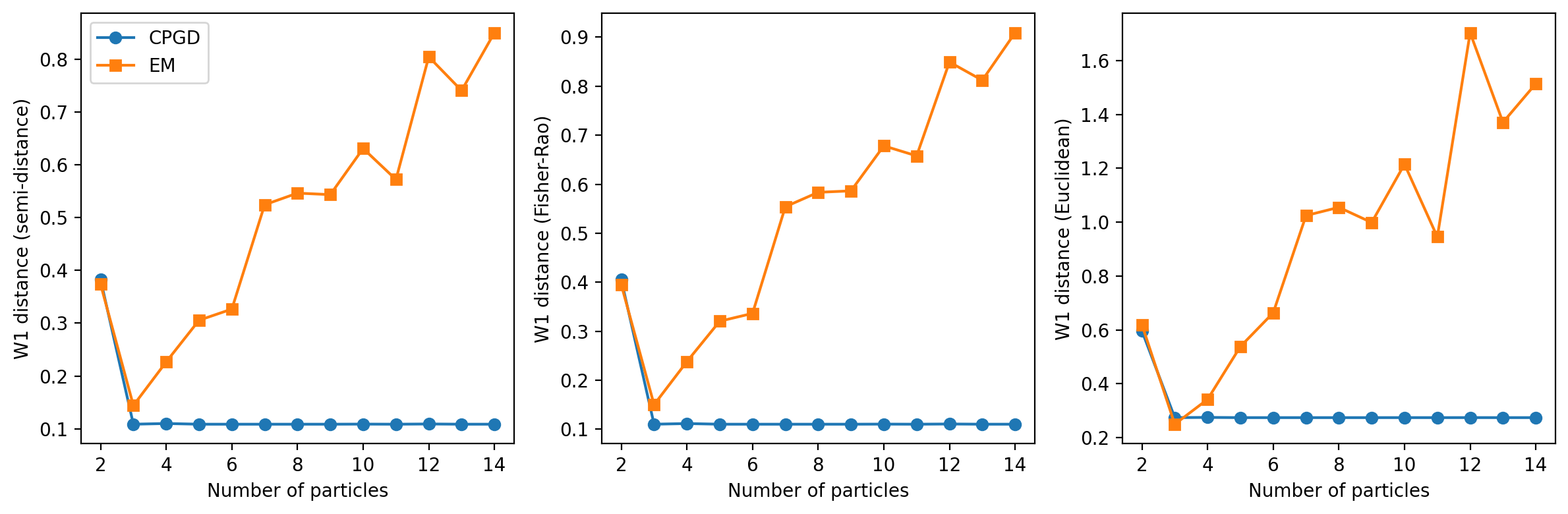}
    \caption{Wasserstein-type scores for EM and Algorithm \ref{algo:CPGD_ngd_adagrad} expressed in terms of $p$.
    }
   \label{fig:comparison_scores_w1_number_particles}
\end{figure}

\begin{figure}[ht]
    \centering
    \includegraphics[width=0.4\linewidth]{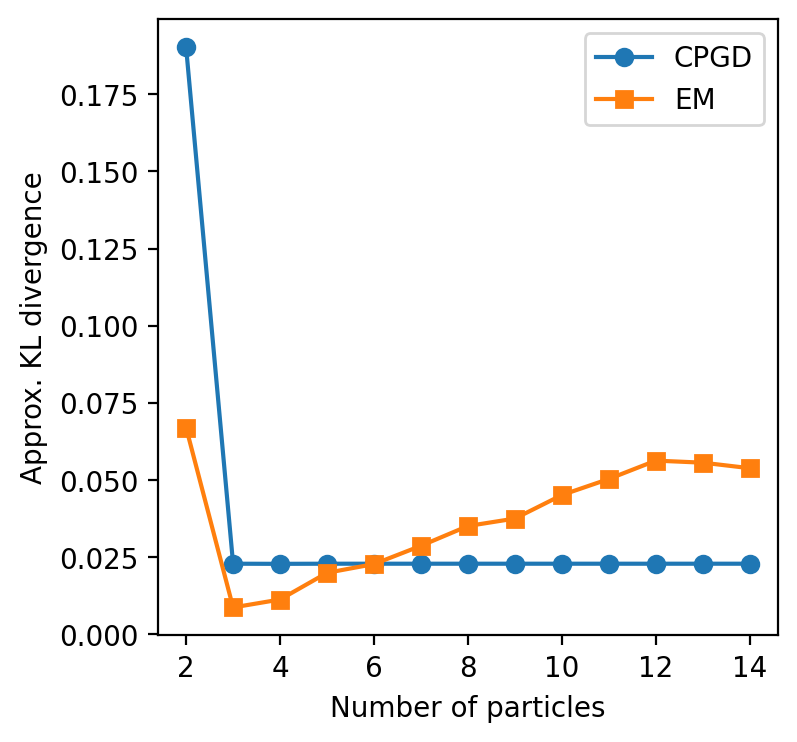}
    \caption{KL divergence for EM and Algorithm \ref{algo:CPGD_ngd_adagrad} expressed in terms  of $p$.
    }
   \label{fig:comparison_scores_kl_number_particles}
\end{figure}

\subsection{Influence of the separation and sample size}
To conclude these numerical investigations, we illustrate the impact of the separation between target components and the sample size on the performance of Algorithm \ref{algo:CPGD_ngd_adagrad}.

\paragraph{Experimental design}
We work in dimension $d=1$. The target is
\[\mu^0 = \frac{1}{2}\delta_{(-1,u^0)}+ \frac{1}{2}\delta_{(1,u^0)}\] with $u^0>0$.
We consider different samples of size $n$ drawn from $\Phi \mu^0$.

We study the convergence of CPGD for various values of $u^0$ and $n$. As $u^0$ gets larger, the separation between components gets smaller (measured by $\mathfrak{d}_\mathfrak{g}$ or $\mathpzc{d}$).

For each pair $(u^0,n)$, Algorithm \ref{algo:CPGD_ngd_adagrad} is run with 10 independently generated samples, and three different initializations. For each sample, we retain the solution achieving the smallest final loss among the initializations. The results are then aggregated over the samples, using the median performance. We get 3 aggregated Wasserstein-type scores depending on the ground distance used in $S$. For each configuration $(u^0,n,d)$ (with $d$ the ground distance), we then select the hyperparameters $\tau,\kappa$ that yield the best (aggregated) score (this choice depend on the score used, via the ground distance). The parameters $\tau$ and $\kappa$ are searched over a grid: $\tau$ ranges from $10^{-5}$ to $100$, and $\kappa$ from $10^{-14}$ to $0.1$. Since this selection uses the score itself, which depends on the ground truth $\mu^0$, the results reported below are an \emph{oracle} envelope. They describe the best accuracy attainable by the method over the hyperparameter grid, not the accuracy a practitioner would obtain from a data-driven choice of $(\tau,\kappa)$ (for instance selecting hyperparameters using an approximation of the KL divergence on a test sample).

\paragraph{Results}
The selected hyperparameters vary significantly with $n$ and $u^0$, and this dependence does not follow a clear or predictable pattern. Still, larger values of $u^0$ tend to correspond to larger $\tau$.

Moreover, the recovery scores themselves exhibit substantial variability across different initializations and sampled observations, indicating that CPGD lacks strong stability in this setting.

Despite this variability, a meaningful trend emerges from the experiment. We report the best median recovery scores, selected over the different initializations and choices of $\tau$ and $\kappa$ (Figure \ref{fig:separation_phase_transition}). 

\begin{figure}[ht]
    \centering
    \includegraphics[width=\linewidth]{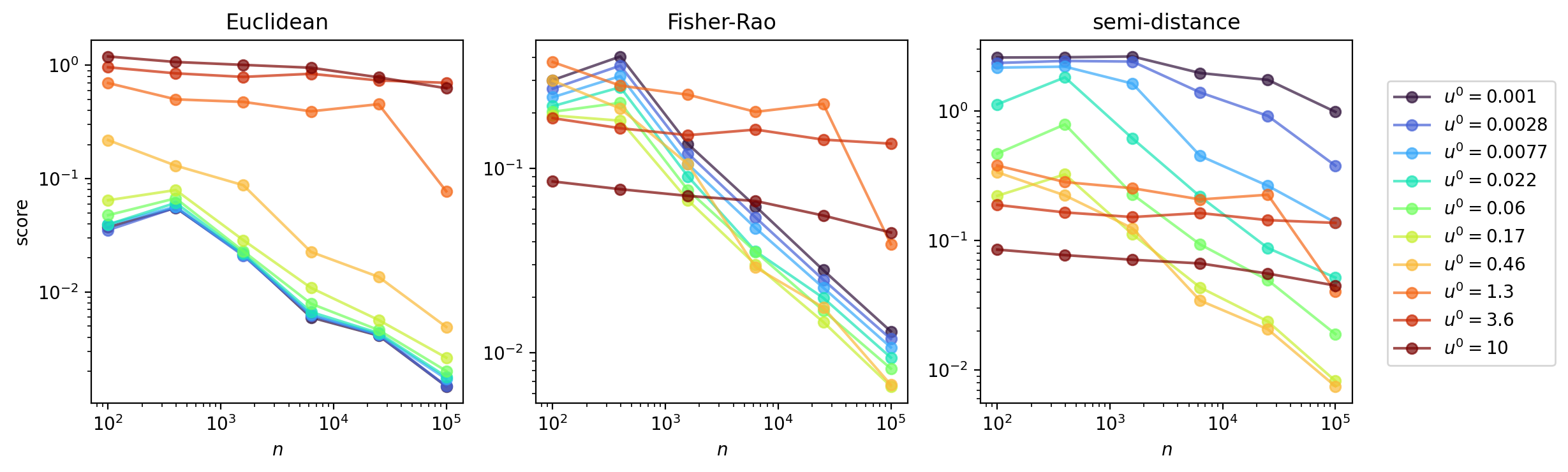}
    \caption{Recovery scores for Algorithm \ref{algo:CPGD_ngd_adagrad} in dimension $d=1$, for the two-component target $\mu^0=\frac{1}{2}\delta_{(-1,u^0)}+\frac{1}{2}\delta_{(1,u^0)}$, as a function of the sample size $n$ (horizontal axis, logarithmic) and of the scale parameter $u^0$ (legend). From left to right: the Wasserstein-type scores $S^e$ (Euclidean ground distance), $S^{\mathfrak{d}_\mathfrak{g}}$ (Fisher-Rao distance) and $S^{\mathpzc{d}}$ (semi-distance) introduced in Section \ref{section:numerical_scores}; lower is better. Each point is the median, over $10$ independently generated samples, of the best score over three initializations, minimized over the grid of hyperparameters $(\tau,\kappa)$. As discussed above, this selection of $(\tau,\kappa)$ uses the score itself, so that the reported values form an oracle envelope. The vertical scales differ across the three panels.
    }
    \label{fig:separation_phase_transition}
\end{figure}

A phase transition phenomenon becomes apparent in the recovery behavior, and it is visible in all three panels of Figure \ref{fig:separation_phase_transition}. The quantity to follow is the \emph{slope} of the curves rather than their level. For the smaller values of $u^0$, that is when the components are well separated relative to their own width, the score decreases steadily with $n$, as expected from a consistent procedure: larger samples improve recovery accuracy.
 
In contrast, beyond a critical scale located around $u^0 \approx 1$—precisely the regime in which the standard deviation of each component becomes comparable to the half-distance between the two means—the curves flatten: CPGD no longer recovers the target measure accurately, and increasing the sample size by three orders of magnitude yields only marginal improvement. The location of this transition is consistent across the Euclidean, Fisher-Rao and semi-distance scores, which suggests that it reflects a genuine difficulty of the recovery problem in this regime rather than an artifact of a particular discrepancy measure.

When interpreting these results, it is important to keep in mind that both the semi-distance and the Fisher–Rao distance tend to yield larger values in regimes characterized by small variances. This effect should be taken into consideration when comparing scores across different configurations. It accounts for the reversal of curve rankings between the Euclidean distance plots and the Fisher–Rao (as well as semi-distance) plots.

\section{Discussion and open problems}\label{section:discussion}
The CPGD algorithm has been studied in a general setting by \citep{cpgd_chizat}, and with stochastic gradient methods in \citep{fastpart}. Its specific
application to the Gaussian mixture problem with unknown covariance matrices has been overlooked. Three aspects of this contribution deserve emphasis.

From a methodological standpoint, unknown diagonal covariances endow the parameter space with a location-scale Fisher-Rao geometry, which we exploit directly in the position update, beyond its role in the analysis. The resulting Riemannian scheme \eqref{eq:RGD} is invariant under any invertible smooth reparameterization of the parameter space, a property that the Euclidean update \eqref{eq:EGD} does not share, and one that becomes meaningful precisely because the metric is no longer constant over $\X$.

On the theoretical side, we establish exponential local convergence with explicit constants for the joint weight-and-position dynamics. The non-degeneracy condition underlying this rate bears on the BLASSO solution, which is not known in advance; Theorem \ref{th:check_assumptions_solution} connects it to an interpretable separation condition on the statistical target, allowing the assumption to be read in terms of the model rather than the optimizer, with an explicit dependence on the regularization parameter $\kappa$.

Numerically, on the test cases considered here, our approach performs on a par with the well-established EM algorithm, with an advantage in robustness to an over-specified number of components and a disadvantage on the Kullback-Leibler score when the number of components is known. The accuracy of the method proves largely insensitive to overparameterization, and the experiments of Section \ref{section:numerical_experiments} exhibit a phase transition governed by the separation between components, consistently across the three ground distances considered.

\medskip
We now turn to the limitations of the approach and to the questions they leave open. We successively discuss the convergence guarantees and their local nature, the estimation of the scale parameters, the gap between the algorithm covered by our analysis and the one used in practice, and the evaluation of the optimality gap along the iterations.

\paragraph{Convergence issues} The discretization of the space of measures destroys convexity, which prevents us from establishing global convergence guarantees. At present, we can only prove the exponential convergence of CPGD when the initialization is close to the solution, provided strong assumptions on the solution (Theorem \ref{th:exponential_cv}). This convexity issue could be bypassed by introducing a spawn-and-prune step to our algorithm (see, \eg, \citep{decastro2026fsp}).

\paragraph{Scale parameter estimation} Estimating the scale parameters of the components requires handling projection strategies. When using natural gradient descent, one can exploit the symmetry of $J_W$ with respect to $u_j$ and apply the simple transformation $u_j \mapsto |u_j|$. 
Alternatively, to enforce the nonnegativity of $u_j$, one may take advantage of the invariance of the Fisher-Rao metric under reparameterization. For instance, we could consider the parameterization $\ln u$, or $\ln(u-u_{\min})$, and use the associated natural gradient descent. The related covariance updates are then \[u_j^{k+1} = (u_j^k-u_{\min})\exp\left( -\beta \frac{(2(u_j^k)^2 + \tau^2)^2 }{(u_j^k - u_{\min}) 2(u_j^k)^2 }\partial_{u_j} J'(x_j^k)\right)+u_{\min}\,.\] 
This issue becomes more challenging when using Riemannian gradient descent, as the exponential map is not globally defined. Freezing particles (see Lemma \ref{lemma:descent_property}) provides a theoretical workaround, but it is not easily implementable in practice.

\paragraph{Optimality gap} A natural direction for further investigation is to exploit the KKT conditions associated with the first variation $J'$ (Section \ref{section:optimization_procedure}).
Formally, let $\mu_\omega^k$ denote an approximation of the minimizer $\mu_\omega^\star$. The function $J_{\mu_\omega^k}'$ can then be used to derive an upper bound on the optimality gap $J_W(\mu_\omega^k)-J_W(\mu_\omega^\star)$, which is generally sharper than the trivial bound $J_W(\mu_\omega^k)$ obtained from the nonnegativity of $J_W$.
More precisely, assume that $\alpha>0$ verifies \begin{equation} \label{eq:condition_alpha_Jp}
\alpha J_{\mu_\omega^k}'(x) \geq (\alpha-1) \kappa \quad \forall \,x\in \X\,.\end{equation} Setting $h=\alpha(\Psi \mu_\omega^k -y)$, we get
\begin{equation}\label{eq:bound_optimality_gap}
    J_W(\mu_\omega^k)-J_W(\mu_\omega^\star) \leq J_W(\mu_\omega^k)+\frac{1}{2}\norm{h}_\L^2 + \innerprod{h}{y}_\L  \,.
\end{equation} We refer to Appendix \ref{section:proof_optimality_gap} for the proof. Observe that the case $\alpha=1$ corresponds exactly to the KKT feasibility condition $J_{\mu_\omega^k}'\geq 0$.

In practice, to determine $\alpha$ satisfying the condition displayed in \eqref{eq:condition_alpha_Jp}, one may compute either an approximation $\ubar{J'}$ of $\inf_{x \in \X} J_{\mu_\omega^k}'(x)$, or a rigorous lower bound, for instance $-\norm{\Psi \mu_\omega^k - y}_\L + \kappa$. 
A valid choice is then $\alpha = \frac{\kappa}{-\min \{\ubar{J'}\,, \,0 \}+\kappa}$, which is the \emph{largest} value of $\alpha$ ensuring the above condition. Indeed, \eqref{eq:condition_alpha_Jp} bounds $\alpha$ from above, while \eqref{eq:bound_optimality_gap} degenerates to the trivial bound $J_W(\mu_\omega^k)$ when $\alpha=0$: the bound is therefore all the more informative as $\alpha$ is large. When no estimate $\ubar{J'}$ is available, the rigorous fallback $\alpha= \frac{\kappa}{\max\{ \norm{\Psi \mu_\omega^k - y}_\L \,, \, \kappa \}}$ of Appendix \ref{section:proof_optimality_gap} is always admissible.

This construction is an instance of a general duality-gap mechanism, which clarifies both its scope and its limitations. Writing the objective as $F(\Psi\mu)+\kappa R(\mu)$ with $F=\frac{1}{2}\norm{\dotp-y}_\L^2$ and $R=\norm{\dotp}_{\rm TV}$, \citep{aubin_decastro_fygap} establishes the exact identity
\[
\Delta(\mu,h)=L_F(\Psi\mu \parallel h)+\kappa\, L_R(\mu \parallel \eta)\,, \qquad \eta=-\Psi^\star h/\kappa\,,
\]
valid for \emph{any} primal $\mu$ and \emph{any} dual $h$, in which $L_F$ and $L_R$ are the Fenchel-Young losses of $F$ and $R$. Evaluated at a dual-feasible $h$, the left-hand side is computable without any knowledge of $\mu_\omega^\star$ and upper bounds the suboptimality of $\mu$: the bound \eqref{eq:bound_optimality_gap} is exactly $\Delta(\mu_\omega^k,h)$ for the choice $h=\alpha(\Psi\mu_\omega^k-y)$, and condition \eqref{eq:condition_alpha_Jp} is precisely the requirement that this $h$ be dual-feasible.

This reading explains the role played by $\alpha$. For the quadratic fidelity $F$, the data-fidelity loss $L_F(\Psi\mu\parallel h)$ vanishes exactly at $h=\nabla F(\Psi\mu)=\Psi\mu-y$, that is, at $\alpha=1$; the whole gap then reduces to the regularizer term. Shrinking $\alpha$ below $1$ is what restores dual feasibility when $J_{\mu_\omega^k}'$ takes negative values, but it moves $h$ away from that alignment and inflates $L_F$. The bound is therefore informative only when $\alpha$ can be taken close to $1$, which requires $\inf_\X J_{\mu_\omega^k}'$ to be close to nonnegative. Computing an informative lower bound $\ubar{J'}$ is, in this respect, the practical bottleneck.

Preliminary experiments are consistent with this analysis. On the setting of Section \ref{section:numerical_experiments}, the bound \eqref{eq:bound_optimality_gap} evaluated at the output of Algorithm \ref{algo:CPGD_ngd_adagrad} remains close to the trivial bound $J_W(\mu_\omega^K)$, and becomes informative only when the components are well separated and $\kappa$ is taken large enough—a regime that in turn requires more particles at initialization. Larger values of $\kappa$ also relax \eqref{eq:condition_alpha_Jp} and thus allow $\alpha$ closer to $1$, so this improvement is not merely an artifact of the term $\kappa\norm{\mu}_{\rm TV}$ in the objective. We do not report these experiments here.

A comprehensive study of this duality gap, including the construction of certified dual-feasible points and the resulting early-stopping rules, is carried out in \citep{aubin_decastro_fygap}. In particular, a constructive form of the Br\o ndsted-Rockafellar theorem turns an approximately feasible pair into a dual-feasible one at controlled distance, which removes the need to compute $\ubar{J'}$ exactly and makes the stopping criterion certified. The Beurling-LASSO with quadratic fidelity considered here is a special case of the Generalized Beurling-LASSO studied therein, so the questions raised in this section fall squarely within the scope of that work.

\newpage
\printbibliography

\newpage

\appendix
{
\begin{center}
\LARGE
    Riemannian Gradient Descent for Gaussian Mixture Models with unknown diagonal covariances\\ --- \\ Appendix
\end{center}
\bigskip
{\ }
}
\section{Technical lemmas and definitions}\label{section:technical_lemmas_and_def}
\subsection{Geodesics, distances}\label{section:geodesics_and_distances}
\paragraph{Compatibility with the semi-distance}
For $A \subset \R^d \times [u_{\min},+\infty)^d$ and $x_0\in A$, we define the set containing all points from geodesics connecting $x_0$ with points of $A$, \[\mathcal{G}_{x_0}(A)\coloneq \{\gamma(y) \: : \: \gamma \text{ is a Fisher-Rao geodesic}\,, \: y\in [0,1]\,, \: \gamma(0)=x_0\, ,\gamma(1)\in A\}\,.\]
For $R>0$ and $x \in \R^d \times [u_{\min},+\infty)^d$, we write 
\[B_{\mathpzc{d}}(x,R) = \{x'\in\R^d \times [u_{\min},+\infty)^d \: : \: \mathpzc{d}(x,x') \leq R\}\,.\]
From \citep[Lemma 5.2]{article_giard_decastro_marteau}, we know that for $r>0$ and $x_0 \in \R^d \times [u_{\min},+\infty)^d$, it holds that $\mathcal{G}_{x_0}(B_{\mathpzc{d}}(x_0,r)) \subset B_{\mathpzc{d}}(x_0,r)$ (the semi-distance is increasing on geodesics, \ie $y \in [0,1] \mapsto \mathpzc{d}(x_0,\gamma_{x_0 \to x}(y))$ is increasing for any $x \in \R^d \times [u_{\min},+\infty)^d$).

\paragraph{Fisher-Rao distance}
\citep[Lemmas H.2 and H.3]{article_giard_decastro_marteau} give a closed-form expression for $\mathfrak{d}_\mathfrak{g}$. This expression depends on $\tau$.

\subsection{Riemannian derivatives}
For the sake of convenience, we reproduce definitions and properties from \citep{article_giard_decastro_marteau}.

The non-zero Christoffel symbols associated with $\mathfrak{g}$ are
\begin{align*}
 \Gamma^{t_k}{}_{u_k t_k}=\Gamma^{t_k}{}_{t_k u_k}&=\frac{-2u_k}{2u_k^2+\tau^2}\,,\\
 \Gamma^{u_k}{}_{t_k t_k}&=\frac{1}{u_k} \,,\\
 \Gamma^{u_k}{}_{u_k u_k}&=\frac{\tau^2-2u_k^2}{u_k(2u_k^2+\tau^2)} 
\end{align*}
with $k= 1, \ldots,d$.
We define \begin{equation}\label{eq:christoffel_symbols}\Gamma^{t_k}=\begin{pmatrix}
 (\Gamma^{t_k}{}_{t_l t_m})_{1\leq l,m\leq d} & (\Gamma^{t_k}{}_{t_l u_m})_{1\leq l,m\leq d}
 \\ (\Gamma^{t_k}{}_{u_l t_m})_{1\leq l,m\leq d} & (\Gamma^{t_k}{}_{u_l u_m})_{1\leq l,m\leq d}
\end{pmatrix} \quad \text{and} \quad \Gamma^{u_k}=\begin{pmatrix}
 (\Gamma^{u_k}{}_{t_l t_m})_{1\leq l,m\leq d} & (\Gamma^{u_k}{}_{t_l u_m})_{1\leq l,m\leq d}
 \\ (\Gamma^{u_k}{}_{u_l t_m})_{1\leq l,m\leq d} & (\Gamma^{u_k}{}_{u_l u_m})_{1\leq l,m\leq d}
\end{pmatrix} \,.\end{equation}

\begin{definition}[Covariant derivatives, associated norms]\label{def:norm_covariant_derivatives}
Let $\psi \in \C^2(\R^d \times [u_{\min}, +\infty)^d)$. Let $x,x' \in \R^d \times [u_{\min}, +\infty)^d$.
\\\underline{Covariant derivatives:} Let $v,v' \in \R^{2d}$.
We define \[D_0[\psi](x)=\psi(x)\,, \quad D_1[\psi](x)[v]=v^T \nabla \psi(x)=\innerprod{v}{\nabla^{\mathfrak{g}} \psi(x)}_x\,, \quad D_2[\psi](x)[v,v'] = v^T H^{\mathfrak{g}} \psi(x) v'\,.\]
\underline{Operator norms:} For $j\in\{0,1,2\}$, we define the operator norm \[\norm{D_j[\psi](x)}_x \coloneq \sup_{\substack{V=[v_1,\ldots,v_j]\in (\R^{2d})^j \\ \forall l=1,\ldots,j\,, \: \norm{v_l}_x \leq 1}} D_j[\psi](x)[V]\,.\]
\underline{Kernel derivatives and associated operators:} Let $i,j\in\{0,1,2\}$.
We define the covariant derivative of the kernel of order $i$ with respect to the first variable $x$ and of order $j$ with respect to the second variable $x'$ by \[[Q]K_{\rm norm}^{(ij)}(x,x')[V]=\innerprod{D_i[\Psi](x)[Q]}{D_j[\Psi](x')[V]}_\L \quad \forall\, Q \in (\R^{2d})^i\,, V \in (\R^{2d})^j\,.\]
The associated operator norm is \[\norm{K_{\rm norm}^{(ij)}(x,x')}_{x,x'}\coloneq \sup_{\substack{Q=[Q_1,\ldots,Q_i]\in (\R^{2d})^i\,,\,V=[V_1,\ldots,V_j]\in (\R^{2d})^j\\ \forall l\, \norm{Q_l}_x\,, \, \norm{V_l}_{x'} \leq 1}} [Q]K_{\rm norm}^{(ij)}(x,x')[V] \,.\]
\end{definition}
Simplified expressions for the operator norms are given in \citep[Lemma J.1]{article_giard_decastro_marteau}. Note in particular that $\norm{\nabla^\mathfrak{g} \psi(x)}_{x} = \norm{\mathfrak{g}_{x}^{-1/2} \nabla \psi(x)}_2$.

We use global controls on the covariant derivatives of the kernel, and also on ``euclidean-flavored'' counterparts, as well as on the local curvature.

In the following lemma and its proof, we adopt the notation introduced in the proof of \citep[Lemma J.2]{article_giard_decastro_marteau}. In particular, we use the decomposition of the kernel as a product:
$K_{\rm norm} (x,x')=\prod_{k=1}^d K_{\rm norm} (x_k,x_k')$ where, by abuse of notation, $K_{\rm norm} (x_k,x_k')$ denotes the one-dimensional kernel evaluated at $x_k=(t_k,u_k)$ and $x_k'=(t_k',u_k')$. The notation $b_k$ stands for either $t_k$ or $u_k$, and $\mathfrak{g}_{b_k b_k} = \partial{b_k} \partial_{b_k'} K_{\rm norm}(x_k, x_k)$ denotes a diagonal entry of the metric tensor.

Note that any derivative of $K_{\rm norm}$ belongs to $\L$, and that for all $x, x' \in \R^d \times [u_{\min}, +\infty)^d$, one has (see \citep[Remark 5.1]{article_giard_decastro_marteau}) $\innerprod{\partial_1 \Psi \delta_x}{\partial_2 \Psi \delta_{x'}}_\L = \partial_1 \partial_2 K_{\rm norm}(x,x')$, where $\partial_1$ (resp.\ $\partial_2$) denotes differentiation with respect to $x$ (resp.\ $x'$).

Finally, the notation $H^\mathfrak{g}\partial_{b_k} \Psi \delta_x$ stands for $H^\mathfrak{g}\left(x \mapsto \partial_{b_k} \Psi \delta_x\right)$.
where $\partial_1$ (resp.\ $\partial_2$) denotes here any derivative \wrt $x$ (resp.\ $x'$).
\begin{lemma}[Controls on the kernel]\label{lemma:controls_kernel} Assume that $\tau \leq u_{\min}$.
\\\underline{Global controls on the covariant derivatives:}
For $i,j \in \{0,1,2\}$, we have \[\underset{x,x' \in \R^d\times [u_{\min},+\infty)^d}{\sup} \norm{K_{\rm norm}^{(ij)}(x,x')}_{x,x'} \leq B_{ij}\]
 with
 \begin{align*}
 B_{00}&=1 \,,\\
 B_{10}=B_{01}&= \sqrt{2d} \,, \\
 B_{11}&=2d\,, \\
 B_{02}=B_{20}&=\sqrt{4d^2+10d} \,, \\
 B_{12}=B_{21}&= \sqrt{2d}B_{02}\,, \\
 B_{22}&=28d^2 \,.
 \end{align*}
\underline{Global controls on the ``euclidean-flavored'' derivatives:}
 We also have \[\sup_{x,x'\in \R^d \times [u_{\min},+\infty)^d]} \norm{\mathfrak{g}_x^{-1/2} \nabla_1^2  K_{\rm norm} (x,x') \mathfrak{g}_{x}^{-1/2}}_2 \leq 6d \eqcolon B_{02}^e \] and
 \[2d \sup_{\substack{x\in \R^d \times [u_{\min},+\infty)^d, \\ b_k \in \{t_k,u_k\}}}\sup_{\norm{v_1}_2\leq 1, \norm{v_2}_2\leq 1} \norm{v_1^T \mathfrak{g}_{b_k b_k}^{-1/2} \mathfrak{g}_{x}^{-1/2} H^\mathfrak{g}\partial_{b_k} \Psi \delta_x \mathfrak{g}_{x}^{-1/2} v_2}_\L^2 \leq 1554 d^3 \eqcolon B_{33}^e\,. \]
 \underline{Local curvature:} Let $x,x' \in \R^d \times [u_{\min},+\infty)^d$. If $\mathpzc{d}(x,x')\leq \frac{0.3025}{\sqrt{d}}$, then \[-K^{(20)}(x,x')[v,v] \geq \bar \varepsilon_2 \norm{v}_{x}^2\quad \forall \: v\in \R^{2d}\,, \quad \text{with } \bar\varepsilon_2=0.13139\,. \]
\end{lemma}
\begin{proof}
    The bounds $B_{00}, B_{10},  B_{11},  B_{02},  B_{12}$ come from the global controls of the covariant derivatives given in \citep[Lemma J.2]{article_giard_decastro_marteau}. The bound $B_{22}$ can be found in the proof of \citep[Lemma K.1]{article_giard_decastro_marteau}.
To derive the global bounds $B_{02}^e$ and $B_{33}^e$, we follow the strategy used in the proof of \citep[Lemma~J.2]{article_giard_decastro_marteau}.
\\\underline{$B_{02}^e$:} We refer to \texttt{cpgd\_sympy.ipynb} in \citep{repo_zenodo_cpgd} for the formal computations used to derive the bounds presented here. 

Using that the metric tensor $\mathfrak{g}_x$ is diagonal, we have
    \begin{align*}
        \sup_{x,x'\in \R^d \times [u_{\min},+\infty)^d]} \norm{\mathfrak{g}_x^{-1/2} \nabla_1^2  K_{\rm norm} (x,x') \mathfrak{g}_{x'}^{-1/2}}_2 &\leq 2d \sup_{b_k,b_l',x,x'\in \R^d \times [u_{\min},+\infty)^d]} \mathfrak{g}_{b_k b_k}^{-1/2} \mathfrak{g}_{b_l b_l}^{-1/2}|\partial_{b_k b_l}  K_{\rm norm} (x,x') |\,.
    \end{align*}
If $b_k\in \{t_k,u_k\}$ and $b_l\in \{t_l,u_l\}$ with $k\neq l$ (\ie the derivatives are not taken in the same dimension), then 
\begin{align*}
\mathfrak{g}_{b_k b_k}^{-1/2} \mathfrak{g}_{b_l b_l}^{-1/2}|\partial_{b_k b_l}  K_{\rm norm} (x,x') | &\leq \mathfrak{g}_{b_k b_k}^{-1/2} |\partial_{b_k}  K_{\rm norm} (x_k,x_k) | \mathfrak{g}_{b_l b_l}^{-1/2}|\partial_{b_l}  K_{\rm norm} (x_l,x_l) | \,, \\
&\leq 1
\end{align*}   
as shown in the proof of \citep[Lemma J.2]{article_giard_decastro_marteau} when deriving $B_{10}$. If $b_k=b_l=t_k$, 
\begin{align*}
\mathfrak{g}_{t_k t_k}^{-1}|\partial_{t_k t_k}  K_{\rm norm} (x,x') | &\leq \mathfrak{g}_{t_k t_k}^{-1}\sqrt{\partial_{t_k t_k}\partial_{t_k' t_k'}  K_{\rm norm} (x_k,x_k)} \,, \\
&= \sqrt{3} \,.
\end{align*}
With the same approach, we obtain \[\mathfrak{g}_{t_k t_k}^{-1/2} \mathfrak{g}_{u_k u_k}^{-1/2}|\partial_{t_k u_k}  K_{\rm norm} (x,x') | \leq \sqrt{5}\]
and \[\mathfrak{g}_{u_k u_k}^{-1}|\partial_{u_k u_k}  K_{\rm norm} (x,x')| = \sqrt{9 - \frac{2\tau^2}{u_k^2} + \frac{\tau^4}{2u_k^4}} \leq 3\,,\] where we used that $\tau \leq u_k$.
So we can take $B_{02}^e = 6d$.
\\\underline{$B_{33}^e$:} We have
\begin{align*}
  2d &\sup_{\substack{x\in \R^d \times [u_{\min},+\infty)^d, \\ b_k \in \{t_k,u_k\}}}\sup_{\norm{v_1}_2\leq 1, \norm{v_2}_2\leq 1} \norm{v_1^T \mathfrak{g}_{b_k b_k}^{-1/2} \mathfrak{g}_{x}^{-1/2} H^\mathfrak{g}\partial_{b_k} \Psi \delta_x \mathfrak{g}_{x}^{-1/2} v_2}_\L^2 \\
  &\leq (2d)^3 \sup_{x\in \R^d \times [u_{\min},+\infty)^d}\sup_{b_k, b_i, b_j}  \mathfrak{g}_{b_k b_k}^{-1} \mathfrak{g}_{b_i b_i}^{-1} \mathfrak{g}_{b_j b_j}^{-1} \norm{(H^{\mathfrak{g}}\partial_{b_k} \Psi \delta_x)_{b_i,b_j}}_\L^2 \,.
\end{align*}
If $b_k$ is not in the same dimension as either $b_i$ or $b_j$, 
\begin{align*}
   \mathfrak{g}_{b_k b_k}^{-1} \mathfrak{g}_{b_i b_i}^{-1} \mathfrak{g}_{b_j b_j}^{-1} \norm{(H^{\mathfrak{g}}\partial_{b_k} \Psi \delta_x)_{b_i,b_j}}_\L^2 &\leq \mathfrak{g}_{b_k b_k}^{-1} \partial_{b_k b_k'} K_{\rm norm}(x_k,x_k) \mathfrak{g}_{b_i b_i}^{-1} \mathfrak{g}_{b_j b_j}^{-1} \norm{(H^{\mathfrak{g}} \Psi \delta_x)_{b_i,b_j}}_\L^2 \,, \\
   & \leq 7
\end{align*}
where for the last inequality we used the controls established in the proofs of \citep[Lemma J.2, Lemma K.1]{article_giard_decastro_marteau} when deriving $B_{11}$ and $B_{22}$.

If $b_k$ is taken on the same dimension as $b_j$ or $b_i$, but $b_j$ and $b_i$ are not on the same dimension (say that for instance $b_k$ is on the same dimension as $b_i$), 
\begin{align*}
  \mathfrak{g}_{b_k b_k}^{-1} \mathfrak{g}_{b_i b_i}^{-1} \mathfrak{g}_{b_j b_j}^{-1} \norm{(H^{\mathfrak{g}}\partial_{b_k} \Psi \delta_x)_{b_i,b_j}}_\L^2 &= \mathfrak{g}_{b_k b_k}^{-1} \mathfrak{g}_{b_i b_i}^{-1} \mathfrak{g}_{b_j b_j}^{-1} \norm{\partial_{b_k b_i b_j} \Psi \delta_x}_\L^2 \,, \\
  & \leq  \mathfrak{g}_{b_j b_j}^{-1} \partial_{b_j b_j} K_{\rm norm}(x_j,x_j) \mathfrak{g}_{b_k b_k}^{-1} \mathfrak{g}_{b_i b_i}^{-1} \partial_{b_k b_i} \partial_{b_k' b_i'}K_{\rm norm}(x_k,x_k) \,, \\
  & \leq 9 \,,
\end{align*}
as shown when establishing $B_{02}^e$.

If $b_k,b_i,b_j$ are all taken on the same dimension, $\mathfrak{g}_{b_k b_k}^{-1} \mathfrak{g}_{b_i b_i}^{-1} \mathfrak{g}_{b_j b_j}^{-1} \norm{(H^{\mathfrak{g}}\partial_{b_k} \Psi \delta_x)_{b_i,b_j}}_\L^2$ can take the forms
\begin{align*} & \mathfrak{g}_{t_k t_k}^{-3}\norm{\partial_{t_k t_k t_k} \Psi \delta_x - \Gamma^{u_k}{}_{t_k t_k}\partial_{u_k t_k}\Psi \delta_x}_\L^2  
\\\text{or} \quad & \mathfrak{g}_{t_k t_k}^{-2}\mathfrak{g}_{u_k u_k}^{-1}\norm{\partial_{t_k t_k u_k} \Psi \delta_x - \Gamma^{t_k}{}_{t_k u_k}\partial_{t_k t_k} \Psi \delta_x}_\L^2
\\\text{or} \quad & \mathfrak{g}_{t_k t_k}^{-1}\mathfrak{g}_{u_k u_k}^{-2}\norm{\partial_{t_k u_k u_k} \Psi \delta_x - \Gamma^{t_k}{}_{t_k u_k}\partial_{t_k u_k} \Psi \delta_x}_\L^2
\\\text{or} \quad & \mathfrak{g}_{u_k u_k}^{-2}\mathfrak{g}_{t_k t_k}^{-1}\norm{\partial_{u_k u_k t_k} \Psi \delta_x - \Gamma^{u_k}{}_{u_k u_k}\partial_{u_k t_k}\Psi \delta_x}_\L^2
\\\text{or} \quad & \mathfrak{g}_{u_k u_k}^{-3}\norm{\partial_{u_k u_k u_k} \Psi \delta_x - \Gamma^{u_k}{}_{u_k u_k}\partial_{u_k u_k}\Psi \delta_x}_\L^2\,.
\end{align*}
These expressions are bounded respectively by $1$, $21$, $\frac{5\tau^4}{2u_k^4}-\frac{10\tau^2}{u_k^2}+55 \leq 55$, $55$, $\frac{\tau^8}{4u_k^8}+\frac{4\tau^6}{u_k^6}+\frac{35\tau^4}{u_k^4}-\frac{16\tau^2}{u_k^2}+171\leq 194.25$. We used that $\tau \leq u_{\min} \leq u_k$. 
Hence we can take $B_{33}^e=(2d)^3 \times 194.25$. 
\\\underline{Local curvature:} This control comes from \citep[Theorem 5.3]{article_giard_decastro_marteau}.
\end{proof}

\begin{lemma}[Taylor expansions on the kernel] \label{lemma:taylor_expansions_kernel}
    Let $x,x' \in \R^d \times [u_{\min},+\infty)^d$.
    We have \[\norm{\Psi(\delta_{x} - \delta_{x'})}_\L^2 \leq B_{02} \mathfrak{d}_\mathfrak{g}(x,x')^2\,.\]
\end{lemma}
\begin{proof}
  Remark that $K_{\rm norm}(x,x) = K_{\rm norm}(x',x') = 1$ and that $\nabla_2 K_{\rm norm}(x,x) = 0$. Writing $\gamma_{x \to x'}$ the Fisher-Rao geodesic such that $\gamma_{x \to x'}(0)=x$, $\gamma_{x \to x'}(1)=x'$, it comes that
  \begin{align*}
      \norm{\Psi(\delta_{x} - \delta_{x'})}_\L^2 &= 2K_{\rm norm}(x,x)-2K_{\rm norm}(x,x')\,, \\
      &= -2\dot \gamma_{x \to x'}(0)^T \nabla_2 K_{\rm norm}(x,x) - 2\int_0^1 (1-y)\dot \gamma_{x \to x'}(y)^T H_2^\mathfrak{g} K_{\rm norm}(x,\gamma_{x \to x'}(y)) \dot \gamma_{x \to x'}(y) \, \d y \,, \\
      &= -2\int_0^1 (1-y)\dot \gamma_{x \to x'}(y)^T H_2^\mathfrak{g} K_{\rm norm}(x,\gamma_{x \to x'}(y)) \dot \gamma_{x \to x'}(y) \, \d y \,.
  \end{align*}
  From the control on the Riemannian Hessian (Lemma \ref{lemma:controls_kernel}), and as $\gamma_{x\to x'}$ stays in $\R^d \times [u_{\min},+\infty)^d$ (Section \ref{section:geodesics_and_distances}), we have \[\left|\dot \gamma_{x \to x'}(y)^T H_2^\mathfrak{g} K_{\rm norm}(x,\gamma_{x \to x'}(y)) \dot \gamma_{x \to x'}(y) \right| \leq B_{02} \mathfrak{d}_\mathfrak{g}(x,x')^2\,.\]
  So $\norm{\Psi(\delta_{x} - \delta_{x'})}_\L^2 \leq B_{02} \mathfrak{d}_\mathfrak{g}(x,x')^2.$
\end{proof}

\subsection{Bound on the TV norm of the iterates, control of the Fréchet derivative}
\begin{lemma}[TV norm of the iterates]\label{lemma:control_tv_iterates}
Let $\X$ be a compact subset of $\R^d \times [u_{\min},+\infty)^d$ and $C_\X \coloneq \inf_{x,x' \in \X} K_{\rm norm}(x,x')$. Let $\mu_\omega^1=\sum_{j=1} \omega_j^1 \delta_{x_j^1}$ be the initialization of CPGD, with $\omega_j^1 \geq 0$ for $j=1,\ldots,p$. Let $\mu_\omega^k$ be the kth iterate of the algorithm, with weights updated via $\omega_j^{k+1} = \omega_j^k e^{-\alpha J_{\mu_\omega^k}'(x_j^k)}$, and location-scale parameters verifying $x_j^k \in \X$ for all $j=1,\ldots,p$ and $k\geq 1$.
Then $\mu_\omega^k$ is a nonnegative measure. Moreover, for all $k\geq 1$, \[\norm{\mu_\omega^k}_{\rm TV} \leq \max\left\{\frac{1}{C_{\X}}(\norm{y}_\L - \kappa) e^{\alpha(\norm{y}_\L-\kappa)} \,, \, \norm{\mu_\omega^1}_{\rm TV} \right\} \leq \max\left\{\frac{1}{C_{\X}}\norm{y}_\L e^{\alpha \norm{y}_\L} \,, \, \norm{\mu_\omega^1}_{\rm TV} \right\} \eqcolon B_{\rm TV} \,.\]
\end{lemma}
\begin{proof}
   Note that as our kernel is positive ($\innerprod{\Psi \delta_x}{\Psi \delta_{x'}}> 0$ for $x,x' \in \R^d \times (0,+\infty)^d$) and as $\X$ is compact, $C_{\X}>0$. As $\omega_j^k\geq 0$ for all $j=1,\ldots,p$, it comes that, for $x \in \X$, 
   \begin{equation}\label{eq:Jp_bound_CX}
       J_{\mu_\omega^k}'(x)=\innerprod{\Psi \delta_x}{\sum_{j=1}^p \omega_j^k \Psi \delta_{x_j^k} - y}_\L + \kappa \geq C_{\X}\norm{\mu_\omega^k}_{\rm TV} - \norm{y}_\L + \kappa\,.
   \end{equation}
\\\underline{Case $\norm{\mu_\omega^k}_{\rm TV} \geq \frac{\norm{y}_\L-\kappa}{C_{\X}}$:} From \eqref{eq:Jp_bound_CX}, as we assumed that $x_j^k\in\X$, it comes that for all $j=1,\ldots,p$, $J_{\mu_\omega^k}'(x_j^k)\geq 0$. Hence $\omega_j^{k+1}= \omega_j^k e^{-\alpha J_{\mu_\omega^k}'(x_j^k)} \leq \omega_j^k$ and $\norm{\mu_\omega^{k+1}}_{\rm TV}\leq \norm{\mu_\omega^k}_{\rm TV}$.
\\\underline{Case $\norm{\mu_\omega^k}_{\rm TV} \leq \frac{\norm{y}_\L-\kappa}{C_{\X}}$:} Then $\norm{y}_\L-\kappa\geq 0$ (because $\norm{\mu_\omega^k}_{\rm TV}\geq 0$) and from \eqref{eq:Jp_bound_CX} we get $-J_{\mu_\omega^k}'(x_j^k) \leq \norm{y}_\L-\kappa$, leading to
$\omega_j^{k+1}\leq \omega_j^k e^{\alpha (\norm{y}_\L-\kappa)}$. It comes that $\norm{\mu_\omega^{k+1}}_{\rm TV} \leq \frac{1}{C_\X}(\norm{y}_\L - \kappa) e^{\alpha(\norm{y}_\L-\kappa)}$. 

We showed that for all $k\geq 1$, \[\norm{\mu_\omega^k}_{\rm TV} \leq \max\left\{\frac{1}{C_\X}(\norm{y}_\L - \kappa) e^{\alpha(\norm{y}_\L-\kappa)} \,, \, \norm{\mu_\omega^1}_{\rm TV}\right\}\,.\]
\end{proof}

\begin{remark}[Controls on the TV norm of the iterates without $C_\X$]\label{remark:control_tv_iterates_wo_C_X}
Without exploiting the positivity of the kernel $K_{\rm norm}$, one can still derive a bound on $\norm{\mu_\omega^k}_{\rm TV}$, although this bound depends on $p$. Indeed, without using that $C_\X>0$, we have \[  J_{\mu_\omega^k}'(x) \geq \omega_j^k - \norm{y}_\L + \kappa \quad \forall\, x\in\X\,, \] from which, by distinguishing the cases $\omega_j^k \geq \norm{y}_\L - \kappa$ and $\omega_j^k \leq \norm{y}_\L - \kappa$, we can obtain \[\norm{\mu_\omega^k}_{\rm TV} \leq \max\left\{p(\norm{y}_\L - \kappa) e^{\alpha(\norm{y}_\L-\kappa)} \,, \, \norm{\mu_\omega^1}_{\rm TV} \right\} \,.\]    
\end{remark}

\begin{lemma}[Control of $J'$]\label{lemma:control_Jp}
For $\mu\in \M(\R^d \times [u_{\min},+\infty)^d)^+$ and $x \in \R^d \times [u_{\min},+ \infty)^d$,
\[-\norm{y}_{\L}+\kappa \leq J_\mu'(x) \leq  \norm{\mu}_{\rm TV} + \norm{y}_\L+ \kappa\,.\] Moreover, for $x,x' \in \R^d \times [u_{\min},+ \infty)^d$,
\[\norm{\nabla^{\mathfrak{g}} J_\mu'(x)}_x \leq \sqrt{B_{11}} (\norm{\mu}_{\rm TV}+\norm{y}_\L)\,,\] 
\[\norm{\mathfrak{g}_x^{-1/2}H^{\mathfrak{g}}J_\mu'(x) \mathfrak{g}_x^{-1/2}}_2 \leq \sqrt{B_{22}} (\norm{\mu}_{\rm TV}+\norm{y}_\L)\] 
and
\[
    |J_{\mu}'(x)-J_{\mu}'(x')|\leq \sqrt{B_{02}} \mathfrak{d}_\mathfrak{g}(x,x') (\norm{\mu}_{\rm TV}+\norm{y}_\L)\,.
\]

\end{lemma}
\begin{proof}
Let $\mu\in \M(\R^d \times [u_{\min},+\infty)^d)^+$. For $x=(t_1,\ldots,t_d,u_1,\ldots,u_d)\in \R^d \times [u_{\min},+\infty)^d$, again as our kernel is positive it comes that $J_\mu'(x) \geq -\norm{y}_\L + \kappa$.
Moreover, as $\norm{\Psi \delta_x}_{\L}\leq 1$, we have $J_{\mu}'(x) \leq \norm{\mu}_{\rm TV} + \norm{y}_\L+ \kappa$.
In addition (\cf Definition \ref{def:norm_covariant_derivatives}),
\begin{align*}
    \norm{\nabla^{\mathfrak{g}} J_\mu'(x)}_x& = \norm{\mathfrak{g}_x^{-1/2} \innerprod{\nabla_x\Psi \delta_{x}}{\Psi \mu -y}_\L}_2 \,, \\
    &\leq \sqrt{B_{11}} (\norm{\mu}_{\rm TV}+\norm{y}_\L)
\end{align*}
and
\begin{align*}
    \norm{\mathfrak{g}_x^{-1/2} H^{\mathfrak{g}} J_\mu'(x)\mathfrak{g}_x^{-1/2}}_2& =  \norm{\mathfrak{g}_x^{-1/2} \innerprod{H_x^{\mathfrak{g}} \Psi \delta_x}{\Psi \mu- y}_\L\mathfrak{g}_x^{-1/2}}_2 \,, \\
    &\leq \sqrt{B_{22}} (\norm{\mu}_{\rm TV}+\norm{y}_\L)\,.
\end{align*}

Finally, using Lemma \ref{lemma:taylor_expansions_kernel}, for $x,x' \in \R^d \times [u_{\min},+ \infty)^d$ we have
\begin{align*}
    |J_{\mu}'(x)-J_{\mu}'(x')|& = |\innerprod{\Psi (\delta_{x}-\delta_{x'})}{\Psi \mu -y}_\L| \,, \\
    &\leq \norm{\Psi (\delta_{x}-\delta_{x'})}_\L (\norm{\mu}_{\rm TV}+\norm{y}_\L)\,,\\
    & \leq   \sqrt{B_{02}} \mathfrak{d}_\mathfrak{g}(x,x') (\norm{\mu}_{\rm TV}+\norm{y}_\L)\,.
\end{align*}
\end{proof}

\section{Proof of Proposition \ref{prop:weight_convergence}}\label{section:proof_prop_weight_convergence}

Let $1\leq k \leq K$. Our condition on $\mathcal{T}$ ensures that $\mu_\omega^k \in\M(\X)^+$. We follow the proof of \citep[Theorem 1.2]{fastpart} (inspired by \citep[Section 4.2]{cpgd_chizat}).

In this proof, we denote $c_a$ a constant depending on $a$, where $a$ can stand for instance for $\norm{y}_\L, \norm{\mu^1}_{\rm TV}$. Its value may change between lines.

Remark that
 \[J_W(\mu_\omega^k) -  J_W(\mu_\omega^\star)=  \int_{\X} J_{\mu_\omega^k}' \, \d(\mu_\omega^k-\mu_\omega^\star) - \frac{1}{2}\norm{\Psi(\mu_\omega^k-\mu_\omega^\star)}_\L^2 \leq \int_{\X} J_{\mu_\omega^k}' \, \d(\mu_\omega^k-\mu_\omega^\star) \,.\]
 
 We introduce an auxiliary measure $\mu_\epsilon^k=\sum_{j=1}^p \omega_j^\epsilon \delta_{x_j^k}$, which shares the same position–scale parameters as $\mu_\omega^k$, but with different weights $\omega_j^\epsilon\geq 0$ which are independent of $k$ (remain fixed through the optimization procedure). We have
\begin{align*}
  J_W(\mu_\omega^k) -  J_W(\mu_\omega^\star)&\leq  \int_{\X} J_{\mu_\omega^k}' \, \d(\mu_\omega^k-\mu_\varepsilon^k) + \int_{\X} J_{\mu_\omega^k}' \, \d(\mu_\varepsilon^k-\mu_\omega^\star)\,,\\
  &= \underbrace{\sum_{j=1}^p (\omega_j^k-\omega_j^\epsilon)J_{\mu_\omega^k}'(x_j^k)}_{A} + \underbrace{\int_{\X} J_{\mu_\omega^k}' \, \d(\mu_\epsilon^k-\mu_\omega^\star)}_{B} \,.
\end{align*}
\underline{Control of A:} We introduce the entropy between $\mu_1,\mu_2 \in \M(\X)^+$: 
\begin{equation} \label{eq:def_H_mu1_mu2}
 \mathcal{H}(\mu_1,\mu_2)=-\int_{\X} \ln\left(\frac{\d \mu_1}{\d \mu_2} \right)\,\d\mu_2 -\norm{\mu_2}_{\rm TV}+\norm{\mu_1}_{\rm TV}\,.   
\end{equation}
Remark that as $\omega_j^k>0$ for all $k\geq 1$ and as $\omega_j^{k+1} = \omega_j^k e^{-\alpha J_{\mu_\omega^k}'(x_j^k)}$, 
$J_{\mu_\omega^k}'(x_j^k)=-\frac{1}{\alpha}\ln\left(\frac{\omega_j^{k+1}}{\omega_j^k} \right)$. 

We also use the inequality $e^{-z}+z-1\leq c_{z_{\min}} z^2$ for $z\geq z_{\min}$, together with Lemma \ref{lemma:control_Jp} which gives $\alpha J_{\mu_\omega^k}'\geq - \alpha_{\max}\norm{y}_\L $ for $\alpha\leq \alpha_{\max}$.
With the convention $\omega_j^\epsilon\ln\left(\frac{\omega}{\omega_j^\epsilon}\right) = 0$ when $\omega_j^\epsilon=0$ and $\omega>0$, it comes that
\begin{align}
  A&= \sum_{j=1}^p \omega_j^k J_{\mu_\omega^k}'(x_j^k)-\omega_j^\epsilon J_{\mu_\omega^k}'(x_j^k)\,, \notag \\
  &=\sum_{j=1}^p \omega_j^k J_{\mu_\omega^k}'(x_j^k)+ \frac{1}{\alpha} \omega_j^\epsilon\ln\left(\frac{\omega_j^{k+1}}{\omega_j^k} \right) \,, \notag\\
  &=\sum_{j=1}^p \omega_j^k J_{\mu_\omega^k}'(x_j^k) + \frac{1}{\alpha}(\omega_j^{k+1}-\omega_j^k)+ \frac{1}{\alpha} \left(-\omega_j^\epsilon\ln\left(\frac{\omega_j^{k}}{\omega_j^\epsilon}\right)+\omega_j^k +\omega_j^\epsilon\ln\left(\frac{\omega_j^{k+1}}{\omega_j^\epsilon} \right)-\omega_j^{k+1}\right) \,, \notag\\
  &=\frac{1}{\alpha}\sum_{j=1}^p \omega_j^k\left(\alpha J_{\mu_\omega^k}'(x_j^k) +e^{-\alpha J_{\mu_\omega^k}'(x_j^k)}-1\right) + \frac{1}{\alpha}(\mathcal{H}(\mu_\omega^k,\mu_\epsilon^k)-\mathcal{H}(\mu_\omega^{k+1},\mu_\epsilon^{k+1}) ) \,, \notag\\
  &\leq \alpha c_{\alpha_{\max},\norm{y}_\L} \norm{\mu_\omega^k}_{\rm TV}  + \frac{1}{\alpha}(\mathcal{H}(\mu_\omega^k,\mu_\epsilon^k)-\mathcal{H}(\mu_\omega^{k+1},\mu_\epsilon^{k+1}) )  \,,  \notag\\
  &\leq \alpha c_{\alpha_{\max},\norm{y}_\L, \X, \norm{\mu_\omega^1}_{\rm TV}} + \frac{1}{\alpha}(\mathcal{H}(\mu_\omega^k,\mu_\epsilon^k)-\mathcal{H}(\mu_\omega^{k+1},\mu_\epsilon^{k+1}) )  \,, \label{eq:prop_3-1_bound_A}
\end{align}
where for the last line we used Lemma \ref{lemma:control_tv_iterates} to get 
\begin{align*}
    \norm{\mu_\omega^k}_{\rm TV} 
    &\leq \max\left\{\frac{1}{C_\X}\norm{y}_\L e^{\alpha_{\max}\norm{y}_\L} \,, \, \norm{\mu_\omega^1}_{\rm TV} \right\} \,, \\
    &\leq c_{\alpha_{\max}, \norm{y}_\L , \X, \norm{\mu_\omega^1}_{\rm TV}}\,.
\end{align*}
\underline{Control of $B$:} For $f$ a continuous real-valued function on $\X$, we denote \[Lip(f)=\sup_{x\neq x' \in \X} \frac{|f(x)-f(x')|}{\mathfrak{d}_\mathfrak{g}(x,x')}\,.\] When $f$ is bounded and has a finite Lipschitz constant, we also define the norm BL (Bounded-Lipschitz) by $\norm{f}_{BL}=\norm{f}_{\infty}+Lip(f)$, \cf \citep[Section 1.3]{cpgd_chizat}. The dual norm is $\norm{\mu}_{BL}^*\coloneq \underset{\substack{f \in \C(\X) \\ \norm{f}_{BL}\leq 1}}{\sup} \int_{\X} f \, \d\mu$.
Using Lemma \ref{lemma:control_Jp}, for $x,x' \in \X$ we have
\begin{align*}
    |J_{\mu_\omega^k}'(x)-J_{\mu_\omega^k}'(x')|
    & \leq   \sqrt{B_{02}} \mathfrak{d}_\mathfrak{g}(x,x') (\norm{\mu_\omega^k}_{\rm TV}+\norm{y}_\L)\,, \\
    & \leq c_{\alpha_{\max},\norm{y}_\L,\X, \norm{\mu_\omega^1}_{\rm TV}} \mathfrak{d}_\mathfrak{g}(x,x')
\end{align*}
(the dependence on $d$ coming from $B_{02}$ is contained in the dependence on $\X$.
From Lemmas \ref{lemma:control_Jp} and \ref{lemma:control_tv_iterates} again, using that $\kappa\leq 1$ it also comes that $|J_{\mu_\omega^k}'| \leq c_{\alpha_{\max},\norm{y}_\L,\X, \norm{\mu_\omega^1}_{\rm TV}}$.
So
\begin{align}
    B& \leq \norm{J_{\mu_\omega^k}'}_{BL} \norm{\mu_\epsilon^k-\mu_\omega^\star
    }_{BL}^* \,, \notag \\
    & \leq c_{\alpha_{\max},\norm{y}_\L,\X, \norm{\mu_\omega^1}_{\rm TV}} \norm{\mu_\epsilon^k-\mu_\omega^\star
    }_{BL}^* \,. \label{eq:prop_3-1_bound_B}
\end{align}
\underline{Control of the Cesàro average:} Gathering \eqref{eq:prop_3-1_bound_A} and \eqref{eq:prop_3-1_bound_B}, we get 
\[J_W(\mu_\omega^k)-J_W(\mu_\omega^\star)\leq \alpha c_{\alpha_{\max},\norm{y}_\L,\X,\norm{\mu_\omega^1}_{\rm TV}} + \frac{1}{\alpha}(\mathcal{H}(\mu_\omega^k,\mu_\epsilon^k)-\mathcal{H}(\mu_\omega^{k+1},\mu_\epsilon^{k+1})) + c_{\alpha_{\max},\norm{y}_\L,\X,\norm{\mu_\omega^1}_{\rm TV}} \norm{\mu_\epsilon^k-\mu_\omega^\star
    }_{BL}^*\,.\] 
As $J_W$ is convex,
\[J_W\left(\frac{1}{K}\sum_{k=1}^K \mu_\omega^k\right) - J_W(\mu_\omega^\star) \leq  \frac{1}{K}\sum_{k=1}^K (J_W(\mu_\omega^k)-J_W(\mu_\omega^\star)) \,,\] leading to
\begin{equation}
    J_W\left(\frac{1}{K}\sum_{k=1}^K \mu_\omega^k\right) - J_W(\mu_\omega^\star) \leq \frac{\mathcal{H}(\mu_\omega^1,\mu_\epsilon^1)}{\alpha K}  +\alpha c_{\alpha_{\max},\norm{y}_\L,\X,\norm{\mu_\omega^1}_{\rm TV}}  +\frac{c_{\alpha_{\max},\norm{y}_\L,\X,\norm{\mu_\omega^1}_{\rm TV}}}{K}\sum_{k=1}^K \norm{\mu_\epsilon^k-\mu_\omega^\star}_{BL}^* \label{eq:prop_3-1_bound_cesaro_average_1}
\end{equation}  
using that $\mathcal{H}\geq 0$.
The triangle inequality gives
\begin{align*}
   \frac{1}{K} \sum_{k=1}^K \norm{\mu_\epsilon^k-\mu_\omega^\star
    }_{BL}^* &\leq \norm{\mu_\epsilon^1-\mu_\omega^\star}_{BL}^*+\frac{1}{K} \sum_{k=2}^K \sum_{l=1}^{k-1} \norm{\mu_\epsilon^{l+1}-\mu_\epsilon^l}_{BL}^* \,.
\end{align*}
Moreover for $l=1,\ldots,K-1$, according to the assumption on our update scheme $\mathcal{T}$, 
\begin{equation}
    \norm{\mu_\epsilon^{l+1}-\mu_\epsilon^l}_{BL}^*\leq \sum_{j=1}^p \omega_j^\epsilon \mathfrak{d}_\mathfrak{g}(x_j^{l+1}-x_j^l)\leq \norm{\mu_\epsilon^1}_{\rm TV} C_{\mathcal{T}} \beta\,. \label{eq:prop_3-1_bound_BL}
\end{equation}
Indeed, any $f$ such that $\norm{f}_{BL}\leq 1$ is 1-Lipschitz, hence leading to
\[
    \int_\X f \, \d(\mu_\epsilon^{l+1}-\mu_\epsilon^l)= \sum_{j=1}^p \omega_j^\epsilon (f(x_j^{l+1})-f(x_j^l)) 
    \leq \sum_{j=1}^p \omega_j^\epsilon \mathfrak{d}_\mathfrak{g}(x_j^{l+1}-x_j^l)\,.
\]
Using \eqref{eq:prop_3-1_bound_cesaro_average_1} together with \eqref{eq:prop_3-1_bound_BL}, we obtain
\begin{equation}\label{eq:prop_3-1_bound_cesaro_average_2}
\begin{split}
    J_W\left(\frac{1}{K}\sum_{k=1}^K \mu_\omega^k\right) - J_W(\mu_\omega^\star)
    &\leq \frac{\mathcal{H}(\mu_\omega^1,\mu_\epsilon^1)}{\alpha K}  +\alpha c_{\alpha_{\max},\norm{y}_\L,\X,\norm{\mu_\omega^1}_{\rm TV}}\\
    &\quad +c_{\alpha_{\max},\norm{y}_\L,\X,\norm{\mu_\omega^1}_{\rm TV}}\norm{\mu_\epsilon^1-\mu_\omega^\star}_{BL}^*+ K \beta \norm{\mu_\epsilon^1}_{\rm TV}  c_{\alpha_{\max},\norm{y}_\L,\X,\norm{\mu_\omega^1}_{\rm TV}} C_{\mathcal{T}} \,.
\end{split}
\end{equation}
\noindent
\underline{Choice of $\{\omega_j^\epsilon\}_{j=1}^p$ and control of $\norm{\mu_\epsilon^1-\mu_\omega^\star}_{BL}^*$:} Let $\{\X_j\}_{j=1}^p$ be disjoint regions of $\X$ such that $x_j^1 \in \X_j$ for $j=1,\ldots, p$, and let $\X_0 \coloneq \X \setminus \left( \cup_{j=1}^p \X_j\right)$. We set 
\begin{equation}
    \omega_j^\epsilon = \mu_\omega^\star(\X_j)\,, \quad \forall \: j=1,\ldots, p \,. \label{eq:prop_3-1_weights_auxilary_measure}
\end{equation} Remark that $\norm{\mu_\epsilon^1}_{\rm TV} \leq \norm{\mu_\omega^\star}_{\rm TV}$.

Let $f$ such that $\norm{f}_{BL} \leq 1$. Then $f$ is 1-Lipschitz and $\norm{f}_\infty \leq 1$.
We get 
\begin{align*}
    \int_\X f \,\d (\mu_\epsilon^1 - \mu_\omega^\star) &= \sum_{j=1}^p \int_{\X_j} f \,\d (\mu_\epsilon^1 - \mu_\omega^\star) - \int_{\X_0} f  \,\d\mu_\omega^\star \,, \\
    & \leq \sum_{j=1}^p  \mu_\omega^\star(\X_j) \sup_{x \in \X_j}\mathfrak{d}_\mathfrak{g}(x_j^1,x)  + \mu_\omega^\star(\X_0) \,, \\
    & \leq \norm{\mu_\omega^\star}_{\rm TV} \sup_{\substack{x \in \X_j \,, \\j=1,\ldots,p}}\mathfrak{d}_\mathfrak{g}(x_j^1,x) + \mu_\omega^\star(\X_0) \,,
\end{align*}
from which we deduce that 
\begin{equation}
 \norm{\mu_\epsilon^1-\mu_\omega^\star}_{BL}^* \leq \norm{\mu_\omega^\star}_{\rm TV} \sup_{\substack{x \in \X_j \,, \\j=1,\ldots,p}}\mathfrak{d}_\mathfrak{g}(x_j^1,x) + \mu_\omega^\star(\X_0) \,.
 \label{eq:prop_3-1_bound_BL_2}
\end{equation}
\underline{Control of $\mathcal{H}(\mu_\omega^1,\mu_\epsilon^1)$:} Recalling \eqref{eq:def_H_mu1_mu2} and using \eqref{eq:prop_3-1_weights_auxilary_measure}, we get
\begin{align}
    \mathcal{H}(\mu_\omega^1,\mu_\epsilon^1) &= -\sum_{\substack{j\in \{1,\ldots, p\} \,, \\ \mu_\omega^\star(\X_j)>0}}\ln\left(\frac{\omega_j^1}{\mu_\omega^\star(\X_j)} \right)\mu_\omega^\star(\X_j) - \norm{\mu_\epsilon^1}_{\rm TV}  + \norm{\mu_\omega^1}_{\rm TV} \,, \notag \\
     &= \sum_{\substack{j\in \{1,\ldots, p\} \,, \\ \mu_\omega^\star(\X_j)>0}}\left( -\ln\left(\frac{\omega_j^1}{\mu_\omega^\star(\X_j)} \right)\mu_\omega^\star(\X_j) - \mu_\omega^\star(\X_j) \right) + \norm{\mu_\omega^1}_{\rm TV} \eqcolon H_{\mu_\omega^1,\mu_\omega^\star}\,. 
\end{align}
\underline{Conclusion:}
Choosing $\alpha  = \frac{1}{\sqrt{K}} \leq 1 \eqcolon \alpha_{\max}$ and $0 \leq \beta \leq \frac{1}{K^{3/2}C_{\mathcal{T}}}$, gathering \eqref{eq:prop_3-1_bound_cesaro_average_2}, \eqref{eq:prop_3-1_bound_BL_2} and \eqref{eq:prop_3-1_weights_auxilary_measure}, we get
\begin{align*}
    J_W\left(\frac{1}{K}\sum_{k=1}^K \mu_\omega^k\right) - J_W(\mu_\omega^\star)
    &\leq \frac{H_{\mu_\omega^1,\mu_\omega^\star}}{\alpha K}  +\alpha c_{\alpha_{\max},\norm{y}_\L,\X,\norm{\mu_\omega^1}_{\rm TV}} +c_{\alpha_{\max},\norm{y}_\L,\X,\norm{\mu_\omega^1}_{\rm TV}}\left(\norm{\mu_\omega^\star}_{\rm TV} r_{\mu_\omega^1,\X} + \mu_\omega^\star(\X_0)\right) \\
    & \quad  + K \beta \norm{\mu_\omega^\star}_{\rm TV}   C_{\mathcal{T}}c_{\alpha_{\max},\norm{y}_\L,\X,\norm{\mu_\omega^1}_{\rm TV}} \,, \\
    &\lesssim_{\norm{\mu_\omega^\star}_{\rm TV},\norm{y}_\L,\X,\norm{\mu_\omega^1}_{\rm TV}} \frac{H_{\mu_\omega^1,\mu_\omega^\star} + 1}{\sqrt{K}} + r_{\mu_\omega^1,\X} + \mu_\omega^\star(\X_0)  \,.
\end{align*}

\section{Proof of Lemma \ref{lemma:descent_property}}\label{section:proof_descent_property}
\paragraph{Well-definedness of the updates}
We denote $\overline{\X^C}\coloneq (\R^d \times [u_{\min},+\infty)^d) \setminus \mathring \X$.
Remark that \[\mathfrak{d}_\mathfrak{g}(\overline{\X^C}, \tilde \X) \coloneq \min_{x \in \overline{\X^C}, \, x'\in \tilde \X} \mathfrak{d}_\mathfrak{g}(x,x')>0\,,\] since $\X,\tilde \X$ are compact sets such that $\tilde \X \subset \mathring \X$. This minimum distance depends on $\X, \tilde \X$, and $\tau$ (through the expression of $\mathfrak{d}_\mathfrak{g}$—\cf Appendix \ref{section:geodesics_and_distances}).
It therefore suffices to choose $\beta$ small enough so that
\[\norm{\beta \nabla^\mathfrak{g}J_{\mu}'(x_j^k)}_{x_j^k} \leq \mathfrak{d}_\mathfrak{g}(\overline{\X^C}, \tilde \X)\] whenever $x_j^k \in \tilde \X$.
This guarantees that $x_j^{k+1} = \exp_{x_j^k}(-\beta \nabla^\mathfrak{g}J_{\mu}'(x_j^k))$ remains in $\X$.

Moreover, we get from Lemmas \ref{lemma:control_Jp} and \ref{lemma:control_tv_iterates} that for $\mu \in \M(\X)^+$ and $x\in \X$, $\norm{\nabla^\mathfrak{g}J_{\mu}'(x)}_{x} \leq \sqrt{B_{11}}(\norm{\mu}_{\rm TV}+ \norm{y}_\L)$ and that $\norm{\mu_\omega^k}_{\rm TV} \leq \max\left\{\frac{1}{C_{\X}}\norm{y}_\L e^{\norm{y}_\L} \,, \, \norm{\mu_\omega^1}_{\rm TV} \right\}$ (we used that $\alpha \leq 1$).

Therefore, the required smallness of $\beta$ depends only on $\mathfrak{d}_\mathfrak{g}(\overline{\X^C}, \tilde \X), C_\X, \norm{y}_\L, \norm{\mu_\omega^1}_{\rm TV},B_{11}$; equivalently, on $\X, \tilde \X, \norm{y}_\L, \norm{\mu_\omega^1}_{\rm TV}, \tau$ (with the dependence on $d$ through $B_{11}$ encoded in $\X$). 

\paragraph{Descent property}
Let $k\geq 1$.
Using \eqref{eq:dev_J_mu_plus_sigma_minus_J_mu} we have 
\begin{align*}
 J_W(\mu_\omega^{k+1})-J_W(\mu_\omega^k) &=  \int_{\X} J_{\mu_\omega^k}'(x)  \,\d(\mu_\omega^{k+1}-\mu_\omega^k) (x) + \frac{1}{2} \norm{\Psi(\mu_\omega^{k+1}-\mu_\omega^k)}_\L^2\,.
\end{align*}
\underline{Study of the second order term:} We write $\tilde \mu_\omega^{k+1}=\sum_{j=1}^p\omega_j^k \delta_{x_j^{k+1}}$. 
We have
\begin{align*}
\frac{1}{2} \norm{\Psi(\mu_\omega^{k+1}-\mu_\omega^k)}_\L^2&\leq \norm{\Psi(\mu_\omega^{k+1}-\tilde \mu_\omega^{k+1})}_\L^2 +\norm{\Psi(\mu^{k}-\tilde \mu_\omega^{k+1})}_\L^2\,, \\
&=\underbrace{\norm{\sum_{j=1}^p \omega_j^k (e^{-\alpha J_{\mu_\omega^k}'(x_j^k)}-1)\Psi\delta_{x_j^{k+1}}}_\L^2}_{E_1} +\underbrace{\norm{\sum_{j=1}^p \omega_j^k \Psi(\delta_{x_j^{k+1}}-\delta_{x_j^{k}})}_\L^2 }_{E_2} \,.
\end{align*}
As $J_{\mu_\omega^k}'(x_j^k) \leq B_{\rm TV} + \norm{y}_\L + 1$ (we used Lemmas \ref{lemma:control_tv_iterates} and \ref{lemma:control_Jp} with $\kappa \leq 1$), for $\alpha$ small enough depending only on $B_{\rm TV}, \norm{y}_\L$, it holds that $|e^{-\alpha J_{\mu_\omega^k}'(x_j^k)}-1|\leq 2\alpha |J_{\mu_\omega^k}'(x_j^k)|$. So
\begin{align}
E_1 &\leq \norm{\mu_\omega^k}_{\rm TV} \sum_{j=1}^p \omega_j^k (e^{-\alpha J_{\mu_\omega^k}'(x_j^k)}-1)^2 \norm{\Psi\delta_{x_j^{k+1}}}_\L^2 \,, \notag \\
&\leq 4 B_{\rm TV}  \sum_{j=1}^p \omega_j^k \alpha^2 |J_{\mu_\omega^k}'(x_j^k)|^2 \,. \label{eq:descent_property_bound_E_1}
\end{align}
Using Lemma \ref{lemma:taylor_expansions_kernel}, we also have 
\begin{align}
    E_2 & \leq \norm{\mu_\omega^k}_{\rm TV} \sum_{j=1}^p \omega_j^k \norm{\Psi(\delta_{x_j^{k+1}}-\delta_{x_j^k})}_\L^2 \,, \notag \\
    &\leq B_{02} \norm{\mu_\omega^k}_{\rm TV} \sum_{j=1}^p \omega_j^k  \mathfrak{d}_{\mathfrak{g}}(x_j^k,x_j^{k+1})^2\,, \notag \\
    & \leq B_{02}B_{\rm TV} \beta^2 \sum_{\substack{1 \leq j \leq p, \\ x_j^k \in \tilde \X }} \omega_j^k  \norm{\nabla^{\mathfrak{g}}J_{\mu_\omega^k}'(x_j^k)}_{x_j^k}^2 \,. \label{eq:descent_property_bound_E_2}
\end{align}
\underline{Study of the first order term:} 
\begin{align*}
    \int_{\X} J_{\mu_\omega^k}'(x) \, \d(\mu_\omega^{k+1}-\mu_\omega^k) &= \int_{\X} J_{\mu_\omega^k}'(x) \, \d(\mu_\omega^{k+1}-\tilde \mu_\omega^{k+1})+ \int_{\X} J_{\mu_\omega^k}'(x) \, \d(\tilde \mu_\omega^{k+1}-\mu_\omega^k) \,, \\
    &=\underbrace{\sum_{j=1}^p (\omega_j^{k+1}-\omega_j^k) J_{\mu_\omega^k}'(x_j^k) }_{A}+ \underbrace{\sum_{j=1}^p \omega_j^k \left(J_{\mu_\omega^k}'(x_j^{k+1})-J_{\mu_\omega^k}'(x_j^{k})\right)}_{B} \\
    &\quad + \underbrace{\sum_{j=1}^p (\omega_j^{k+1}-\omega_j^k) (J_{\mu_\omega^k}'(x_j^{k+1})-J_{\mu_\omega^k}'(x_j^k))}_{E_3}\,.
\end{align*}
$E_3$ is a second order term. Indeed, using Lemmas \ref{lemma:control_Jp} and \ref{lemma:control_tv_iterates},
\begin{align}
|E_3|&\leq \sum_{j=1}^p \omega_j^k\left|e^{-\alpha J_{\mu_\omega^k}'(x_j^k)}-1\right| |J_{\mu_\omega^k}'(x_j^{k+1})-J_{\mu_\omega^k}'(x_j^k)| \,, \notag \\
& \leq 2\alpha \beta \sqrt{B_{02}}(\norm{y}_\L + \norm{\mu_\omega^k}_{\rm TV}) \sum_{\substack{1 \leq j \leq p, \\ x_j^k \in \tilde \X }} \omega_j^k |J_{\mu_\omega^k}'(x_j^k)|\norm{\nabla^\mathfrak{g} J_{\mu_\omega^k}'(x_j^k)}_{x_j^k}\,, \notag \\
&\leq \sqrt{B_{02}}  (\norm{y}_\L+ B_{\rm TV})  \sum_{\substack{1 \leq j \leq p, \\ x_j^k \in \tilde \X }} \omega_j^k \left(\alpha^2|J_{\mu_\omega^k}'(x_j^k)|^2+\beta^2\norm{\nabla^\mathfrak{g} J_{\mu_\omega^k}'(x_j^k)}_{x_j^k}^2\right) \,. \label{eq:descent_property_bound_E_3}
\end{align}
Again using the bound on $J_{\mu_\omega^k}'(x_j^k)$ from Lemma \ref{lemma:control_Jp}, we also have, for $\alpha$ small enough only depending on $B_{\rm TV}, \norm{y}_\L$, \[|e^{-\alpha J_{\mu_\omega^k}'(x_j^k)}-1 +\alpha J_{\mu_\omega^k}'(x_j^k)|\leq \alpha^2 J_{\mu_\omega^k}'(x_j^k)^2\,.\] By dealing separately with the cases $J_{\mu_\omega^k}'(x_j^k)>0$ and $J_{\mu_\omega^k}'(x_j^k)\leq 0$, we get
\begin{align}
    A&=\sum_{j=1}^p \omega_j^k(e^{-\alpha J_{\mu_\omega^k}'(x_j^k)}-1) J_{\mu_\omega^k}'(x_j^k) \,, \notag \\
    &\leq  -\alpha \sum_{j=1}^p \omega_j^k |J_{\mu_\omega^k}'(x_j^k)|^2 + \alpha^2\sum_{j=1}^p \omega_j^k |J_{\mu_\omega^k}'(x_j^k)|^3 \,, \notag \\
    & \leq -\alpha \sum_{j=1}^p \omega_j^k |J_{\mu_\omega^k}'(x_j^k)|^2 + \alpha^2(B_{\rm TV}+ \norm{y}_\L +1)\sum_{j=1}^p \omega_j^k |J_{\mu_\omega^k}'(x_j^k)|^2 \,. \label{eq:descent_property_bound_A}
\end{align}
To deal with $B$, we consider 2 cases: if $x_j^k\in \X\setminus \tilde \X$, $x_j^{k+1}=x_j^k$ and $\omega_j^k \left(J_{\mu_\omega^k}'(x_j^{k+1})-J_{\mu_\omega^k}'(x_j^{k+1})\right)=0$. If $x_j^k \in \tilde \X$, by definition of the update using the exponential map we have (using the control on $H^\mathfrak{g} J_{\mu_\omega^k}'$ from Lemma \ref{lemma:control_Jp})
\begin{align*}
    J_{\mu_\omega^k}'(x_j^{k+1})-J_{\mu_\omega^k}'(x_j^k)&=\innerprod{\dot \gamma_{x_j^k\to x_j^{k+1}}(0)^T }{\nabla^{\mathfrak{g}} J_{\mu_\omega^k}'(x_j^k)}_{x_j^k} \\
    &\quad + \int_0^1 (1-y) \dot \gamma_{x_j^k\to x_j^{k+1}}(y)^T H^\mathfrak{g} J_{\mu_\omega^k}'(\gamma_{x_j^k\to x_j^{k+1}}(y)) \dot \gamma_{x_j^k\to x_j^{k+1}}(y) \, \d y \,, \\
    &\leq -\beta \norm{\nabla^{\mathfrak{g}} J_{\mu_\omega^k}'(x_j^k)}_{x_j^k}^2 + \frac{1}{2}\beta^2 \norm{\nabla^{\mathfrak{g}} J_{\mu_\omega^k}'(x_j^k)}_{x_j^k}^2 \sqrt{B_{22}}(\norm{y}_\L +B_{\rm TV})\,.
\end{align*}
Finally, 
\begin{align}
    B&= \sum_{\substack{1\leq j\leq p\,,\\ x_j^k \in \tilde\X}} \omega_j^k \left(J_{\mu_\omega^k}'(x_j^{k+1})-J_{\mu_\omega^k}(x_j^k)\right) \notag \,,\\
    &\leq -\beta \sum_{\substack{1\leq j\leq p\,,\\ x_j^k \in \tilde\X}} \omega_j \norm{\nabla^{\mathfrak{g}} J_{\mu_\omega^k}'(x_j^k)}_{x_j^k}^2 + \frac{1}{2}\beta^2 \sum_{\substack{1\leq j\leq p\,,\\ x_j^k \in \tilde\X}} \omega_j  \norm{\nabla^{\mathfrak{g}} J_{\mu_\omega^k}'(x_j^k)}_{x_j^k}^2 \sqrt{B_{22}}(\norm{y}_\L +B_{\rm TV})\,. \label{eq:descent_property_bound_B}
\end{align}
\underline{Choice of $\alpha, \beta$ and end of the proof:}
Combining \eqref{eq:descent_property_bound_E_1}, \eqref{eq:descent_property_bound_E_2}, \eqref{eq:descent_property_bound_E_3}, \eqref{eq:descent_property_bound_A}, \eqref{eq:descent_property_bound_B}, we get 
\begin{align*}
  J_W(\mu_\omega^{k+1})-J_W(\mu_\omega^k)&\leq -\alpha \sum_{j=1}^p \omega_j^k |J_{\mu_\omega^k}'(x_j^k)|^2  -\beta \sum_{\substack{1\leq j\leq p, \\ x_j^k \in \tilde \X} } \omega_j^k \norm{\nabla^{\mathfrak{g}} J_{\mu_\omega^k}'(x_j^k)}_{x_j^k}^2 
  \\&\quad + \alpha^2(5B_{\rm TV}+ \norm{y}_\L +1 + \sqrt{B_{02}}  (\norm{y}_\L+ B_{\rm TV}))\sum_{j=1}^p \omega_j^k |J_{\mu_\omega^k}'(x_j^k)|^2 
  \\&\quad + \beta^2 \left(\frac{1}{2}\sqrt{B_{22}}(\norm{y}_\L +B_{\rm TV}) +  B_{02}B_{\rm TV} + \sqrt{B_{02}}  (\norm{y}_\L+ B_{\rm TV})\right)
  \\&\qquad\quad \times \sum_{\substack{1 \leq j \leq p, \\ x_j^k \in \tilde \X }} \omega_j^k \norm{\nabla^{\mathfrak{g}} J_{\mu_\omega^k}'(x_j^k)}_{x_j^k}^2 \,.
\end{align*}
We can choose $\alpha,\beta$ small enough only depending on $\X,\tilde \X,\norm{y}_\L,\tau, \norm{\mu_\omega^1}_{\rm TV}$ (the dependence on $d$ is captured by the dependence on $\X$) such that the gradients steps are well-defined and ensuring that for all $k\geq 1$,
\[J_W(\mu_\omega^{k+1})-J_W(\mu_\omega^k)\leq -\frac{1}{2} \int_{\tilde \X} \left(\alpha |J_{\mu_\omega^k}'(x)|^2 +\beta \norm{\nabla^{\mathfrak{g}} J_{\mu_\omega^k}'(x)}_{x}^2\right) \, \d\mu_\omega^k(x) -\frac{1}{2} \int_{\X\setminus \tilde \X} \alpha |J_{\mu_\omega^k}'(x)|^2  \, \d\mu_\omega^k(x) \,.\]

\section{Proof of Theorem \ref{th:exponential_cv}}\label{section:proof_exponential_cv}
We rewrite here the calculations of the proof of \citep[Corollary 3.5]{cpgd_chizat} with the geometric objects induced by our particular geometry, including geodesics and Riemannian derivatives. In particular, we make use of the specific retraction \eqref{eq:RGD}, corresponding to Riemannian gradient descent.

\paragraph{Roadmap of the proof}
In the next subsections, we assume that Assumption~\ref{assumption:non_degenerate_certif_solution} holds. We fix the parameters~$\alpha$ and $\beta$ as in Lemma~\ref{lemma:descent_property}, so that for all $k \geq 1$,
\begin{equation}\label{eq:objective_decrease_descent_property}
    J_W(\mu_\omega^{k+1})-J_W(\mu_\omega^k)\leq -\frac{1}{2} \min\{\alpha,\beta\} \left(\int_{\tilde \X} \left( |J_{\mu_\omega^k}'(x)|^2 +\norm{\nabla^{\mathfrak{g}} J_{\mu_\omega^k}'(x)}_{x}^2\right) \, \d\mu_\omega^k(x) + \int_{\X\setminus \tilde \X}  |J_{\mu_\omega^k}'(x)|^2  \, \d\mu_\omega^k(x) \right)\,.
\end{equation}
This estimate on the objective decrease serves as the starting point for the convergence analysis of the algorithm.

\medskip

When the solution $\mu_\omega^\star$ is discrete and of the form $\sum_{j=1}^{p^\star} \omega_j^\star \delta_{x_j^\star}$, we define the associated near and far regions, for $r>0$, as \[\X_j^{\rm near}(r)\coloneq \{x \in \X \: :\: \mathpzc{d}(x,x_j^\star)\leq r\}\,, \qquad \X^{\rm far}(r)\coloneq \X \setminus \X^{\rm near}(r) \qquad \text{with}\quad \X^{\rm near}(r)=\bigcup_j \X_j^{\rm near}(r)\,.\]
We choose $r$ small enough so that $(\X_j^{\rm near}(r))_{1\leq j \leq p^\star}$ are disjoint and contained in $\tilde \X$ (Lemma \ref{lemma:control_certif_near_regions}). In particular, this implies that $\X \setminus \tilde \X \subset \X^{\rm far}(r)$. As a consequence, \eqref{eq:objective_decrease_descent_property} yields
\begin{equation}\label{eq:objective_decrease_G_mu}
    J_W(\mu_\omega^{k+1})-J_W(\mu_\omega^k)\leq - \frac{1}{2} \min\{\alpha, \beta\} G_{\mu_\omega^k}
\end{equation}
where for $\mu \in \M(\X)^+$,
\begin{equation}\label{eq:def_G_mu}
    G_\mu\coloneq \int_{\X^{\rm near}(r)} \left(|J_\mu'(x)|^2 + \norm{\nabla^{\mathfrak{g}} J_\mu'(x)}_{x}^2\right) \, \d\mu(x) +\int_{\X^{\rm far}(r)} |J_\mu'(x)|^2  \, \d\mu(x) \,.
\end{equation}

\medskip
Using Assumption \ref{assumption:non_degenerate_certif_solution}, we derive quantitative controls on the certificate $\eta^\star$ over the near and far regions (see Appendix \ref{section:controls_certificate_solution_near_far_regions}).

In Appendix \ref{section:lower_bound_objective_decrease}, we establish a lower bound on the objective decrease $G_\mu$ for measures $\mu\in \M(\X)^+$ such that $\norm{\mu}_{\rm TV} \leq B_{\rm TV}$.

In Appendix \ref{section:upper_bound_optimality_gap}, we derive an upper bound on the optimality gap $J_W(\mu)-J_W(\mu_\omega^\star)$. 

Finally, in Appendix \ref{section:assembling_lower_upper_bounds}, we combine these estimates. For $r$ sufficiently small, and for $\mu$ such that $J_W(\mu)$ is close enough to $J_W(\mu_\omega^\star)$, we show that $G_{\mu}\gtrsim J_W(\mu)-J_W(\mu_\omega^\star)$.

Recalling \eqref{eq:objective_decrease_G_mu}, this relation leads to exponential convergence, provided that the initialization is chosen sufficiently close to the solution (see Appendix~\ref{section:proof_exponential_cv_from_upper_lower_bounds}).

Throughout the proof, the notations $A$, $B$, and $E_i$ may refer to different quantities depending on the subsection.

\subsection{Controls of the certificate over near and far regions}\label{section:controls_certificate_solution_near_far_regions}
\subsubsection{Control of the Fisher-Rao metric}
\begin{lemma}[Bounding the semi-distance by the Fisher–Rao distance]\label{lemma:bound_semidist_fisher_rao}
    For all $x,x' \in \R^d \times [u_{\min},+\infty)^d$ such that $\mathpzc{d}(x,x')\leq \frac{0.3025}{\sqrt{d}}$, it holds that \[\mathfrak{d}_\mathfrak{g}(x,x')^2 \leq \frac{1}{\bar\varepsilon_2}\mathpzc{d}(x,x')^2 \quad \text{with} \quad  \bar\varepsilon_2=0.13139\] along with \[\mathfrak{d}_\mathfrak{g}(x,x')^2 \geq \frac{1}{\bar\varepsilon_3}\mathpzc{d}(x,x')^2 \quad \text{with} \quad  \bar\varepsilon_3=2.84\,.\]
The radius $\frac{0.3025}{\sqrt{d}}$ and the constants $\bar\varepsilon_2$, $\bar\varepsilon_3$ are the local positive curvature parameters of $K_{\rm norm}$ obtained in \citep[Theorem 5.3 and Lemma 5.4]{article_giard_decastro_marteau}.
\end{lemma}
\begin{proof}
    Let $x,x' \in \R^d \times [u_{\min},+\infty)^d$ such that $\mathpzc{d}(x,x')\leq \frac{0.3025}{\sqrt{d}}$. For $y \in [0,1]$, $\gamma_{x \to x'}(y) \in \R^d \times [u_{\min},+\infty)^d$ and $\mathpzc{d}(\gamma_{x \to x'}(y),x)\leq \mathpzc{d}(x,x')\leq \frac{0.3025}{\sqrt{d}}$ (\cf Appendix \ref{section:geodesics_and_distances}). We then use the control on the local curvature (Lemma \ref{lemma:controls_kernel}) to deduce that $-K^{(20)}(\gamma_{x \to x'}(y),x)[v,v] \geq \bar \varepsilon_2 \norm{v}_{\gamma_{x \to x'}(y)}^2$ for $v\in \R^{2d}$.

    So
    \begin{align*}
    1-e^{-\frac{1}{2}\mathpzc{d}(x,x')^2}&=1-K_{\rm norm}(x',x) \,, \\
    &=-\int_{0}^1(1-y) K^{(20)}(\gamma_{x \to x'}(y),x)[\dot \gamma_{x \to x'}(y), \dot \gamma_{x \to x'}(y)] \, \d y\,, \\
        &\geq \int_{0}^1(1-y) \bar \varepsilon_2 \norm{\dot \gamma_{x \to x'}(y)}_{\gamma_{x \to x'}(y)}^2\, \d y\,, \\
        &= \frac{\bar \varepsilon_2}{2} \mathfrak{d}_\mathfrak{g}(x,x')^2 \,.
    \end{align*}
    As $z^2 \geq 2(1-e^{-\frac{1}{2}z^2})$, we deduce that 
    \[\mathfrak{d}_\mathfrak{g}(x,x')^2 \leq \frac{1}{\bar\varepsilon_2}\mathpzc{d}(x,x')^2 \,.\]
   The second inequality comes from \citep[Lemma 5.4]{article_giard_decastro_marteau}.
\end{proof}

\begin{lemma}[Local variation of the metric tensor]\label{lemma:local_variation_metric}
For $x, x' \in \R^d \times [u_{\min},+\infty)^d$ such that $\mathpzc{d}(x,x') \leq 1$,
\[\norm{\mathfrak{g}_{x'}^{1/2} \mathfrak{g}_{x}^{-1/2}}_2 \leq \sqrt{6}e \eqcolon C_\mathfrak{g}\,. \]
\end{lemma}
\begin{proof}
    Recall that \[\mathfrak{g}_x=\diag \left(\frac{1}{2u_1^2+\tau^2},\ldots, \frac{1}{2u_d^2+\tau^2},\frac{2u_1^2}{(2u_1^2+\tau^2)^2} ,\ldots, \frac{2u_d^2}{(2u_d^2+\tau^2)^2} \right)\,.\] 
    So the eigenvalues of $\mathfrak{g}_{x'}^{1/2} \mathfrak{g}_{x}^{-1/2}$ are bounded by $\lambda\coloneq \max_i \left\{ \frac{\sqrt{2(u_i)^2+\tau^2}}{\sqrt{2(u_i')^2+\tau^2}}, \frac{(2(u_i)^2+\tau^2)\sqrt{2(u_i')^2}}{(2(u_i')^2+\tau^2)\sqrt{2u_i^2}}\right\}$. As $\tau \leq u_{\min}$, $\lambda \leq \sqrt{\frac{3}{2}} \max_i \frac{\sqrt{2(u_i)^2+\tau^2}}{\sqrt{2(u_i')^2+\tau^2}}$. From \citep[Equation (34)]{article_giard_decastro_marteau}, as $\frac{(u_i')^2+u_i^2+\tau^2}{\sqrt{2( u_i')^2+\tau^2}\sqrt{2u_i^2+\tau^2}}\leq e$ we have \[\frac{\sqrt{2(u_i)^2+\tau^2}}{\sqrt{2(u_i')^2+\tau^2}} \leq e+\sqrt{e^2-1} \leq 2e\,.\]
\end{proof}

\subsubsection{Implications of Assumption \ref{assumption:non_degenerate_certif_solution}}
\begin{lemma}[Control of the certificate on near and far regions]\label{lemma:control_certif_near_regions}
Under Assumption \ref{assumption:non_degenerate_certif_solution}, we can choose $0\leq r_{\max}\leq \frac{0.3025}{\sqrt{d}}$ small enough, depending on $d, \X, \tilde \X, \tau, \mu_\omega^\star$, such that
    \begin{enumerate}
        \item $(\X_j^{\rm near}(r_{\max}))_{j=1}^{p^\star}$ are non-intersecting regions contained in $\tilde \X$,
        \item $\eta^\star(x) \leq 1-\varepsilon_0(r_{\max})$ for all $x\in \X^{\rm far}(r_{\max})$, for $\varepsilon_0(r_{\max})>0$ depending on $r_{\max}$, $\eta^\star$, $\X$,
        \item $-v H^\mathfrak{g} \eta^\star(x) v \geq \frac{\varepsilon_2}{2} \norm{v}_x^2$ for all $v\in \R^{2d}$, $x\in \mathcal{G}_{x_j^\star}(B_{\mathpzc{d}}(x_j^\star,r_{\max}))$, implying that $\eta^\star(x) \leq 1 - \frac{\varepsilon_2}{4} \mathfrak{d}_\mathfrak{g}(x,x_j^\star)^2$ for all $x\in\X_j^{\rm near}(r_{\max})$,
\item For $0<r \leq \min\left\{r_{\max}, \sqrt{\frac{4\varepsilon_0(r_{\max}) \bar\varepsilon_3}{\varepsilon_2}}\right\}\eqcolon r_{\max}'$ and $x\in \X^{\rm far}(r)$, $\eta^\star(x)\leq 1-\varepsilon_0' r^2$ (with $\varepsilon_0'\coloneq \frac{\varepsilon_2}{4\bar\varepsilon_3}$).
    \end{enumerate} 
\end{lemma}
\begin{proof}
   We exploit the information about the location of the solution provided by Assumption \ref{assumption:non_degenerate_certif_solution}, the smoothness of $\mathpzc{d}$, the fact that $\mathpzc{d}(x,x')=0 \iff x=x'$, the compactness of $\tilde \X$, and that $x_j^\star \in \mathring{\tilde \X}$ to establish the first item, for $r_{\max}$ small enough.

The items 2 and 3 come from Assumption \ref{assumption:non_degenerate_certif_solution} together with the smoothness of $\eta^\star$ and the compactness of $\X$. With item 3, we also have, for $x\in\X_j^{\rm near}(r_{\max})$, 
\begin{align*}
    \eta^\star(x) &= \eta^\star(x_j^\star) + \dot\gamma_{x_j^\star \to x}(0)^T \nabla\eta^\star(x_j^\star) + \int_0^1 (1-y) \dot\gamma_{x_j^\star \to x}(y)^T  H^\mathfrak{g} \eta^\star(\gamma_{x_j^\star \to x}(y)) \dot\gamma_{x_j^\star \to x}(y) \, \d y \,, \\
    &= 1 + \int_0^1 (1-y) \dot\gamma_{x_j^\star \to x}(y)^T  H^\mathfrak{g} \eta^\star(\gamma_{x_j^\star \to x}(y)) \dot\gamma_{x_j^\star \to x}(y) \, \d y \,, \\
    &\leq 1- \frac{\varepsilon_2}{4}\mathfrak{d}_\mathfrak{g}(x,x_j^\star)^2 \,.
\end{align*}

The last item can be deduced using the previous ones, along with Lemma \ref{lemma:bound_semidist_fisher_rao}, taking $r_{\max} \leq \frac{0.3025}{\sqrt{d}}$. For $x\in \X^{\rm far}(r_{\max})$, $\eta^\star(x)\leq 1-\varepsilon_0(r_{\max})$. For $x\in \X^{\rm far}(r)\setminus \X^{\rm far}(r_{\max})$, there exists $j\in \{1,\ldots p^\star\}$ such that $x\in \X_j^{\rm near}(r_{\max})$. So as $\mathpzc{d}(x,x_j^\star)>r$, \[\eta^\star(x)\leq 1- \frac{\varepsilon_2}{4} \mathfrak{d}_{\mathfrak{g}}(x,x_j^\star)^2 \leq 1 - \frac{\varepsilon_2}{4\bar \varepsilon_3}\mathpzc{d}(x,x_j^\star)^2 \leq 1 - \frac{\varepsilon_2}{ 4\bar\varepsilon_3} r^2\,.\]
\end{proof}

\medskip

In the next subsections, we consider $0<r\leq r_{\max}'$ small enough so that the conclusions of Lemma \ref{lemma:control_certif_near_regions} apply. In particular, \eqref{eq:objective_decrease_G_mu} holds. The radius $r$ will be chosen even smaller later in the proof.

\paragraph{Identification of the error terms}
Throughout the proof, we rely on a basic control of the optimality gap $J_W(\mu)-J_W(\mu_\omega^\star)$ (for $\mu \in \M(\X)^+$), which makes it possible to isolate the relevant error terms. Combining \eqref{eq:dev_J_mu_plus_sigma_minus_J_mu} with Lemma~\ref{lemma:control_certif_near_regions}, and using $J_{\mu_\omega^\star}' = \kappa(1 - \eta^\star)$, we obtain
\begin{align}
    J_W(\mu)-J_W(\mu_\omega^\star) &= \int_{\X} J_{\mu_\omega^\star}'(x) \, \d\mu(x) + \frac{1}{2}\norm{\Psi(\mu-\mu_\omega^\star)}_{\L}^2\,, \notag \\
    &=  \sum_{j=1}^{p^\star} \kappa\int_{\X_j^{\rm near}(r)} (1-\eta^\star) \, \d\mu + \kappa\int_{\X^{\rm far}(r)} (1-\eta^\star) \, \d\mu + \frac{1}{2}\norm{\Psi(\mu-\mu_\omega^\star)}_{\L}^2\,, \notag \\
    &\geq \kappa\frac{\varepsilon_2}{4} \sum_{j=1}^{p^\star}\int_{\X_j^{\rm near}(r)}\mathfrak{d}_\mathfrak{g}(x_j^\star,x)^2 \, \d\mu(x) + \kappa\varepsilon_0'r^2 \mu(\X^{\rm far}(r)) + \frac{1}{2}\norm{\Psi(\mu-\mu_\omega^\star)}_{\L}^2\,. \label{eq:identification_error_terms}
\end{align}

\subsection{Lower bound on the objective decrease}\label{section:lower_bound_objective_decrease}
Let $\mu \in \M^+(\X)$ such that $\norm{\mu}_{\rm TV}\leq B_{\rm TV}$ (\cf Lemma \ref{lemma:control_tv_iterates}). The goal of this section is to provide a lower bound on $G_{\mu}$ (see \eqref{eq:def_G_mu}). There are 3 terms to consider: $\int_{\X^{\rm near}(r)} |J_\mu'(x)|^2 \, \d\mu(x)$, $\int_{\X^{\rm near}(r)} \norm{\nabla^\mathfrak{g} J_\mu'(x)}_x^2 \, \d\mu(x)$ and $\int_{\X^{\rm far}(r)} |J_\mu'(x)|^2\, \d\mu(x)$.

For $x\in \X$, we control $|J_\mu'(x)|$ and $\norm{\nabla^{\mathfrak{g}}J_\mu'(x)}_x$
using the decomposition
 \begin{align}
    J_{\mu}'(x)&=J_{\mu_\omega^\star}'(x)+ \innerprod{\Psi \delta_x}{\Psi(\mu-\mu_\omega^\star)}_\L \,, \notag \\
    &=J_{\mu_\omega^\star}'(x)+ \int K_{\rm norm}(x,x') \, \d (\mu-\mu_\omega^\star)(x') \label{eq:decomposition_Jp} \,.
 \end{align}
We study these terms with Taylor expansions on $J_{\mu_\omega^\star}$ and $K_{\rm norm}$.
Throughout the proof, we will use the identity $J_{\mu_\omega^\star}' = \kappa(1- \eta^\star)$ together with the controls on $\eta^\star$ provided by Assumption \ref{assumption:non_degenerate_certif_solution} and Lemma \ref{lemma:control_certif_near_regions}. 
In particular, 
\begin{equation}\label{eq:prop_Jp_grad_0}
J_{\mu_\omega^\star}'(x_j^\star)=0 \quad \text{and} \quad \nabla J_{\mu_\omega^\star}'(x_j^\star)=0 \quad \forall\, j=1,\ldots,p^\star \,.
\end{equation}

We will also rely on the controls on $J_{\mu}'$ given in Lemma \ref{lemma:control_Jp}.
 
\paragraph{Definitions for the main terms}
For any $i,j\in \{1,\ldots,p^\star\}$, we denote $\tilde b_{x_j}=\int_{\X_j^{\rm near} }\mathfrak{g}_{x_j^\star}^{1/2} \dot \gamma_{x_j^\star \to x}(0) \, \d \mu(x)$, $\tilde b_{\omega_j}=\mu(\X_j^{\rm near}(r)) - \omega_j^\star$, $\tilde b = \begin{pmatrix}
\tilde b_{\omega}\\
 \tilde b_{x}
 \end{pmatrix}$ and $\tilde K(x_j^\star,x_i^\star)= K_{\rm norm}(x_j^\star,x_i^\star)$. We multiply the derivatives of $K_{\rm norm}$ by the metric, defining
\begin{align*}
    \tilde K_1(x_j^\star,x_i^\star)&\coloneq \mathfrak{g}_{x_j^\star}^{-1/2} \nabla_1K_{\rm norm}(x_j^\star,x_i^\star) \,, \\ 
    \tilde K_2(x_j^\star,x_i^\star)&\coloneq  \mathfrak{g}_{x_i^\star}^{-1/2} \nabla_2 K_{\rm norm}(x_j^\star,x_i^\star) \,, \\
   \tilde K_{12}(x_j^\star,x_i^\star)&\coloneq  \mathfrak{g}_{x_j^\star}^{-1/2} \nabla_2 \nabla_1K_{\rm norm}(x_j^\star,x_i^\star) \mathfrak{g}_{x_i^\star}^{-1/2}\,.
\end{align*}
We also write $\tilde H_j =\mathfrak{g}_{x_j^\star}^{-1/2}  H^{\mathfrak{g}}J_{\mu_\omega^\star}'(x_j^\star)\mathfrak{g}_{x_j^\star}^{-1/2} $.

Remark that as $\dot \gamma_{x_j^\star \to x_j^\star}=0$, $\tilde b_{x_j}=\int_{\X_j^{\rm near}(r)} \mathfrak{g}_{x_j^\star}^{1/2}\dot \gamma_{x_j^\star \to x}(0)\, \d(\mu-\mu_\omega^\star)(x)$. 
\subsubsection{Lower bound on \texorpdfstring{$\int_{\X^{\rm near}(r)}|J_{\mu}'|^2 \, \d\mu$}{Jp over near regions}}
We first derive an upper bound on $\int_{\X^{\rm near}(r)} |J_{\mu_\omega^\star}'|^2 \, \d\mu$, then a lower bound on
\[\int_{\X^{\rm near}(r)} \left| \int_\X K_{\rm norm}(x,x') \, \d(\mu - \mu_\omega^\star)(x') \right|^2 \, \d\mu(x)\,,\]
and finally conclude using \eqref{eq:decomposition_Jp}.

\medskip

\noindent
\underline{Control of $J_{\mu_\omega^\star}'(x)$:}
Let $x \in \X_j^{\rm near}(r)$. Using \eqref{eq:prop_Jp_grad_0}, we obtain
\begin{align*}
   E_1(x) \coloneq J_{\mu_\omega^\star}'(x) &= J_{\mu_\omega^\star}'(x_j^\star) + \dot \gamma_{x_j^\star \to x}(0)^T \nabla J_{\mu_\omega^\star}'(x_j^\star) + \int_0^1 (1-y) \dot \gamma_{x_j^\star \to x}(y)^TH^{\mathfrak{g}}J_{\mu_\omega^\star}'(\gamma_{x_j^\star \to x}(y))\dot \gamma_{x_j^\star \to x}(y) \, \d y \,, \\
   &=\int_0^1 (1-y) \dot \gamma_{x_j^\star \to x}(y)^TH^{\mathfrak{g}}J_{\mu_\omega^\star}'(\gamma_{x_j^\star \to x}(y))\dot \gamma_{x_j^\star \to x}(y) \, \d y \,.
\end{align*}
Applying Lemma~\ref{lemma:control_Jp}, we deduce that
\begin{align*}
    \int_{X^{\rm near}(r)}|E_1(x)|^2 \, \d\mu(x)
    &\leq \sum_{j=1}^{p^\star} \int_{\X_j^{\rm near}(r)}\left(\mathfrak{d}_\mathfrak{g}(x_j^\star,x)^2 \frac{\sqrt{B_{22}}}{2} (\norm{\mu_\omega^\star}_{\rm TV}+\norm{y}_\L) \right)^2 \,\d\mu(x) \,, \\
  &\lesssim_{d,\norm{y}_\L,\norm{\mu_\omega^\star}_{\rm TV}} r^2 \sum_{j=1}^{p^\star} \int_{\X_j^{\rm near}(r)}\mathfrak{d}_\mathfrak{g}(x_j^\star,x)^2 \,\d\mu(x) \,,
\end{align*}
where we used that, for $x \in \X_j^{\rm near}(r)$, \[\mathfrak{d}_\mathfrak{g}(x_j^\star,x)^2\leq \frac{1}{\bar \varepsilon_2}\mathpzc{d}(x_j^\star,x)^2 \leq \frac{r^2}{\bar \varepsilon_2}\] by Lemma \ref{lemma:bound_semidist_fisher_rao}.

\noindent
\\\underline{Control of $ \int K_{\rm norm}(x,x')\, \d(\mu-\mu_\omega^\star)(x')$:}
Moreover, for any $x \in \X_j^{\rm near}(r)$,
\begin{align*}
    \int K_{\rm norm}(x,x') \,& \d(\mu-\mu_\omega^\star)(x')
   \\ &= \int  K_{\rm norm}(x_j^\star,x') \, \d(\mu-\mu_\omega^\star)(x') + \int_0^1 \dot \gamma_{x_j^\star \to x}(y)^T \int \nabla_1 K_{\rm norm}( \gamma_{x_j^\star \to x}(y),x') \, \d(\mu-\mu_\omega^\star)(x')\, \d y \,, \\
   &= \sum_{i=1}^{p^\star}\int_{\X_i^{\rm near}(r)}  K_{\rm norm}(x_j^\star,x') \, \d(\mu-\mu_\omega^\star)(x') \\
   &\quad + \int_0^1 \dot \gamma_{x_j^\star \to x}(y)^T \int \nabla_1 K_{\rm norm}( \gamma_{x_j^\star \to x}(y),x') \, \d(\mu-\mu_\omega^\star)(x')\, \d y +\int_{\X^{\rm far}(r)}  K_{\rm norm}(x_j^\star,x') \, \d\mu(x') \,,  \\
   &= \sum_{i=1}^{p^\star} \tilde b_{\omega_i} \tilde K(x_j^\star,x_i^\star) +  \sum_{i=1}^{p^\star}  \tilde K_2(x_j^\star,x_i^\star)^T \tilde b_{x_i^\star}\\
   &\quad +
   \sum_{i=1}^{p^\star} \int_{\X_i^{\rm near}(r)} \int_0^1 (1-y)\dot \gamma_{x_i^\star \to x'}(y)^T  H_2^{\mathfrak{g}} K_{\rm norm}(x_j^\star,\gamma_{x_i^\star \to x'}(y)) \dot \gamma_{x_i^\star \to x'}(y) \, \d y \, \d(\mu-\mu_\omega^\star)(x')\\
  &\quad +\int_0^1 \dot \gamma_{x_j^\star \to x}(y)^T \int \nabla_1 K_{\rm norm}( \gamma_{x_j^\star \to x}(y),x') \, \d(\mu-\mu_\omega^\star)(x')\, \d y +\int_{\X^{\rm far}(r)}  K_{\rm norm}(x_j^\star,x') \, \d\mu(x') \,, \\
   &\eqcolon \sum_{i=1}^{p^\star} \tilde b_{\omega_i} \tilde K(x_j^\star,x_i^\star)  + \sum_{i=1}^{p^\star} \tilde K_2(x_j^\star,x_i^\star)^T \tilde b_{x_i}  + E_2(x)
\end{align*}
where 
\begin{align*}
|E_2(x)|&\leq  \frac{B_{02}}{2} \sum_{i=1}^{p^\star} \int_{\X_i^{\rm near}(r)} \mathfrak{d}_\mathfrak{g}(x_i^\star,x')^2 \, \d\mu(x')  + \mathfrak{d}_\mathfrak{g}(x_j^\star,x) \sqrt{B_{11}} \norm{\Psi(\mu-\mu_\omega^\star)}_\L  + \mu(\X^{\rm far}(r))\,.
\end{align*}
Note that we used $\dot \gamma_{x_i^\star \to x_i^\star}=0$ to get
\begin{align*}
 &  \left|\sum_{i=1}^{p^\star} \int_{\X_i^{\rm near}(r)} \int_0^1 (1-y) \dot \gamma_{x_i^\star \to x'}(y)^T H_2^{\mathfrak{g}} K_{\rm norm}(x, \gamma_{x_i^\star \to x'}(y))\dot \gamma_{x_i^\star \to x'}(y) \, \d y \, \d(\mu-\mu_\omega^\star)(x') \right|\\ &=  \left|\sum_{i=1}^{p^\star} \int_{\X_i^{\rm near}(r)} \int_0^1 (1-y) \dot \gamma_{x_i^\star \to x'}(y)^T H_2^{\mathfrak{g}} K_{\rm norm}(x, \gamma_{x_i^\star \to x'}(y))\dot \gamma_{x_i^\star \to x'}(y) \, \d y \, \d\mu(x') \right| \\
 &\leq \frac{B_{02}}{2}\sum_{i=1}^{p^\star} \int_{\X_i^{\rm near}(r)}\mathfrak{d}_\mathfrak{g}(x_i^\star,x')^2 \, \d\mu(x')\,.
\end{align*}
So according to \eqref{eq:identification_error_terms},
\begin{align*}
\int_{X^{\rm near}(r)}|E_2(x)|^2 \, \d\mu(x)&\lesssim_{B_{\rm TV}, d} \left(\sum_i \int_{\X_i^{\rm near}(r)} \mathfrak{d}_\mathfrak{g}(x_i^\star,x')^2 \, \d\mu(x') \right)^2  + \sum_j \int_{\X_i^{\rm near}(r)} \mathfrak{d}_\mathfrak{g}(x_j^\star,x)^2 \,\d\mu(x) \norm{\Psi(\mu-\mu_\omega^\star)}_\L^2\\
&\quad + \mu(\X^{\rm far}(r))^2\,, \\
& \lesssim_{B_{\rm TV}, d, \varepsilon_0', \varepsilon_2} \frac{1}{\kappa^2 r^4}(J_W(\mu) - J_W(\mu_\omega^\star))^2 \,.
\end{align*}
\underline{Conclusion for $\int_{\X^{\rm near}(r)}|J_\mu'(x)|^2\, \d \mu(x)$:}
Using that $(a+b)^2\geq \frac{1}{2}a^2-b^2$ and collecting the error terms $\int_{X^{\rm near}(r)}|E_1(x)|^2 \, \d\mu(x)$ and $\int_{X^{\rm near}(r)}|E_2(x)|^2 \, \d\mu(x)$ into $b$, we obtain
\begin{equation} \label{eq:lower_bound_Jp_near_regions}
\begin{split}
    \frac{1}{2} \sum_{j=1}^{p^\star} &\mu(\X_j^{\rm near}(r)) \left(\sum_{i=1}^{p^\star}\tilde b_{\omega_i} \tilde K(x_j^\star,x_i^\star)  + \sum_{i=1}^{p^\star} \tilde K_2(x_j^\star,x_i^\star)^T \tilde b_{x_i} \right)^2- \int_{\X^{\rm near}(r)}|J_\mu'(x)|^2 \, \d \mu(x)
   \\
   &\lesssim_{d,B_{\rm TV},\norm{y}_\L,\norm{\mu_\omega^\star}_{\rm TV}, \varepsilon_2, \varepsilon_0'}  \frac{1}{\kappa^2 r^4}(J_W(\mu)-J_W(\mu_\omega^\star))^2  +  r^2 \sum_j \int_{\X_j^{\rm near}(r)}\mathfrak{d}_\mathfrak{g}(x_j^\star,x)^2 \,\d\mu(x) \,.
   \end{split}
\end{equation}

\subsubsection{Lower bound on \texorpdfstring{$\int_{\X^{\rm far}(r)} |J_{\mu}'|^2 \, \d\mu$}{Jp over the far region}}
According to Lemma \ref{lemma:control_certif_near_regions} and recalling that $J_{\mu_\omega^\star}=\kappa(1-\eta^\star)$, we have $|J_{\mu_\omega^\star}'(x)|\geq \kappa \varepsilon_0'r^2$ for $x\in \X^{\rm far}(r)$. Hence
\[|J_\mu'(x)|=|J_{\mu_\omega^\star}'(x)+\innerprod{\Psi \delta_x}{\Psi(\mu-\mu_\omega^\star)}_\L|\geq \kappa \varepsilon_0'r^2- \norm{\Psi(\mu-\mu_\omega^\star)}_\L\,.\] 
Using again the inequality $(a+b)^2\geq \frac{1}{2}a^2-b^2$ with \eqref{eq:identification_error_terms},
\begin{align*}
    \int_{\X^{\rm far}(r)}|J_\mu'(x)|^2 \, \d \mu(x)&\geq \left(\frac{1}{2}\kappa^2 \varepsilon_0'^2 r^4 - \norm{\Psi(\mu-\mu_\omega^\star)}_\L^2\right) \mu(\X^{\rm far}(r)) \,, \\
    & \geq \frac{1}{2}\kappa^2 \varepsilon_0'^2 r^4 \mu(\X^{\rm far}(r)) - \frac{1}{\kappa^2 \varepsilon_0'^2 r^2}(J_W(\mu)-J_W(\mu_\omega^\star))^2 \,,
\end{align*}
leading to 
\begin{equation}\label{eq:lower_bound_Jp_far_region}
  \frac{1}{2}\kappa^2 \varepsilon_0'^2 r^4 \mu(\X^{\rm far}(r))- \int_{\X^{\rm far}(r)}|J_\mu'(x)|^2 \, \d \mu(x) \lesssim_{\varepsilon_0'} \frac{1}{\kappa^2 r^2}(J_W(\mu)-J_W(\mu_\omega^\star))^2 \,.  
\end{equation}

\subsubsection{Lower bound on \texorpdfstring{$\int_{\X^{\rm near}(r)}\norm{\nabla^{\mathfrak{g}}J_\mu'(x)}_x^2 \, \d \mu(x)$}{the gradient of Jp over near regions}}
\underline{Dealing with Riemannian gradients:}
For $x\in \R^d \times [u_{\min},+\infty)^d$ such that $\mathpzc{d}(x_j^\star,x)\leq r$, Lemma \ref{lemma:local_variation_metric} gives \[\norm{v\mathfrak{g}_{x_j^\star}^{-1/2}}_{2}=\norm{v\mathfrak{g}_{x}^{-1/2}\mathfrak{g}_{x}^{1/2}\mathfrak{g}_{x_j^\star}^{-1/2}}_{2} \leq \norm{ v \mathfrak{g}_x^{-1}}_{x} \norm{\mathfrak{g}_{x}^{1/2}\mathfrak{g}_{x_j^\star}^{-1/2}}_2\leq C_\mathfrak{g}\norm{ v \mathfrak{g}_x^{-1}}_{x} \quad \forall \, v\in \R^d\,.\]  
So for $x\in \X_j^{\rm near}(r)$ ($j=1,\ldots,p^\star$), 
\begin{equation}\label{eq:control_grad_jp_metric_xjstar}
\norm{\nabla^{\mathfrak{g}} J_\mu'(x)}_x=  \norm{\mathfrak{g}_x^{-1} \nabla J_\mu'(x)}_x\geq  \frac{1}{C_\mathfrak{g}}\norm{\mathfrak{g}_{x_j^\star}^{-1/2} \nabla J_\mu'(x)}_2\,.\end{equation}
We use this inequality together with $\mathpzc{d}(x_j^\star,\gamma_{x_j^\star \to x}(y))\leq \mathpzc{d}(x_j^\star,x)\leq r$ for $y \in [0,1]$ (see Appendix \ref{section:geodesics_and_distances}).

\noindent
\underline{Decomposition of $\mathfrak{g}_{x_j^\star}^{-1/2} \nabla J_\mu'(x)$:} Let $x\in \X_j^{\rm near}(r)$. 
As $\nabla^2 J_{\mu_\omega^\star}'(x_j^\star)=H^{\mathfrak{g}} J_{\mu_\omega^\star}'(x_j^\star)$ (because $\nabla J_{\mu_\omega^\star}'(x_j^\star)=0$), using \eqref{eq:decomposition_Jp} we get the decomposition
\begin{align}
  \mathfrak{g}_{x_j^\star}^{-1/2} \nabla J_\mu'(x)&=  \mathfrak{g}_{x_j^\star}^{-1/2}  \nabla^2 J_{\mu_\omega^\star}'(x_j^\star) \dot \gamma_{x_j^\star \to x}(0) \notag \\
    &\quad + \underbrace{ \mathfrak{g}_{x_j^\star}^{-1/2} \int \nabla_1 K_{\rm norm} (x_j^\star,x') \, \d(\mu-\mu_\omega^\star)(x')}_{A(x)} \notag \\
    &\quad + \underbrace{ \mathfrak{g}_{x_j^\star}^{-1/2}\int \int_0^1 \nabla_1^2  K_{\rm norm} (\gamma_{x_j^\star \to x}(y),x') \dot \gamma_{x_j^\star \to x}(y) \, \d(\mu-\mu_\omega^\star)(x')}_{E_1(x)}  + E_2(x)\,, \notag \\
    &= \mathfrak{g}_{x_j^\star}^{-1/2} H^{\mathfrak{g}} J_{\mu_\omega^\star}'(x_j^\star)\dot \gamma_{x_j^\star \to x}(0) +A(x)+ E_1(x)+E_2(x) \label{eq:decomposition_grad_jp}
\end{align}
where
\[E_2(x)\coloneq 
\begin{pmatrix}
\mathfrak{g}_{t_{j,1}^\star t_{j,1}^\star}^{-1/2}
\int_0^1 (1-y)
\dot \gamma_{x_j^\star \to x}(y)^T H^{\mathfrak{g}} \partial_{t_1} J_{\mu_\omega^\star}'(\gamma_{x_j^\star \to x}(y)) \dot \gamma_{x_j^\star \to x}(y)
\, \d y
\\
\vdots
\\
\mathfrak{g}_{t_{j,d}^\star t_{j,d}^\star}^{-1/2}
\int_0^1 (1-y)
\dot \gamma_{x_j^\star \to x}(y)^T H^{\mathfrak{g}} \partial_{t_d} J_{\mu_\omega^\star}'(\gamma_{x_j^\star \to x}(y)) \dot \gamma_{x_j^\star \to x}(y) 
\\
\mathfrak{g}_{u_{j,1}^\star u_{j,1}^\star}^{-1/2}
\int_0^1 (1-y)
\dot \gamma_{x_j^\star \to x}(y)^T H^{\mathfrak{g}} \partial_{u_1} J_{\mu_\omega^\star}'(\gamma_{x_j^\star \to x}(y)) \dot \gamma_{x_j^\star \to x}(y)
\, \d y
\\
\vdots
\\
\mathfrak{g}_{u_{j,d}^\star u_{j,d}^\star}^{-1/2}
\int_0^1 (1-y)
\dot \gamma_{x_j^\star \to x}(y)^T H^{\mathfrak{g}} \partial_{u_d} J_{\mu_\omega^\star}'(\gamma_{x_j^\star \to x}(y)) \dot \gamma_{x_j^\star \to x}(y)
\end{pmatrix}\,.
\]
$E_1$ and $E_2$ are negligible. In fact, 
\begin{align*}
   \norm{E_1(x)}_2&\leq \sup_y \norm{ \mathfrak{g}_{x_j^\star}^{-1/2} \mathfrak{g}_{\gamma_{x_j^\star \to x}(y)}^{1/2}}_2   \norm{\mathfrak{g}_{\gamma_{x_j^\star \to x}(y)}^{-1/2} \innerprod{\nabla^2 \Psi \delta_{\gamma_{x_j^\star \to x}(y)}}{\Psi(\mu-\mu_\omega^\star)} \mathfrak{g}_{\gamma_{x_j^\star \to x}(y)}^{-1/2}}_2  \norm{\dot \gamma_{x_j^\star \to x}(y)^T \mathfrak{g}_{\gamma_{x_j^\star \to x}(y)}^{1/2}}_2 \,, \\ 
   &\leq C_\mathfrak{g} \mathfrak{d}_{\mathfrak{g}}(x,x_j^\star) B_{02}^e \norm{\Psi(\mu-\mu_\omega^\star)}_\L \,.
\end{align*}
We also have 
\begin{align*}
    \norm{E_2(x)}_2 &\leq \frac{1}{2}\mathfrak{d}_{\mathfrak{g}}(x,x_j^\star)^2 \sup_y \norm{\mathfrak{g}_{\gamma_{x_j^\star \to x}(y)}^{1/2}\mathfrak{g}_{x_j^\star}^{-1/2}}_2 \sqrt{2d} \\
    &\qquad \times \sup_{x\in \R^d \times [u_{\min},+\infty)^d\,, \: b_k\in \{t_k,u_k\}}\norm{\mathfrak{g}_{b_k b_k}^{-1/2}\mathfrak{g}_{x}^{-1/2}
H^{\mathfrak{g}} \partial_{b_k} J_{\mu_\omega^\star}'(x) \mathfrak{g}_{x}^{-1/2}}_2\,, \\
    &\leq \frac{1}{2} \sqrt{B_{33}^e} C_\mathfrak{g} (\norm{\mu_\omega^\star}_{\rm TV} +\norm{y}_\L)\mathfrak{d}_{\mathfrak{g}}(x,x_j^\star)^2  \,, \\
    &\leq \frac{1}{2} \sqrt{B_{33}^e} C_\mathfrak{g}(\norm{\mu_\omega^\star}_{\rm TV} +\norm{y}_\L) \tfrac{1}{\sqrt{\bar \varepsilon_2}} r \, \mathfrak{d}_{\mathfrak{g}}(x,x_j^\star) \,.
\end{align*}

Moreover,
\begin{align*}
    A(x)&=\int \mathfrak{g}_{x_j^\star}^{-1/2}\nabla_1 K_{\rm norm}(x_j^\star,x') \, \d(\mu-\mu_\omega^\star)(x') \,, \\
    &=\sum_i \tilde b_{\omega_i}  \tilde K_1(x_j^\star,x_i^\star)  + \sum_i \tilde K_{12}(x_j^\star,x_i^\star) \tilde b_{x_i}\\
    & \quad + \underbrace{\sum_i \int_{X_i^{\rm near}(r)} \int_0^1 (1-y)  \mathfrak{g}_{x_j^\star}^{-1/2}H^{\mathfrak{g}}_2\nabla_1 K_{\rm norm}(x_j^\star, \gamma_{x_i^\star \to x'}(y)) [\dot \gamma_{x_i^\star \to x'}(y),\dot \gamma_{x_i^\star \to x'}(y)]  \, \d y  \, \d(\mu-\mu_\omega^\star)(x')}_{E_3(x)} \\
    & \quad + \underbrace{\int_{\X^{\rm far}(r)} \mathfrak{g}_{x_j^\star}^{-1/2} \nabla_1 K_{\rm norm}(x_j^\star,x') \, \d\mu(x') }_{E_4(x)}\,.
\end{align*}
$E_3$ and $E_4$ are negligible:
$\norm{E_4(x)}_2 \leq \mu(\X^{\rm far}(r))  B_{10}$ and
using again that $\int_{X_i^{\rm near}(r)}\mathfrak{d}_\mathfrak{g}(x_i^\star,x') \, \d\mu_\omega^\star(x')=0$, we get 
\[\norm{E_3(x)}_2 \leq \frac{B_{12}}{2} \sum_i \int_{X_i^{\rm near}(r)} \mathfrak{d}_\mathfrak{g}(x_i^\star , x')^2 \, \d\mu(x')\,.\]
Gathering the error terms, we have
\begin{align}
    \int_{\X^{\rm near}(r)} &\norm{E_1(x)+E_2(x)+E_3(x)+E_4(x)}_2^2 \,\d\mu(x)  \notag \\&\lesssim_{d,B_{\rm TV},\norm{y}_\L,\norm{\mu_\omega^\star}_{\rm TV}} \left(\sum_i \int_{X_i^{\rm near}(r)} \mathfrak{d}_\mathfrak{g}(x_i^\star , x')^2 \, \d\mu(x')  \right)\norm{\Psi(\mu-\mu_\omega^\star)}^2+ r^2 \sum_j \int_{X_j^{\rm near}(r)} \mathfrak{d}_\mathfrak{g}(x_j^\star , x')^2 \, \d\mu(x') \notag \\
    &\quad + \left(\sum_i \int_{X_i^{\rm near}(r)} \mathfrak{d}_\mathfrak{g}(x_i^\star , x')^2 \, \d\mu(x')  \right)^2 + \mu(\X^{\rm far}(r))^2 \,, \notag\\
    &\lesssim_{d,B_{\rm TV},\norm{y}_\L,\norm{\mu_\omega^\star}_{\rm TV}, \varepsilon_2, \varepsilon_0'} \frac{1}{\kappa^2 r^4}(J_W(\mu)-J_W(\mu_\omega^\star))^2 +r^2 \sum_j \int_{\X_j^{\rm near}(r)}\mathfrak{d}_\mathfrak{g}(x_j^\star,x)^2 \,\d\mu(x) \,, \label{eq:grad_jp_control_error_terms}
\end{align}
where we used \eqref{eq:identification_error_terms} to obtain the last inequality.

\noindent
\underline{Assembling the bounds:}
From \eqref{eq:control_grad_jp_metric_xjstar}, \eqref{eq:decomposition_grad_jp} and \eqref{eq:grad_jp_control_error_terms}, it follows that
\begin{align*}
    \frac{1}{2C_\mathfrak{g}^2} \sum_j \int_{\X_j^{\rm near}(r)}&\norm{\tilde H_j \mathfrak{g}_{x_j^\star}^{1/2} \dot \gamma_{x_j^\star \to x}(0) + \sum_i \tilde b_{\omega_i}  \tilde K_1(x_j^\star,x_i^\star)  + \sum_i \tilde K_{12}(x_j^\star,x_i^\star) \tilde b_{x_i}}_2^2 \, \d\mu(x) \\
    -\int_{\X^{\rm near}(r)} &\norm{\nabla^\mathfrak{g}J_\mu'(x)}_x^2 \, \d\mu(x)\\ & \lesssim_{d,B_{\rm TV},\norm{y}_\L,\norm{\mu_\omega^\star}_{\rm TV}, \varepsilon_2, \varepsilon_0'} \frac{1}{\kappa^2 r^4} (J_W(\mu)-J_W(\mu_\omega^\star))^2 + r^2 \sum_j \int_{\X_j^{\rm near}(r)}\mathfrak{d}_\mathfrak{g}(x_j^\star,x)^2 \,\d\mu(x) \,.
\end{align*}
 Finally, remark that for some $v$ not depending on $x$, a bias–variance decomposition gives, for $j$ such that $\mu(\X_j^{\rm near}(r))>0$, \begin{align*}
    \int_{\X_j^{\rm near}(r)}&\norm{\tilde H_j \mathfrak{g}_{x_j^\star}^{1/2}\dot \gamma_{x_j^\star \to x}(0) + v }_2^2 \, \d\mu(x) \\
    &= \int_{\X_j^{\rm near}(r)}\norm{\tilde H_j \frac{\tilde b_{x_j}}{\mu(\X_j^{\rm near}(r))}+ v }_2^2 \, \d\mu(x) + \int_{\X_j^{\rm near}(r)}\norm{\tilde H_j \left(\mathfrak{g}_{x_j^\star}^{1/2}\dot \gamma_{x_j^\star \to x}(0)-\frac{\tilde b_{x_j}}{\mu(\X_j^{\rm near}(r))}\right)}_2^2 \, \d\mu(x)\,,\\  
    &= \mu(\X_j^{\rm near}(r))\norm{\tilde H_j \frac{\tilde b_{x_j}}{\mu(\X_j^{\rm near}(r))}+ v }_2^2+ \int_{\X_j^{\rm near}(r)}\norm{\tilde H_j \left(\mathfrak{g}_{x_j^\star}^{1/2}\dot \gamma_{x_j^\star \to x}(0)-\frac{\tilde b_{x_j}}{\mu(\X_j^{\rm near}(r))}\right)}_2^2 \, \d\mu(x) \,.
\end{align*}
As $\tilde H_j$ is a symmetric matrix of positive definite type, with eigenvalues greater than $\kappa \varepsilon_2$ (by Assumption \ref{assumption:non_degenerate_certif_solution}), it comes that 
\begin{align*}
    \int_{\X_j^{\rm near}(r)}&\norm{\tilde H_j \mathfrak{g}_{x_j^\star}^{1/2}\dot \gamma_{x_j^\star \to x}(0) + v }_2^2 \, \d\mu(x) \\&\geq \mu(\X_j^{\rm near}(r))\norm{\tilde H_j \frac{\tilde b_{x_j}}{\mu(\X_j^{\rm near}(r))}+ v }_2^2 + \kappa^2 \varepsilon_2^2\int_{\X_j^{\rm near}(r)}  \norm{\mathfrak{g}_{x_j^\star}^{1/2}\dot \gamma_{x_j^\star \to x}(0)-\frac{\tilde b_{x_j}}{\mu(\X_j^{\rm near}(r))}}_2^2 \, \d\mu(x) \,.
\end{align*}
So applying this inequality with $v=\sum_i \tilde b_{\omega_i}  \tilde K_1(x_j^\star,x_i^\star)
+ \sum_i \tilde K_{12}(x_j^\star,x_i^\star) \tilde b_{x_i}$, we get
\begin{equation}\label{eq:lower_bound_grad_Jp_near_regions}
\begin{split}
\frac{1}{2C_\mathfrak{g}^2} &\sum_{j, \mu(\X_j^{\rm near}(r))>0} \mu(\X_j^{\rm near}(r)) \norm{\tilde H_j\frac{\tilde b_{x_j}}{\mu(\X_j^{\rm near}(r))}  + \sum_i \tilde b_{\omega_i}  \tilde K_1(x_j^\star,x_i^\star)
+ \sum_i \tilde K_{12}(x_j^\star,x_i^\star) \tilde b_{x_i}}_2^2 
\\
&+\frac{\kappa^2 \varepsilon_2^2}{2C_\mathfrak{g}^2} \sum_{j, \mu(\X_j^{\rm near}(r))>0} \int_{\X_j^{\rm near}(r)}\norm{\mathfrak{g}_{x_j^\star}^{-1/2} \dot \gamma_{x_j^\star \to x}(0)- \frac{\tilde b_{x_j}}{\mu(\X_j^{\rm near}(r))}}_2^2 \, \d\mu(x)
    -\int_{\X^{\rm near}(r)} \norm{\nabla^\mathfrak{g}J_\mu'(x)}_x^2 \, \d\mu(x)
    \\ & \quad\lesssim_{d,B_{\rm TV},\norm{y}_\L,\norm{\mu_\omega^\star}_{\rm TV}, \varepsilon_2, \varepsilon_0'} \frac{1}{\kappa^2 r^4} (J_W(\mu)-J_W(\mu_\omega^\star))^2 +  r^2 \sum_j \int_{\X_j^{\rm near}(r)}\mathfrak{d}_\mathfrak{g}(x_j^\star,x)^2 \,\d\mu(x) \,.
\end{split}
\end{equation}

\subsubsection{Conclusion for the lower bound, choice of \texorpdfstring{$r$}{r}} 
Remark that
\begin{equation}
   \label{eq:control_int_fisher_rao_norm_b}
\int_{\X_j^{\rm near}(r)} \mathfrak{d}_\mathfrak{g}(x_j^\star,x)^2 \, \d\mu(x) = \mathds{1}_{\mu(\X_j^{\rm near}(r))>0} \left(\int_{\X_j^{\rm near}(r)}  \norm{\mathfrak{g}_{x_j^\star}^{1/2}\dot \gamma_{x_j^\star \to x}(0)-\frac{\tilde b_{x_j}}{\mu(\X_j^{\rm near}(r))}}_2^2 \, \d\mu(x) +\frac{\norm{\tilde b_{x_j}}_2^2}{\mu(\X_j^{\rm near}(r))} \right)\,. 
\end{equation}
We now have all the ingredients to control the objective decrease $G_{\mu}$ \eqref{eq:def_G_mu}. 
Assembling the bounds \eqref{eq:lower_bound_Jp_near_regions}, \eqref{eq:lower_bound_Jp_far_region}, \eqref{eq:lower_bound_grad_Jp_near_regions} and using \eqref{eq:control_int_fisher_rao_norm_b}, we get
\begin{align*}
G_{\mu} &\gtrsim_{\varepsilon_0', \varepsilon_2} \kappa^2 r^4 \mu(\X^{\rm far}(r)) + \sum_{j=1}^{p^\star} \mu(\X_j^{\rm near}(r)) \left(\sum_{i=1}^{p^\star} \tilde b_{\omega_i} \tilde K(x_j^\star,x_i^\star)  + \sum_{i=1}^{p^\star} \tilde K_2(x_j^\star,x_i^\star)^T \tilde b_{x_i} \right)^2  \\
&\quad + \sum_{j,\mu(\X_j^{\rm near}(r))>0} \mu(\X_j^{\rm near}(r)) \norm{\tilde H_j\frac{\tilde b_{x_j}}{\mu(\X_j^{\rm near}(r))}  + \sum_i \tilde b_{\omega_i}  \tilde K_1(x_j^\star,x_i^\star)  + \sum_i \tilde K_{12}(x_j^\star,x_i^\star) \tilde b_{x_i}}_2^2 
    \\
   &\quad  +\kappa^2 \sum_{j,\mu(\X_j^{\rm near}(r))>0} \int_{\X_j^{\rm near}(r)}\norm{\mathfrak{g}_{x_j^\star}^{-1/2} \dot \gamma_{x_j^\star \to x}(0)- \frac{\tilde b_{x_j}}{\mu(\X_j^{\rm near}(r))}}_2^2 \, \d\mu(x) \\
& \quad -C_{lower} \frac{1}{\kappa^2 r^4}(J_W(\mu)-J_W(\mu_\omega^\star))^2\\
&\quad -C_{lower} r^2 \sum_{j,\mu(\X_j^{\rm near}(r))>0}\left( \int_{\X_j^{\rm near}(r)}  \norm{\mathfrak{g}_{x_j^\star}^{1/2}\dot \gamma_{x_j^\star \to x}(0)-\frac{\tilde b_{x_j}}{\mu(\X_j^{\rm near}(r))}}_2^2 \, \d\mu(x) +\frac{\norm{\tilde b_{x_j}}_2^2}{\mu(\X_j^{\rm near}(r))}\right)\,,\\
&\gtrsim_{\varepsilon_0', \varepsilon_2} \kappa^2r^4\mu(\X^{\rm far}(r))+ \min_j\mu(\X_j^{\rm near}(r))
    \norm{\left(
\tilde \Upsilon + 
\begin{pmatrix}
0 & 0 \\
 0& \diag(\mathds{1}_{\mu(\X_j^{\rm near}(r))>0}\frac{1}{\mu(\X_j^{\rm near}(r))}\tilde H_j)_{j=1,\ldots,p^\star} 
\end{pmatrix}\right)\begin{pmatrix}
\tilde b_{\omega}\\
 \tilde b_{x}
 \end{pmatrix}}_2^2 \\
 & \quad +\kappa^2 \sum_{j, \mu(\X_j^{\rm near}(r))>0} \int_{\X_j^{\rm near}(r)}\norm{\mathfrak{g}_{x_j^\star}^{-1/2} \dot \gamma_{x_j^\star \to x}(0)- \frac{\tilde b_{x_j}}{\mu(\X_j^{\rm near}(r))}}_2^2 \, \d\mu(x) \\
& \quad -C_{lower} \frac{1}{\kappa^2 r^4}(J_W(\mu)-J_W(\mu_\omega^\star))^2\\
&\quad -C_{lower} r^2\sum_{j,\mu(\X_j^{\rm near}(r))>0} \left( \int_{\X_j^{\rm near}(r)}  \norm{\mathfrak{g}_{x_j^\star}^{1/2}\dot \gamma_{x_j^\star \to x}(0)-\frac{\tilde b_{x_j}}{\mu(\X_j^{\rm near}(r))}}_2^2 \, \d\mu(x) + \frac{\norm{\tilde b_{x_j}}_2^2}{\mu(\X_j^{\rm near}(r))}\right)\,,
\end{align*}
with $C_{lower}>0$ only depending on $d,B_{\rm TV},\norm{y}_\L,\norm{\mu_\omega^\star}_{\rm TV}, \varepsilon_2, \varepsilon_0'$, and where the matrix $\tilde \Upsilon$ has been introduced in \eqref{eq:tilde_Upsilon}.

Now, remark that $\tilde \Upsilon$ and $
\begin{pmatrix}
0 & 0 \\
 0& \diag(\mathds{1}_{\mu(\X_j^{\rm near}(r))>0}\frac{1}{\mu(\X_j^{\rm near}(r))}\tilde H_j)_{j=1,\ldots,p} 
\end{pmatrix}$ are symmetric positive matrices. Moreover, $\tilde \Upsilon \succeq C_{\tilde \Upsilon} I$. Remark also that for $j=1,\ldots,p^\star$, as $\omega_j^\star \geq \omega_\star$ (\cf Assumption \ref{assumption:non_degenerate_certif_solution}), $\min_j \mu(X_j^{\rm near}(r)) \geq \omega_\star - \norm{\tilde b_{\omega}}_{\infty}\geq \omega_\star-  \norm{\tilde b_{\omega}}_2$.
So
\begin{align*}
G_\mu &\gtrsim_{\varepsilon_0', \varepsilon_2, C_{\tilde \Upsilon}, \omega_\star} \kappa^2r^4\mu(\X^{\rm far}(r))+ 
    \tilde b^T \left(
\tilde \Upsilon + 
\begin{pmatrix}
0 & 0 \\
 0& \diag(\mathds{1}_{\mu(\X_j^{\rm near}(r))>0}\frac{1}{\mu(\X_j^{\rm near}(r))}\tilde H_j)_{j=1,\ldots,p^\star} 
\end{pmatrix}\right)\tilde b \\
 & \quad +\kappa^2 \sum_{j, \mu(\X_j^{\rm near}(r))>0} \int_{\X_j^{\rm near}(r)}\norm{\mathfrak{g}_{x_j^\star}^{-1/2} \dot \gamma_{x_j^\star \to x}(0)- \frac{\tilde b_{x_j}}{\mu(\X_j^{\rm near}(r))}}_2^2 \, \d\mu(x) \\
 & \quad -C_{lower}' \norm{\tilde b_{\omega}}_2 \tilde b^T \left(
\tilde \Upsilon + 
\begin{pmatrix}
0 & 0 \\
 0& \diag(\mathds{1}_{\mu(\X_j^{\rm near}(r))>0}\frac{1}{\mu(\X_j^{\rm near}(r))}\tilde H_j)_{j=1,\ldots,p^\star} 
\end{pmatrix}\right) \tilde b \\
& \quad -C_{lower}' \frac{1}{\kappa^2 r^4}(J_W(\mu)-J_W(\mu_\omega^\star))^2\\
&\quad -C_{lower}' r^2\sum_{j,\mu(\X_j^{\rm near}(r))>0} \left( \int_{\X_j^{\rm near}(r)}  \norm{\mathfrak{g}_{x_j^\star}^{1/2}\dot \gamma_{x_j^\star \to x}(0)-\frac{\tilde b_{x_j}}{\mu(\X_j^{\rm near}(r))}}_2^2 \, \d\mu(x) + \frac{\norm{\tilde b_{x_j}}_2^2}{\mu(\X_j^{\rm near}(r))}\right)\,.
\end{align*}
with $C_{lower}'>0$ only depending on $d,B_{\rm TV},\norm{y}_\L,\norm{\mu_\omega^\star}_{\rm TV}, \varepsilon_2, \varepsilon_0', C_{\tilde \Upsilon}, \omega_\star$. 

\paragraph{Choice of $r$}
According to Lemma \ref{lemma:control_certif_near_regions}, $\tilde H_j$ has eigenvalues greater than $\kappa\varepsilon_2$, leading to the inequality
\begin{equation} \label{eq:control_norm_b_bHb}
    \tilde b_x^T  \diag(\mathds{1}_{\mu(\X_j^{\rm near}(r))>0}\frac{1}{\mu(\X_j^{\rm near}(r))}\tilde H_j)_{j=1,\ldots,p^\star}  \tilde b_x \geq \kappa \varepsilon_2 \sum_{j, \mu(\X_j^{\rm near}(r))>0} \frac{\norm{\tilde b_{x_j}}_2^2}{\mu(\X_j^{\rm near}(r))}\,.
\end{equation}
Setting $C_r=\sqrt{\frac{1}{2C_{lower}'}\min\{\varepsilon_2\,,\,1\}}$ and $r = \min\{C_r \kappa,r_{\max}'\}$, we obtain
\begin{equation}\label{eq:lower_bound_G_mu}
\begin{aligned}
G_\mu &\gtrsim_{\varepsilon_0', \varepsilon_2, C_{\tilde \Upsilon}, \omega_\star} \kappa^2r^4\mu(\X^{\rm far}(r))+ 
    \tilde b^T \left(
\tilde \Upsilon + 
\begin{pmatrix}
0 & 0 \\
 0& \diag(\mathds{1}_{\mu(\X_j^{\rm near}(r))>0}\frac{1}{\mu(\X_j^{\rm near}(r))}\tilde H_j)_{j=1,\ldots,p^\star} 
\end{pmatrix}\right)\tilde b \\
 & \quad +\kappa^2 \sum_{j, \mu(\X_j^{\rm near}(r))>0} \int_{\X_j^{\rm near}(r)}\norm{\mathfrak{g}_{x_j^\star}^{-1/2} \dot \gamma_{x_j^\star \to x}(0)- \frac{\tilde b_{x_j}}{\mu(\X_j^{\rm near}(r))}}_2^2 \, \d\mu(x) \\
 & \quad -C_{lower}'' \norm{\tilde b_{\omega}}_2 \tilde b^T \left(
\tilde \Upsilon + 
\begin{pmatrix}
0 & 0 \\
 0& \diag(\mathds{1}_{\mu(\X_j^{\rm near}(r))>0}\frac{1}{\mu(\X_j^{\rm near}(r))}\tilde H_j)_{j=1,\ldots,p^\star} 
\end{pmatrix}\right) \tilde b \\
& \quad -C_{lower}'' \frac{1}{\kappa^2 r^4}(J_W(\mu)-J_W(\mu_\omega^\star))^2\,.
\end{aligned}
\end{equation}
with $C_{lower}''>0$ only depending on $d,B_{\rm TV},\norm{y}_\L,\norm{\mu_\omega^\star}_{\rm TV}, \varepsilon_2, \varepsilon_0', C_{\tilde \Upsilon}, \omega_\star$. 
Moreover, the constant $C_r$ depends only on $d,B_{\rm TV},\norm{y}_\L,\norm{\mu_\omega^\star}_{\rm TV}, \varepsilon_2, \varepsilon_0', C_{\tilde \Upsilon}, \omega_\star$.
\subsection{Upper bound on the optimality gap}\label{section:upper_bound_optimality_gap}
Let $\mu \in \M^+(\X)$ such that $\norm{\mu}_{\rm TV}\leq B_{\rm TV}$. We aim to upper bound $J_W(\mu)-J_W(\mu_\omega^\star)$.
Recall that (see \eqref{eq:dev_J_mu_plus_sigma_minus_J_mu})
 \begin{equation}\label{eq:decomposition_optimality_gap}
     J_W(\mu)-J_W(\mu_\omega^\star)= \underbrace{\int_\X J_{\mu_\omega^\star}' \, \d\mu}_{A} + \underbrace{\frac{1}{2}\norm{\Psi(\mu - \mu_\omega^\star)}_\L^2}_{B}\,.
 \end{equation} 
\underline{Control of $A$:}
Remark that $J_{\mu_\omega^\star}\leq \kappa + \norm{\mu_\omega^\star}_{\rm TV} + \norm{y}_\L$ (Lemma \ref{lemma:control_Jp}).
As $J_{\mu_\omega^\star}'(x_j^\star)=0$ and $\nabla J_{\mu_\omega^\star}'(x_j^\star)=0$ (see \eqref{eq:prop_Jp_grad_0}), using a Taylor expansion we have
\begin{align}
  A &= \sum_{j=1}^{p^\star} \int_{\X_j^{\rm near}(r)} J_{\mu_\omega^\star}'(x) \, \d\mu(x) + \int_{\X^{\rm far}(r)} J_{\mu_\omega^\star}'(x) \, \d\mu(x) \,,\\
  &= \sum_{j=1}^{p^\star} \int_{\X_j^{\rm near}(r)} \int_0^1 (1-y) \dot \gamma_{x_j^\star \to x}(y)^T \innerprod{H^\mathfrak{g} \Psi \delta_{\gamma_{x_j^\star \to x}(y)}}{\Psi \mu_\omega^\star-y}_\L \dot \gamma_{x_j^\star \to x}(y) \, \d y \: \d\mu(x) + \int_{\X^{\rm far}(r)} J_{\mu_\omega^\star}'(x) \, \d\mu(x) \,, \notag \\
  &\leq \frac{1}{2} \sqrt{B_{22}} (\norm{y}_\L + \norm{\mu_\omega^\star}_{\rm TV}) \sum_{j=1}^{p^\star} \int_{X_j^{\rm near}(r)}\mathfrak{d}_\mathfrak{g}(x_j^\star,x)^2 \, \d\mu(x) + (\kappa + \norm{\mu_\omega^\star}_{\rm TV} + \norm{y}_\L) \mu(\X^{\rm far}(r)) \,, \notag \\
  \begin{split} \label{eq:upper_bound_Jpstar}
  & \lesssim_{d, \varepsilon_2, \norm{y}_\L, \norm{\mu_\omega^\star}_{\rm TV}} \sum_{j,\mu(\X_j^{\rm near}(r))>0} \left( \frac{1}{\kappa} \frac{1}{\mu(\X_j^{\rm near}(r))} \tilde b_{x_j}^T \tilde H_j \tilde b_{x_j}+ \int_{\X_j^{\rm near}(r)} \norm{\mathfrak{g}_x^{1/2}\dot \gamma_{x_j^\star \to x}(0)-\frac{\tilde b_{x_j}}{\mu(\X_j^{\rm near}(r))}}_2^2 \, \d\mu(x)\right) \\
  &\quad + \mu(\X^{\rm far}(r)) \,,
  \end{split}
\end{align}
where we used Equations \eqref{eq:control_int_fisher_rao_norm_b} and \eqref{eq:control_norm_b_bHb} to get the last inequality.
\\\underline{Control of $B$:}
We also have
 \begin{align}
     B&=\frac{1}{2}\innerprod{\Psi \mu}{\Psi(\mu-\mu_\omega^\star)}_\L - \frac{1}{2} \innerprod{\Psi \mu_\omega^\star}{\Psi(\mu-\mu_\omega^\star)}_\L \,, \notag \\
     &= \frac{1}{2} \sum_j \left( \int_{\X} \tilde b_{\omega_j} K_{\rm norm}(x_j^\star,x')\,\d(\mu-\mu_\omega^\star)(x')+ \int_{\X} \nabla_1 K_{\rm norm}(x_j^\star,x')^T \tilde b_{x_j} \,\d(\mu-\mu_\omega^\star)(x') \right) \notag \\
     &\quad + \underbrace{\frac{1}{2}\sum_j \int_{\X_j^{\rm near}(r)} \int_\X \int_0^1 (1-y) \dot\gamma_{x_j^\star \to x}(y)^T H_1^{\mathfrak{g}}K_{\rm norm}(\gamma_{x_j^\star \to x}(y),x') \dot \gamma_{x_j^\star \to x}(y) \, \d y \, \d(\mu-\mu_\omega^\star)(x') \, \d\mu(x)}_{E_1} \notag \\
     &\quad + \underbrace{\frac{1}{2}\int_{\X^{\rm far}(r)} \int_{\X} K_{\rm norm}(x,x') \, \d(\mu-\mu_\omega^\star)(x') \, \d\mu(x)}_{E_2}  \,, \notag \\
     &=  \frac{1}{2}\sum_j \tilde b_{\omega_j} \sum_i \left(\tilde b_{\omega_i} \tilde K(x_j^\star,x_i^\star) + \nabla_2 \tilde K(x_j^\star,x_i^\star)^T \tilde b_{x_i} \right) \notag \\
     &\quad + \frac{1}{2}\sum_j  \sum_i \left(\tilde b_{\omega_i} \nabla_1\tilde K(x_j^\star,x_i^\star)^T \tilde b_{x_j} +  \tilde b_{x_i}^T \nabla_2\nabla_1\tilde K(x_j^\star,x_i^\star)^T \tilde b_{x_j} \right) \notag \\
     &\quad + \underbrace{\frac{1}{2}\sum_j \sum_i\tilde b_{\omega_j} \int_{\X_i^{\rm near}(r)} \int_0^1 (1-y)\dot\gamma_{x_i^\star \to x}(y)^T H_2^{\mathfrak{g}}K_{\rm norm}(x_j^\star,\gamma_{x_i^\star \to x'}(y)) \dot \gamma_{x_i^\star \to x}(y) \, \d y \, \d \mu(x')}_{E_3} \notag \\
     & \quad + \underbrace{\frac{1}{2} \sum_j \int_{\X^{\rm far}(r)} \tilde b_{\omega_j} K_{\rm norm}(x_j^\star,x') \, \d\mu(x')}_{E_4} \notag \\
      &\quad + \underbrace{\frac{1}{2}\sum_j \sum_i \int_{\X_i^{\rm near}(r)} \int_0^1 (1-y) [\tilde b_{x_j}]K_{\rm norm}^{(12)}(x_j^\star,\gamma_{x_i^\star \to x'}(y))[\dot\gamma_{x_i^\star \to x}(y), \dot \gamma_{x_i^\star \to x}(y)] \, \d y \, \d \mu(x')}_{E_5} \notag \\
    &\quad + \underbrace{\frac{1}{2}\sum_j \int_{\X^{\rm far}(r)} \nabla_1 K_{\rm norm}(x_j^\star,x')^T \tilde b_{x_j} \,\d\mu(x')}_{E_6} \notag \\
     & \quad + E_1+E_2\,, \notag \\
     &= \frac{1}{2} \tilde b^T \tilde \Upsilon \tilde b  + E_1+E_2 +E_3+E_4+E_5+E_6 \label{eq:norm_psi_mu_psi_mu_star_explicit}\,,
 \end{align}
 where $\Tilde \Upsilon$ is defined by \eqref{eq:tilde_Upsilon}.
 Using controls from Lemma \ref{lemma:controls_kernel} and \eqref{eq:identification_error_terms}, we get 
 \begin{align}
     \sum_{k=1}^6 |E_i| & \leq \frac{1}{4} \sqrt{B_{22}} \norm{\Psi(\mu-\mu_\omega^\star)}_\L \sum_j \int_{\X_j^{\rm near}}  \mathfrak{d}_\mathfrak{g}(x_j^\star,x)^2 \, \d\mu(x) + \frac{1}{2}\mu(\X^{\rm far}(r)) \norm{\Psi(\mu-\mu_\omega^\star)}_\L \notag \\
     & \quad +  \frac{B_{02}}{4}\sum_j |\tilde b_{\omega_j}| \sum_i \int_{\X_i^{\rm near}(r)} \mathfrak{d}_\mathfrak{g}(x_i^\star,x')^2 \, \d\mu(x') + \frac{1}{2} \mu(\X^{\rm far}(r))\sum_j |\tilde b_{\omega_j}| \notag \\
     &\quad + \frac{B_{12}}{4} \sum_j \norm{\tilde b_{x_j}}_2 \sum_i \int_{\X_i^{\rm near}(r)} \mathfrak{d}_\mathfrak{g}(x_i^\star,x')^2 \, \d\mu(x') + \frac{B_{10}}{2}\mu(\X^{\rm far}(r))\sum_j \norm{\tilde b_{x_j}}_2 \,, \notag \\
     & \lesssim_{d,\varepsilon_2, \varepsilon_0'} \frac{1}{\kappa r^2}(J_W(\mu)-J_W(\mu_\omega^\star))^{3/2} + \frac{1}{\kappa r^2} (J_W(\mu)-J_W(\mu_\omega^\star))\left(\sum_j |\tilde b_{\omega_j}| + \norm{ \tilde b_{x_j}}_2 \right) \,, \notag\\
 & \lesssim_{d,\varepsilon_2, \varepsilon_0'} \frac{1}{\kappa r^2}(J_W(\mu)-J_W(\mu_\omega^\star))^{3/2}  +  \frac{1}{\kappa r^2} p^\star(J_W(\mu)-J_W(\mu_\omega^\star))\norm{\tilde b}_2 \label{eq:upper_bound_error_terms_norm_psi_mu_psi_mu_star}\,.
 \end{align}
As $\tilde \Upsilon \succeq C_{\tilde \Upsilon} I$, we have $\norm{\tilde b}_2^2 \leq \frac{1}{C_{\tilde \Upsilon}} \tilde b^T \tilde \Upsilon \tilde b$. So
 \begin{align}
     B&\lesssim_{d,\varepsilon_2, \varepsilon_0', p^\star} \tilde b^T \tilde \Upsilon \tilde b + \frac{1}{\kappa r^2}(J_W(\mu)-J_W(\mu_\omega^\star))^{3/2}  +  \frac{1}{\kappa^2 r^4} (J_W(\mu)-J_W(\mu_\omega^\star))^2+\norm{\tilde b}_2^2\,, \notag \\
    &\lesssim_{d,\varepsilon_2, \varepsilon_0', p^\star, C_{\tilde \Upsilon}} \tilde b^T \tilde \Upsilon \tilde b + \frac{1}{\kappa r^2}(J_W(\mu)-J_W(\mu_\omega^\star))^{3/2}  +  \frac{1}{\kappa^2 r^4} (J_W(\mu)-J_W(\mu_\omega^\star))^2\,. \label{eq:upper_bound_norm_psi_mu_psi_mu_star}
 \end{align}

\noindent
\underline{Conclusion for the upper bound:}
Assembling \eqref{eq:decomposition_optimality_gap}, \eqref{eq:upper_bound_Jpstar} and \eqref{eq:upper_bound_norm_psi_mu_psi_mu_star}, we get
\begin{equation}\label{eq:upper_bound_optimality_gap}
    \begin{aligned}
J_W(\mu)-J_W(\mu_\omega^\star)&\lesssim_{\norm{y}_\L, \norm{\mu_\omega^\star}_{\rm TV},d,\varepsilon_2, \varepsilon_0', p^\star, C_{\tilde \Upsilon}} \sum_{j, \mu(\X_j^{\rm near}(r))>0} \int_{\X_j^{\rm near}(r)} \norm{\mathfrak{g}_x^{1/2}\dot \gamma_{x_j^\star \to x}(0)-\frac{\tilde b_{x_j}}{\mu(\X_j^{\rm near}(r))}}_2^2 \, \d\mu(x) \\
&\quad +  \mu(\X^{\rm far}(r)) + \frac{1}{\kappa}
\tilde b^T \left(
\tilde \Upsilon + 
\begin{pmatrix}
0 & 0 \\
 0& \diag(\mathds{1}_{\mu(\X_j^{\rm near}(r))>0}\frac{1}{\mu(\X_j^{\rm near}(r))}\tilde H_j)_{j=1,\ldots,p^\star} 
\end{pmatrix}\right) \tilde b  \\
 &\quad + \frac{1}{\kappa r^2}(J_W(\mu)-J_W(\mu_\omega^\star))^{3/2}  +  \frac{1}{\kappa^2 r^4} (J_W(\mu)-J_W(\mu_\omega^\star))^2\,.
    \end{aligned}
\end{equation}
\underline{Bound on $\norm{\tilde b}_2$:} We revisit the bound on $J_W(\mu)-J_W(\mu_\omega^\star)$. Remark that $A\geq 0$, because $\eta^\star \leq 1$ (implying that $J_{\mu_\omega^\star} \geq 0$). 
Recall also that $\tilde b^T
\tilde \Upsilon \tilde b \geq C_{\tilde \Upsilon} \norm{\tilde b}_2^2$. Hence, using \eqref{eq:decomposition_optimality_gap}, \eqref{eq:norm_psi_mu_psi_mu_star_explicit} and \eqref{eq:upper_bound_error_terms_norm_psi_mu_psi_mu_star}, we get 
\begin{align*}
    J_W(\mu)-J_W(\mu_\omega^\star) &\geq \frac{1}{2}C_{\tilde \Upsilon}\norm{\tilde b}_2^2 - \sum_{k=1}^6 |E_i| \,, \\
    & \geq \frac{1}{2}C_{\tilde \Upsilon}\norm{\tilde b}_2^2 - c_{d,\varepsilon_2,\varepsilon_0',p^\star} \left(\frac{1}{\kappa r^2}(J_W(\mu)-J_W(\mu_\omega^\star))^{3/2}  +  \frac{1}{\kappa r^2} (J_W(\mu)-J_W(\mu_\omega^\star))\norm{\tilde b}_2\right) \,,
\end{align*}
with $c_a$ denoting a positive constant depending only on $a$.
It comes that 
\begin{align*}
    \frac{1}{2}C_{\tilde \Upsilon}\left(\norm{\tilde b}_2-\frac{c_{d,\varepsilon_2,\varepsilon_0',p^\star}}{C_{\tilde \Upsilon} \kappa r^2}(J_W(\mu)-J_W(\mu_\omega^\star)) \right)^2 &\leq J_W(\mu)-J_W(\mu_\omega^\star) + c_{d,\varepsilon_2,\varepsilon_0',p^\star} \frac{1}{\kappa r^2}(J_W(\mu)-J_W(\mu_\omega^\star))^{3/2} \\
    &\quad + \frac{1}{2}\frac{c_{d,\varepsilon_2,\varepsilon_0',p^\star}^2}{C_{\tilde \Upsilon} \kappa^2 r^4}(J_W(\mu)-J_W(\mu_\omega^\star))^2
\end{align*}

If $\mu$ is such that $J_W(\mu)-J_W(\mu_\omega^\star)\leq \kappa^2 r^4$, we get
\begin{equation}\label{eq:control_norm_tilde_b}
    \norm{\tilde b}_2^2 \lesssim_{d,\varepsilon_2,\varepsilon_0',p^\star,C_{\tilde \Upsilon}} J_W(\mu)-J_W(\mu_\omega^\star) \,.
\end{equation}

\subsection{Assembling the bounds}\label{section:assembling_lower_upper_bounds}
\underline{Choice of $J_0'$:} We recall Equations \eqref{eq:lower_bound_G_mu}, \eqref{eq:upper_bound_optimality_gap} and \eqref{eq:control_norm_tilde_b}. We deduce that there exists $0<J_0'' \leq 1$ small enough, depending on $C_{lower}'',d,\varepsilon_2,\varepsilon_0',p^\star,C_{\tilde \Upsilon}, \norm{y}_\L, \norm{\mu_\omega^\star}_{\rm TV}$, such that if $J_W(\mu)-J_W(\mu_\omega^\star) \leq J_0''\kappa^2 r^4$, then it holds that
\begin{equation*}
\begin{aligned}
G_\mu &\gtrsim_{\varepsilon_0', \varepsilon_2, C_{\tilde \Upsilon}, \omega_\star} \kappa^2r^4\mu(\X^{\rm far}(r))+ 
    \tilde b^T \left(
\tilde \Upsilon + 
\begin{pmatrix}
0 & 0 \\
 0& \diag(\mathds{1}_{\mu(\X_j^{\rm near}(r))>0}\frac{1}{\mu(\X_j^{\rm near}(r))}\tilde H_j)_{j=1,\ldots,p^\star} 
\end{pmatrix}\right)\tilde b \\
 & \quad +\kappa^2 \sum_{j, \mu(\X_j^{\rm near}(r))>0} \int_{\X_j^{\rm near}(r)}\norm{\mathfrak{g}_{x_j^\star}^{-1/2} \dot \gamma_{x_j^\star \to x}(0)- \frac{\tilde b_{x_j}}{\mu(\X_j^{\rm near}(r))}}_2^2 \, \d\mu(x) \\
& \quad -C_{lower}''' \frac{1}{\kappa^2 r^4}(J_W(\mu)-J_W(\mu_\omega^\star))^2
\end{aligned}
\end{equation*}
(where $C_{lower}'''$ depends on $d,B_{\rm TV},\norm{y}_\L,\norm{\mu_\omega^\star}_{\rm TV}, \varepsilon_2, \varepsilon_0',  C_{\tilde \Upsilon}, \omega_\star$) together with
\begin{equation*}
    \begin{aligned}
J_W(\mu)-J_W(\mu_\omega^\star)&\lesssim_{\norm{y}_\L, \norm{\mu_\omega^\star}_{\rm TV},d,\varepsilon_2, \varepsilon_0', p^\star, C_{\tilde \Upsilon}} \sum_{j, \mu(\X_j^{\rm near}(r))>0} \int_{\X_j^{\rm near}(r)} \norm{\mathfrak{g}_x^{1/2}\dot \gamma_{x_j^\star \to x}(0)-\frac{\tilde b_{x_j}}{\mu(\X_j^{\rm near}(r))}}_2^2 \, \d\mu(x) \\
&\quad +  \mu(\X^{\rm far}(r)) + \frac{1}{\kappa}
\tilde b^T \left(
\tilde \Upsilon + 
\begin{pmatrix}
0 & 0 \\
 0& \diag(\mathds{1}_{\mu(\X_j^{\rm near}(r))>0}\frac{1}{\mu(\X_j^{\rm near}(r))}\tilde H_j)_{j=1,\ldots,p^\star} 
\end{pmatrix}\right) \tilde b  \,.
    \end{aligned}
\end{equation*}
This implies the existence of $R_0''>0$ depending on $\norm{y}_\L, \norm{\mu_\omega^\star}_{\rm TV},d,\varepsilon_2, \varepsilon_0', p^\star, C_{\tilde \Upsilon}, \omega_\star$ such that
\[G_\mu \geq R_0'' \kappa^2 r^4 (J_W(\mu)-J_W(\mu_\omega^\star)) -c_{\varepsilon_2, \varepsilon_0', C_{\tilde \Upsilon}, \omega_\star} C_{lower}''' \frac{1}{\kappa^2 r^4}(J_W(\mu)-J_W(\mu_\omega^\star))^2 \,.  \]
Hence there exists $J_0'>0$ depending on $J_0''$, $C_{lower}'''$, $\varepsilon_0'$, $\varepsilon_2$, $C_{\tilde \Upsilon}$, $\omega_\star$, and $R_0'>0$ depending on $\norm{y}_\L$, $\norm{\mu_\omega^\star}_{\rm TV}$, $d$, $\varepsilon_2$, $\varepsilon_0'$, $p^\star$, $C_{\tilde \Upsilon}$, $\omega_\star$, such that
\[J_W(\mu)-J_W(\mu_\omega^\star) \leq J_0' \kappa^4 r^8 \quad \Longrightarrow \quad G_\mu \geq R_0' \kappa^2 r^4 (J_W(\mu) -J_W(\mu_\omega^\star) )\,.\]

\noindent
\underline{Tracking dependence on parameters:}
We recall the definition of the constant factors, and the choice of $r$ made in Appendix \ref{section:lower_bound_objective_decrease}: $r=\min\{C_r \kappa, r_{\max}'\}$ with $C_r>0$ depending on $d$, $\norm{y}_\L$, $\norm{\mu_\omega^\star}_{\rm TV}$, $\varepsilon_2$, $\varepsilon_0'$, $C_{\tilde \Upsilon}$, $\omega_\star$, $\X$, $\norm{\mu_\omega^1}_{\rm TV}$, and $r_{\max}'$ depending on $d$, $\X$, $\tilde \X$, $\tau$, $\mu_\omega^\star$, $\eta^\star$. The constant $J_0'>0$ depends on $p^\star$, $d$, $\X$, $\norm{\mu_\omega^1}_{\rm TV}$, $\norm{y}_\L$, $\norm{\mu_\omega^\star}_{\rm TV}$, $\varepsilon_2$, $\varepsilon_0'$, $C_{\tilde \Upsilon}$, $\omega_\star$.
We denote $R_0\coloneq R_0' \min\{r_{\max}',C_r\}^4$. This constant depends on $\X$, $\tilde \X$, $\tau$, $\mu_\omega^\star$, $\eta^\star$, $\norm{y}_\L$, $\norm{\mu_\omega^1}_{\rm TV}$.
(Remark that $\eta^\star$, $\mu_\omega^\star$ contain the dependence on $C_{\tilde \Upsilon}$, $\varepsilon_0'$, $\varepsilon_2$, $\omega_\star$, $p^\star$, $\norm{\mu_\omega^\star}_{\rm TV}$, and that $\X$ contains the dependence on $d$.)
We also define $J_0 \coloneq J_0' \min\{r_{\max}',C_r\}^8$. It also depends on $\X$, $\tilde \X$, $\tau$, $\mu_\omega^\star$, $\eta^\star$, $\norm{y}_\L$, $\norm{\mu_\omega^1}_{\rm TV}$.

If $\mu$ satisfies $J_W(\mu) - J_W(\mu_\omega^\star) \leq J_0 \kappa^{12}$, then 
\begin{equation} \label{eq:bound_G_mu_with_optimality_gap}
 G_\mu \geq R_0 \kappa^6 (J_W(\mu) -J_W(\mu_\omega^\star) )\,.   
\end{equation}

\subsection{Exponential convergence}\label{section:proof_exponential_cv_from_upper_lower_bounds}
If $\mu_\omega^1$ verifies 
\begin{equation}\label{eq:condition_initialisation}
J_W(\mu_\omega^1)-J_W(\mu_\omega^\star)\leq J_0 \kappa^{12}\,,
\end{equation}
we retrieve the result of \citep[Corollary 3.5]{cpgd_chizat}.
First note that $\norm{\mu_\omega^k}_{\rm TV} \leq B_{\rm TV}$ for all $k\geq 1$ (Lemma \ref{lemma:control_tv_iterates}).

Then, as $k \in \mathbb{N}^* \mapsto J_W(\mu_\omega^k)$ is decreasing, it also holds that $J_W(\mu_\omega^k)-J_W(\mu_\omega^\star) \leq J_0 \kappa^{12}$ and we get \[J_W(\mu_\omega^{k+1}) - J_W(\mu_\omega^{k}) \leq -\frac{1}{2} \min\{\alpha,\beta\} G_{\mu_\omega^k} \leq -\frac{1}{2} \min\{\alpha,\beta\} R_0 \kappa^6 (J_W(\mu_\omega^k)- J_W(\mu_\omega^\star)) \,,\]
where we used \eqref{eq:bound_G_mu_with_optimality_gap}. Note that the contraction factor is automatically admissible: writing $c\coloneq \frac{1}{2}\min\{\alpha,\beta\}R_0\kappa^6$ and $g_k \coloneq J_W(\mu_\omega^k)-J_W(\mu_\omega^\star)$, the above reads $g_{k+1}\leq (1-c)g_k$, while $g_{k+1}\geq 0$ since $\mu_\omega^\star$ minimizes $J_W$ over $\M(\X)^+$. Hence $c\leq 1$ as soon as $g_k>0$, so that $1-c\in[0,1)$ and the geometric bound below is never vacuous.
So for $k\geq 1$, 
\begin{equation}\label{eq:rate_convergence}
J_W(\mu_\omega^k)- J_W(\mu_\omega^\star) \leq \left(1-\frac{1}{2} \min\{\alpha,\beta\} R_0 \kappa^6 \right)^{k-1} (J_W(\mu_\omega^1)- J_W(\mu_\omega^\star))\,.
\end{equation}

\begin{remark}[Refining the dependence on $\kappa$]\label{remark:dependence_kappa_exp_convergence}
To obtain the lower bound for $G_\mu$, we used several times the inequality $(a+b)^2 \geq \frac{1}{2}a^2 - b^2$. A parameterized version of this inequality holds: for $\delta\in (0,1)$, $(a+b)^2 \geq (1 - \delta)a^2 - \frac{1-\delta}{\delta} b^2$. By optimizing $\delta$ where we use this inequality, we can improve both the condition on the initialization and the convergence rate. Specifically, this allows one to reduce the exponent associated with $\kappa$ in Equations \eqref{eq:condition_initialisation} and \eqref{eq:rate_convergence}.
\\We do not present this approach here in order to keep the exposition simple, since it would require choosing $\delta$ based on results given later in the proof.
\end{remark}

\section{Proof of Theorem \ref{th:check_assumptions_solution}}\label{section:proof_checking_non_degeneracy_solution}

This proof relies on calculations established in a previous work \citep{article_giard_decastro_marteau}. More precisely, we import results from the following parts: Lemma K.1, Theorem 5.1, Theorem 6.1, Proposition K.1 and its proof, Lemma K.4, Proposition I.1 and its proof.

\paragraph{Non degeneracy of the statistical target}
We suppose that $y=\Psi \mu_\omega^0 + b$, where $b \in \L$ and $\mu_\omega^0 \coloneq \sum_{j=1}^s \omega_j^0 \delta_{x_j^0}$, with $\omega_j^0>0$ and $x_j^0 \in \mathring{\tilde \X}$, and $s\geq 2$.

We then use \citep[Lemma K.1]{article_giard_decastro_marteau}. If the minimal distance between the particles of $\mu_\omega^0$ is sufficiently large, \ie if $\min_{i \neq j}\mathpzc{d}(x_j^0,x_i^0) \geq \Delta_\tau$ where $\Delta_\tau$ depends on $d,u_{\min},u_{\max},\tau,s$ (see \citep[Theorem 5.1]{article_giard_decastro_marteau} for the precise value of $\Delta_\tau$), then $\mu_\omega^0$ satisfies the NDSC. 

We can apply \citep[Theorem 6.1]{article_giard_decastro_marteau}: there exists $\kappa_0>0$ (depending on $\X,\tau,\mu_\omega^0$), $\gamma_0>0$ (depending on $d$) such that for all $0<\kappa \leq \kappa_0$ and if $\norm{b}_\L \leq \gamma_0 \kappa$, then the solution $\mu_\omega^\star$ is unique and $s$-sparse: $\mu_\omega^\star=\sum_{j=1}^s \omega_j^\star \delta_{x_j^\star}$ with $\omega_j^\star>0$ (\ie $p^\star=s$). 

\paragraph{Non degeneracy of the solution}
Under these assumptions, the proof of \citep[Proposition K.1]{article_giard_decastro_marteau} shows that $\eta^\star$ verifies items 1,2,3,4 of Assumption \ref{assumption:non_degenerate_certif_solution} with $\varepsilon_2= 0.06158$.

\paragraph{Control of the solution}
We have $\omega_\star=\min_j \omega_j^\star>0$ (\cf \citep[Lemma K.4]{article_giard_decastro_marteau}). This minimum weight depends a priori on the problem \eqref{eq:pb_BLASSO_gaussian_kernel_norm}.
We can obtain finer results with \citep[Theorem 6.1]{article_giard_decastro_marteau}
It comes that for $1 \leq j \leq s$, \[ |\omega_j^0 - \omega_j^\star| \lesssim_{\X,\mu^0,\tau} \kappa \quad \text{and} \quad \mathpzc{d}(x_j^0,x_j^\star)\lesssim_{\X,\mu^0,\tau} \kappa \,. \] 
If $\kappa$ is chosen sufficiently small, and if the noise is small \wrt $\kappa$, we have $\omega^\star=\min_j \omega_j^\star$ close to $\min_j\omega_j^0$,and $x_j^\star$ close to $x_j^0$ and in particular, $x_j^\star \in \mathring{\tilde \X}$. This provides a condition ensuring that the statement regarding the location of the solution in Assumption \ref{assumption:non_degenerate_certif_solution} is satisfied. More explicitly, there exists $C_\kappa'$ depending on $\X,\tilde \X, \tau,\mu_\omega^0$ such that if $\kappa \leq \min\{\kappa_0,C_\kappa'\}$ and $\norm{b}_\L \leq \gamma_0 \kappa$, then $\min_j \omega_j^\star \geq \frac{1}{2}\min_j \omega_j^0$, and for all $j\in \{1,\ldots,s\}$, $x_j^\star \in \mathring{\tilde \X}$.

\paragraph{Positive definiteness of $\tilde \Upsilon$}
We know that 
\[\tilde \Upsilon_0 \coloneq \begin{pmatrix}
(K_{\rm norm}(x_j^0,x_i^0))_{j,i=1,\ldots,s} & (\nabla_2 K_{\rm norm}(x_j^0,x_i^0)^T \mathfrak{g}_{x_i^0}^{-1/2})_{j,i=1,\ldots,s} \\
(\mathfrak{g}_{x_j^0}^{-1/2}\nabla_1 K_{\rm norm}(x_j^0,x_i^0))_{j,i=1,\ldots,s} & (\mathfrak{g}_{x_j^0}^{-1/2}\nabla_1 \nabla_2 K_{\rm norm}(x_j^0,x_i^0) \mathfrak{g}_{x_i^0}^{-1/2})_{j,i=1,\ldots,s}
\end{pmatrix}\] is positive definite, from the separation assumptions made to obtain the NDSC. We refer to the proof of \citep[Proposition I.1]{article_giard_decastro_marteau}: $\tilde \Upsilon_0\succeq \frac{1}{2} I$.

As the minimal eigenvalue of $\tilde \Upsilon_0$ is continuous in the parameters $\{x_j^0\}_{j=1}^s$, and as $\tilde \Upsilon$ is the matrix built from the same expression at the points $\{x_j^\star\}_{j=1}^s$, it suffices for each $x_j^\star$ to be sufficiently close to $x_j^0$ for the minimal eigenvalue of $\tilde \Upsilon$ to remain above $\frac{1}{4}$.
In particular, there exists $0<C_\kappa\leq \min\{C_\kappa',\kappa_0\}$ depending on $\X,\tilde \X,\tau,\mu_\omega^0$ such that if $0<\kappa \leq C_\kappa$ and $\norm{b}_\L \leq \gamma_0 \kappa$, then $\tilde \Upsilon \succeq \frac{1}{4} I$.

\section{Proof of the bound on the optimality gap}\label{section:proof_optimality_gap}
We prove the result stated by \eqref{eq:bound_optimality_gap} in the Discussion section (Section \ref{section:discussion}).

Let $\mu_\omega^\star$ be the minimizer of $J_W$ on $\M(\X)^+$. Let $\mu \in \M(\X)^+$. We want to give an upper bound on $J_W(\mu)-J_W(\mu_\omega^\star)$ that does not depend on $\mu_\omega^\star$, but that is still more informative than the value of $J_W$ at point $\mu$.

Let $\alpha>0$ such that $\inf_{x \in \X} \alpha(J_\mu'(x) - \kappa) \geq - \kappa$. Let $h=\alpha(\Psi \mu-y)$. As $\mu_\omega^\star$ is nonnegative, it holds that 
\begin{equation}\label{eq:property_h_optim_gap}
- \innerprod{h}{\Psi \mu_\omega^\star} - \kappa \norm{\mu_\omega^\star}_{\rm TV} \leq 0 \,.
\end{equation}
Furthermore, as $\mu_\omega^\star$ is a minimizer, the KKT conditions hold and in particular $\int_{\X} J_{\mu_\omega^\star}' \, \d\mu_\omega^\star=0$, which can be rewritten as \[\innerprod{\Psi \mu_\omega^\star}{\Psi\mu_\omega^\star - y}_\L + \kappa \norm{\mu_\omega^\star}_{\rm TV}=0 \,.\]
Using \eqref{eq:property_h_optim_gap}, it follows that $\innerprod{\Psi \mu_\omega^\star}{\Psi\mu_\omega^\star - y}_\L \leq \innerprod{h}{\Psi \mu_\omega^\star}$, implying that $\norm{\Psi \mu_\omega^\star}_\L^2 \leq \norm{h+y}_\L^2$.
Hence
\begin{align*}
    J_W(\mu)-J_W(\mu_\omega^\star) &= J_W(\mu) + \frac{1}{2}\norm{\Psi \mu_\omega^\star}_\L^2 - \frac{1}{2}\norm{y}_\L^2 - \innerprod{\Psi \mu_\omega^\star}{\Psi\mu_\omega^\star - y}_\L - \kappa \norm{\mu_\omega^\star}_{\rm TV} \,, \\
    &= J_W(\mu) + \frac{1}{2}\norm{\Psi \mu_\omega^\star}_\L^2  - \frac{1}{2}\norm{y}_\L^2\,, \\
    &\leq J_W(\mu) + \frac{1}{2}\norm{h+y}_\L^2  - \frac{1}{2}\norm{y}_\L^2 \,, \\
    &= J_W(\mu)+\frac{1}{2}\norm{h}_\L^2 + \innerprod{h}{y}_\L \,.
\end{align*}

Since \[J_\mu'(x)=\innerprod{\Psi \mu-y}{\Psi \delta_x}_\L + \kappa \geq - \norm{\Psi \mu-y}_\L + \kappa \quad \forall \, x \in \X \,,\] the condition on $\alpha$ is always satisfied by choosing $\alpha= \frac{\kappa}{\max\{ \norm{\Psi \mu - y}_\L \,, \, \kappa \}}$.

\end{document}